\pdfoutput=1
\documentclass{article}
\PassOptionsToPackage{numbers,compress}{natbib}
\usepackage[main,final]{neurips_2026}

\usepackage[utf8]{inputenc}
\usepackage[T1]{fontenc}
\usepackage{amsmath, amssymb, amsthm}
\usepackage{graphicx}
\usepackage{subcaption}
\usepackage{algorithm}
\usepackage{algorithmic}
\usepackage{booktabs}
\usepackage{multirow}
\usepackage{xcolor}
\usepackage{bm}
\usepackage{enumitem}
\usepackage{microtype}
\usepackage{hyperref}
\usepackage[nameinlink,noabbrev]{cleveref}
\usepackage{nicefrac}
\usepackage{array}

\graphicspath{{figures/}}

\definecolor{darkblue}{rgb}{0.0, 0.0, 0.5}
\hypersetup{
    pdftitle={Universal Adaptive Proximal Gradient Methods via Gradient Mapping Accumulation},
    pdfauthor={Zimeng Wang and Alp Yurtsever},
    colorlinks=true,
    linkcolor=darkblue,
    citecolor=darkblue,
    urlcolor=darkblue
}

\numberwithin{equation}{section}

\newtheorem{theorem}{Theorem}
\newtheorem{lemma}{Lemma}

\newtheorem{corollary}{Corollary}
\theoremstyle{definition}

\newtheorem{remark}{Remark}
\newtheorem{assumption}{Assumption}

\newcommand{\rbb}{\mathbb{R}}
\newcommand{\ebb}{\mathbb{E}}
\newcommand{\dcal}{\mathcal{D}}
\newcommand{\gcal}{\mathcal{G}}

\DeclareMathOperator*{\argmin}{arg\,min}

\DeclareMathOperator{\dom}{dom}
\def\p{\operatorname{prox}}

\title{Universal Adaptive Proximal Gradient Methods \\ via Gradient Mapping Accumulation}

\author{%
  Zimeng Wang \quad\textnormal{\&}\quad Alp Yurtsever \\[0.25em]
  Department of Mathematics and Mathematical Statistics \\
  Ume{\aa} University, Sweden \\
  \texttt{\{zimeng.wang, alp.yurtsever\}@umu.se}
}

\begin{document}

\maketitle

\begin{abstract}
We propose an adaptive proximal gradient method for minimizing the sum of two functions, where one is a simple convex function, and the other belongs to one of the three classes: nonconvex smooth, convex nonsmooth, or convex smooth. The key feature of the method is an adaptive step size that accumulates historical gradient mapping norms in the denominator. Without any modification or knowledge of problem parameters, the method converges across all three problem classes under mild bounded-iterates and bounded-variance assumptions, with rates matching those of the proximal gradient method up to logarithmic factors, in both deterministic and stochastic settings. For the convex setting, we further propose an accelerated variant. It retains a similar near-optimal convergence rate for the nonsmooth case and achieves an improved rate of order $\widetilde{O}\big(1/t^2 + \sigma/\sqrt{t}\big)$ for the smooth case, which is optimal up to logarithmic factors. Notably, we develop new techniques for controlling the effect of stochastic noise, which are applicable across all three problem classes in the stochastic setting and enable a simplified analysis.
\end{abstract}

\section{Introduction}
\label{sec:intro}
We study the following nonsmooth composite optimization problem:
\begin{equation}\label{main-prob}
\min_{x\in\rbb^d} F(x):=f(x)+h(x),
\end{equation}
where $f:\rbb^d\rightarrow \rbb$ is either nonconvex smooth or convex (possibly nonsmooth), and $h:\rbb^d\rightarrow \rbb\cup \{+\infty\}$ is convex, proper, and lower semicontinuous, but possibly nonsmooth. Problems of this form are ubiquitous in machine learning, signal processing, and statistics. Representative examples include LASSO \citep{tibshirani1996regression}, support vector machines \citep{cortes1995support}, regularized logistic regression \citep{berkson1944application}, and trace-norm matrix completion \citep{candes2012exact}. In the stochastic setting, we are particularly interested in the case $f(x)=\ebb_{\xi\in\dcal}[\ell(x;\xi)]$, where $\ell(\cdot;\xi):\rbb^d\rightarrow \rbb$ is the loss function associated with the sample $\xi$ drawn from the underlying data space $\dcal$.

A widely used algorithm for solving~\eqref{main-prob} is the proximal gradient (PG) method \citep{lions1979splitting,tseng2000modified}, which iteratively updates (see \eqref{def:prox-h} for the definition of $\p_h(\cdot)$):
\begin{equation}\label{Alg:prox-grad}
\tag{PG}
x_{k+1}=\p_{\eta_k h}(x_k - \eta_k \nabla f(x_k))=\argmin_{x\in\rbb^d}\Big\{\langle \nabla f(x_k), x\rangle+\frac{1}{2 \eta_k}\|x-x_k\|^2+h(x)\Big\},
\end{equation}
with a step size $\eta_k>0$. When $f$ is $L$-smooth, PG with $\eta_k=1/L$ achieves $O(1/t)$ convergence for both nonconvex and convex $f$ \citep{ghadimi2016mini}, with accelerated variants \citep{beck2009fast,nesterov1983method,nesterov2013gradient} improving the convex rate to the optimal $O(1/t^2)$. For convex Lipschitz $f$, PG achieves the optimal $O(1/\sqrt{t})$ rate \citep{nesterov2003introductory} with decreasing step sizes $\eta_k=O(1/\sqrt{k})$. 
A key limitation of PG is that step sizes must be tuned to the problem class and parameters (e.g., $L$) to achieve the desired rates, which is often impractical.

A principled remedy is to use adaptive step sizes that adjust automatically without prior knowledge of the problem class or parameters. AdaGrad \citep{duchi2011adaptive} pioneers this by accumulating $\|\nabla f(x_k)\|^2$ in the denominator of step sizes, and has inspired a rich family of adaptive methods \citep{kingma2015adam,reddi2018convergence,tieleman2012rmsprop,zeiler2012adadelta}. 
In the non-composite setting ($h=0$), AdaGrad-type methods have been rigorously analyzed for both convex~\citep{ene2020adaptive,kavis2019unixgrad,levy2017online,levy2018online} and nonconvex objectives \citep{attia2023sgd,kavis2022high,wang2023convergence,ward2020adagrad}, automatically matching the rates of gradient methods without prior knowledge of problem parameters.

While adaptive step sizes are well-studied for non-composite problems, their extension to~\eqref{main-prob} remains largely unexplored. A core difficulty is that $\nabla f(x^*)$ at a stationary point $x^*$ is generally nonzero, so accumulating $\|\nabla f(x_k)\|^2$ in an AdaGrad manner drives step sizes to zero.
While this gives the optimal $O(1/\sqrt{t})$ rate for nonsmooth $f$~\citep{duchi2011adaptive}, for smooth $f$ it conflicts with the standard constant step size of order $1/L$.
This motivates the following question:

\emph{Can we design a universal\footnote{Here, \emph{universal} means that one method covers the three problem classes. The notion of universality in~\citet{nesterov2015universal} means adaptation to an unknown H\"older smoothness level, which we address in Appendix~\ref{app:holder-stoch} for convex $f$.} adaptive proximal gradient method for solving \eqref{main-prob} that automatically achieves convergence rates matching those of PG for $f$ being nonconvex smooth, convex smooth, or convex nonsmooth, without any prior knowledge of the problem-related parameters?}

\subsection{Our Contributions}
We answer affirmatively by proposing Algorithm~\ref{Alg:adaprox}, whose key idea is to replace squared gradient norms in the AdaGrad denominator with squared \emph{gradient mapping} norms, which vanish at stationary points of~\eqref{main-prob}.
We show that Algorithm \ref{Alg:adaprox} is \emph{universal}: it applies to three problem classes without modification or prior knowledge of parameters (e.g., smoothness constant, diameter $D$, or noise level $\sigma$); see also Table~\ref{tab:results}. We summarize the main contributions as follows:
\begin{itemize}[leftmargin=2em]
  \item For nonconvex smooth $f$, Algorithm~\ref{Alg:adaprox} achieves an $O(1/t)$ stationarity rate deterministically and $\widetilde{O}\big(\frac{1}{t}+\frac{\sigma^2}{t} \sum_{k=1}^t \frac{1}{b_k}\big)$ with  high probability, matching the best known rates of PG \citep{ghadimi2013stochastic,ghadimi2016mini}.
  To the best of our knowledge, this is the first convergence analysis of an AdaGrad-type method for nonconvex composite objectives.

  \item The same algorithm adapts to convex settings without any modification. Specifically, the averaged iterate of Algorithm~\ref{Alg:adaprox} with a unit batch size converges to an optimal solution at a rate of $\widetilde{O}\big((1+\sigma)/\sqrt{t}\big)$ for Lipschitz $f$, and improves to $\widetilde{O}(1/t+\sigma/\sqrt{t})$ for smooth $f$. The corresponding deterministic rates follow directly by setting $\sigma=0$.

  \item For convex $f$, we develop an accelerated variant Algorithm~\ref{Alg:acc-adaprox} that achieves better convergence rates when the diameter $D$ is known. In particular, Algorithm~\ref{Alg:acc-adaprox} with a unit batch size achieves an $\widetilde{O}((1+\sigma)/\sqrt{t})$ rate for Lipschitz $f$ and $\widetilde{O}(1/t^2+\sigma/\sqrt{t})$ for smooth $f$. Both rates are optimal up to logarithmic factors \citep{lan2012optimal}.
  \item Furthermore, for convex $(L_\nu,\nu)$-H\"older-smooth objectives with $\nu\in[0,1]$, the two algorithms also interpolate between the Lipschitz-continuous and smooth rates. Specifically, their expected objective gaps are $\widetilde{O}(t^{-(1+\nu)/2}+\sigma/\sqrt{t})$ for Algorithm~\ref{Alg:adaprox} and $\widetilde{O}(t^{-(1+3\nu)/2}+\sigma/\sqrt{t})$ for Algorithm~\ref{Alg:acc-adaprox}. See Theorems~\ref{thm:cvx-holder-stoch} and~\ref{thm:acc-holder-stoch} in Appendix~\ref{app:holder-stoch}.
\end{itemize}

A key technical challenge in the stochastic setting is that the adaptive step sizes become random variables, making it difficult to bound terms involving their products with other random variables.
We overcome this difficulty by estimating cumulative inner products between the stochastic noise and iterate differences, which enables us to establish convergence with simplified analysis across different problem classes.
Additionally, our analysis removes the restrictive assumption that the stochastic gradients are uniformly bounded almost surely, which is frequently imposed in the related literature for nonsmooth analysis (see, e.g., \citep{kavis2019unixgrad,kavis2022high,levy2017online,levy2018online,ward2020adagrad}).

\begin{table}[ht]
\centering
\small
\caption{Summary of convergence results for different classes of $f$. ``Full'' means full batch (deterministic). $\widetilde{O}(\cdot)$ hides constants and log factors.}
\label{tab:results}
\renewcommand{\arraystretch}{1.2}
\setlength{\tabcolsep}{2pt}
\begin{tabular}{@{}l ccc p{1.2cm} p{3.0cm} p{2.8cm} l@{}}
\toprule
\multirow{2}{*}{\textbf{Algorithm}} & \multicolumn{3}{c}{\textbf{Problem Class $f$}} & \multirow{2}{*}{\textbf{Batch}} & \multirow{2}{*}{\textbf{Convergence}} & \multirow{2}{*}{\textbf{Assumptions}} & \multirow{2}{*}{\textbf{Ref.}} \\
\cmidrule(lr){2-4}
 & \textbf{Convex} & \textbf{Smooth} & \textbf{Lipschitz} & & & & \\
\midrule
\multirow{6}{*}{Alg.~\ref{Alg:adaprox}}
  & $\times$ & $\checkmark$ & $\times$ & Full & $O(1/t)$ & \leavevmode None & Thm.~\ref{thm:nc-deter} \\
  & $\times$ & $\checkmark$ & $\times$ & $b_k$ & $\widetilde{O}\!\big(\frac{1}{t}+\frac{\sigma^2}{t}\sum_k\frac{1}{b_k}\big)$ & Asm.~\ref{ass:seq-bd}(a),\ref{ass:bd-stoch} & Thm.~\ref{thm:nc-stoc} \\
  & $\checkmark$ & $\times$ & $\checkmark$ & Full & $\widetilde{O}(1/\sqrt{t})$ & Asm.~\ref{ass:seq-bd}(b) & Cor.~\ref{cor:cvx-ns-deter} \\
  & $\checkmark$ & $\times$ & $\checkmark$ & $1$ & $\widetilde{O}((1+\sigma)/\sqrt{t})$ & Asm.~\ref{ass:seq-bd}(b),\ref{ass:bd-stoch} & Thm.~\ref{thm:stoch-convex-nonsmooth} \\
  & $\checkmark$ & $\checkmark$ & $\times$ & Full & $O(1/t)$ & Asm.~\ref{ass:seq-bd}(b) & Cor.~\ref{coro:stoch-convex-smooth} \\
  & $\checkmark$ & $\checkmark$ & $\times$ & $1$ & $\widetilde{O}(1/t+\sigma/\sqrt{t})$ & Asm.~\ref{ass:seq-bd}(b),\ref{ass:bd-stoch} & Thm.~\ref{thm:stoch-convex-smooth} \\
\midrule
\multirow{4}{*}{Alg.~\ref{Alg:acc-adaprox}}
  & $\checkmark$ & $\times$ & $\checkmark$ & Full & $\widetilde{O}(1/\sqrt{t})$ & Asm.~\ref{ass:z-bd}; $\eta{>}\frac{\sqrt{2}}{2}D$ & Cor.~\ref{cor:lip-acc-deter} \\
  & $\checkmark$ & $\times$ & $\checkmark$ & $1$ & $\widetilde{O}((1+\sigma)/\sqrt{t})$ & Asm.~\ref{ass:bd-stoch},\ref{ass:z-bd}; $\eta{>}\frac{\sqrt{2}}{2}D$ & Thm.~\ref{thm:cvx-lip-acc} \\
  & $\checkmark$ & $\checkmark$ & $\times$ & Full & $O(1/t^2)$ & Asm.~\ref{ass:z-bd}; $\eta{>}\frac{\sqrt{2}}{2}D$ & Cor.~\ref{cor:smooth-acc-deter} \\
  & $\checkmark$ & $\checkmark$ & $\times$ & $1$ & $\widetilde{O}(1/t^2+\sigma/\sqrt{t})$ & Asm.~\ref{ass:bd-stoch},\ref{ass:z-bd}; $\eta{>}\frac{\sqrt{2}}{2}D$ & Thm.~\ref{thm:smooth-acc} \\
\bottomrule
\end{tabular}
\end{table}

The remainder of the paper is organized as follows. Section~\ref{sec:setup} introduces the problem setup and Algorithm~\ref{Alg:adaprox}. Sections~\ref{sec:nc} and~\ref{sec:cvx} present convergence results for nonconvex and convex settings, respectively. Section~\ref{sec:acc} introduces the accelerated variant Algorithm~\ref{Alg:acc-adaprox} and establishes its convergence in the convex setting. Section~\ref{sec:experiments} presents numerical experiments, and Section~\ref{sec:conclusion} concludes the paper. All technical proofs are deferred to Appendices~\ref{app:aux-lemmas}--\ref{app:acc-proofs}. Appendix~\ref{app:discussion} discusses the boundedness assumption and a block-diagonal extension. Appendix~\ref{app:holder-stoch} gives the stochastic H\"older-smooth extensions, and Appendix~\ref{app:experiments} provides additional experiments.

\subsection{Related Work}

\noindent\textbf{Proximal Gradient Methods.}
The PG method \eqref{Alg:prox-grad} traces back to the splitting algorithms of \citet{lions1979splitting} and \citet{passty1979ergodic}, and is surveyed comprehensively in \citet{parikh2014proximal}. 
Special cases include projected gradient descent \citep{levitin1966constrained} (indicator $h$) and ISTA/FISTA \citep{beck2009fast} ($\ell_1$ regularizer $h$). For convex smooth $f$, the $O(1/t)$ rate is classical and can be accelerated to the optimal $O(1/t^2)$ via Nesterov-type momentum \citep{beck2009fast,nesterov1983method,nesterov2013gradient}. For convex nonsmooth $f$, PG achieves the optimal $O(1/\sqrt{t})$ rate \citep{nesterov2003introductory}. Nesterov's universal gradient method \citep{nesterov2015universal} achieves near-optimal rates for both smooth and nonsmooth objectives through line search. In the stochastic setting, \citet{lan2012optimal} establishes optimal rates for convex objectives. These results are later extended to nonconvex cases \citep{ghadimi2016accelerated,ghadimi2016mini}.

\noindent\textbf{Adaptive Gradient Methods.}
AdaGrad \citep{duchi2011adaptive} introduced adaptive step sizes by accumulating squared gradient norms, with popular variants including Adam \citep{kingma2015adam} and AMSGrad \citep{reddi2018convergence}. In the non-composite setting ($h=0$) or the constrained setting ($h$ is the indicator function of a convex set), universal accelerated methods \citep{ene2020adaptive,kavis2019unixgrad,levy2017online,levy2018online} improve over AdaGrad by achieving near-optimal rates simultaneously for convex smooth and nonsmooth objectives. For nonconvex objectives, \citet{kavis2022high} and \citet{ward2020adagrad} provide high-probability guarantees under a bounded-gradient assumption, which is later removed through refined analysis \citep{attia2023sgd,faw2022power,wang2023convergence}.
Extensions to the composite setting~\eqref{main-prob} are less studied and mainly focus on the convex regime. The original AdaGrad \citep{duchi2011adaptive} handles the composite structure but may converge slowly for smooth $f$ due to the accumulation of squared gradient norms. This issue is later addressed by \citet{joulani2020simpler} and \citet{rodomanov2024universality}, who accumulate norms of gradient differences instead. Nevertheless, these works are restricted to the convex regime, and their step sizes use an estimate of the diameter $D$, as do accelerated adaptive methods \citep{ene2020adaptive,kavis2019unixgrad,levy2018online} and our Algorithm~\ref{Alg:acc-adaprox}, but not our Algorithm~\ref{Alg:adaprox}.
More closely related to our work are AdaGrad+ and AdaACSA~\citep{ene2020adaptive} for constrained convex problems, whose step sizes are driven by gradient-mapping-type quantities: when $h$ is the indicator function of a convex set, Algorithms~\ref{Alg:adaprox} and~\ref{Alg:acc-adaprox} reduce to their scalar versions up to the choice of $\eta$ (see also \citep[Remark~5.5]{ene2020adaptive}).
Compared with these methods, our work provides several improvements: (i) general composite problems; (ii) one method for all three problem classes, including the nonconvex one; (iii) guarantees for Algorithm~\ref{Alg:adaprox} with any $\eta>0$, whereas both methods in~\citep{ene2020adaptive} set $\eta$ proportional to $D$; and (iv) a simpler analysis.
Gradient mappings are also accumulated in the adaptive bilevel method of \citet{antonakopoulos2025adaptive}, which does not target universality across problem classes.
For nonconvex composite problems, we are only aware of \citet{yun2021adaptive}, who propose a general adaptive PG framework. However, they only provide a rough convergence analysis for non-increasing step sizes, which requires prior knowledge of the smoothness constant, making the method not fully adaptive.
A separate line of adaptivity, pioneered by \citet{malitsky2020adaptive}, uses step sizes based on local gradient variation, achieving adaptivity to local smoothness without line search. This approach is later extended to the composite setting \citep{latafat2024convergence,latafat2025adaptive,malitsky2024adaptive}.

\section{Problem Setup}
\label{sec:setup}
\textbf{Notation and Definitions.}
We use $\|\cdot\|$ for the Euclidean norm, $[a]_+:=\max\{0,a\}$ for the positive part, and $\log(\cdot)$ for the natural logarithm. For a stochastic algorithm, we denote by $\ebb_k[\cdot]:=\ebb\left[\cdot \mid \mathcal{F}_{k-1}\right]$ the conditional expectation given the $\sigma$-algebra $\mathcal{F}_{k-1}$, which is the expectation given all the information available at step $k$. We use $\nabla f$ to denote the gradient or a subgradient of $f$.
For a proper, lower-semicontinuous, and convex function $h: \rbb^d \rightarrow \rbb\cup\{+\infty\}$, we denote its (effective) domain by $\dom(h):=\{x\in\rbb^d: h(x)<+\infty\}$ and its proximal operator by $\p_h(\cdot)$, defined as
\begin{equation}\label{def:prox-h}
	\p_{h}(x):=\underset{z\in \rbb^d}{\argmin}~\Big\{h(z)+\frac{1}{2}\|z-x\|^2\Big\},~\forall x\in \rbb^d.
\end{equation}
The proximal operator $\p_h(\cdot)$ is nonexpansive:
\begin{equation}\label{eq:prox-nonexpansive}
    \|\p_h(x)-\p_h(y)\| \leq \|x-y\|,~\forall x, y\in \rbb^d.
\end{equation}

Throughout the paper, we always assume that the objective function $F$ in \eqref{main-prob} is bounded from below, and we denote $F^*:=\inf_{x\in\rbb^d} F(x)$. If $F$ is convex, we also assume that it attains a global minimizer, which we denote by $x^*\in\argmin_{x\in\rbb^d} F(x)$, so that $F^*=F(x^*)$. When $f$ is nonconvex but differentiable, stationarity is measured via the gradient mapping~\citep{ghadimi2016mini}
\begin{equation}\label{def:grad-mapping}
    \gcal_\eta(x):=\frac{1}{\eta}\big(x-\p_{\eta h}(x-\eta \nabla f(x))\big),
\end{equation}
which satisfies $\gcal_\eta(x^*)=0$ if and only if $x^*$ is a stationary point of~\eqref{main-prob}.

\textbf{Algorithm Design.}
Recall that AdaGrad \citep{duchi2011adaptive} employs adaptive step sizes that accumulate $\|\nabla f(x_k)\|^2$ in the denominator. The motivation is that one has $\nabla f(x^*) = 0$ for any stationary point $x^*$ when $h=0$. Consequently, for smooth $f$, the accumulated norms $\sum_{i=1}^k\|\nabla f(x_i)\|^2$ grow slowly as $\nabla f(x_k)$ diminishes, making the step size approximately constant. For nonsmooth $f$, the subgradient norms may stay bounded away from zero, leading to decreasing step sizes $O(1/\sqrt{k})$.

However, this strategy does not directly extend to the composite setting. The reason is that the gradient $\nabla f(x^*)$ at a stationary point $x^*$ of \eqref{main-prob} is generally nonzero since the first-order optimality condition merely requires $-\nabla f(x^*)\in\partial h(x^*)$. Hence, accumulating $\|\nabla f(x_k)\|^2$ would drive the step size to zero even for smooth $f$, which is undesirable. The appropriate stationarity measure in the composite setting is the gradient mapping $\gcal_\eta(\cdot)$ \eqref{def:grad-mapping}, which reduces to $\nabla f$ if $h=0$ and satisfies $\gcal_\eta(x^*)=0$ at any stationary point $x^*$ of \eqref{main-prob}. 
Moreover, the update rules in the form of~\eqref{Alg:prox-grad} provide direct access to the gradient mapping $\gcal_k:=\gcal_{\eta_k}(x_k)$ due to
\begin{equation}\label{def:Gt}
\gcal_k=\frac{1}{\eta_k}\big(x_k-\p_{\eta_k h}(x_k-\eta_k \nabla f(x_k))\big)=\frac{1}{\eta_k}(x_k - x_{k+1}).
\end{equation}
This motivates us to accumulate $\|\gcal_k\|^2$ in the denominator of step sizes, yielding the following adaptive scheme at iteration $k$:
\begin{equation}\label{eq:Sk-def}
\eta_k=\frac{\eta}{S_k}, \quad S_k^2=\gamma^2+\sum_{j=1}^{k-1}\|\gcal_j\|^2,
\end{equation}
where $\gamma, \eta>0$ are hyperparameters. Unlike standard AdaGrad, we accumulate up to index $k-1$ instead of $k$ because $\gcal_k=(x_k-x_{k+1})/\eta_k$ is not available until the $k$-th iteration is performed. Note that the step size \eqref{eq:Sk-def} is different from the ones in \citet{joulani2020simpler} and \citet{rodomanov2024universality}, which accumulate squared norms of gradient differences to handle the composite structure.
We focus on scalar step sizes for notational simplicity; see Appendix~\ref{app:block-diag} for a block-diagonal extension.

Applying \eqref{eq:Sk-def}, we obtain the adaptive proximal gradient method for solving \eqref{main-prob}, as summarized in Algorithm~\ref{Alg:adaprox}. 
For stochastic settings, we define the stochastic gradient mapping as
\begin{equation}\label{def:tilde-Gt}
\widetilde{\gcal}_k:=\frac{1}{\eta_k}\big(x_k-\p_{\eta_k h}(x_k-\eta_k g_k)\big) = \frac{1}{\eta_k}(x_k - x_{k+1}).
\end{equation}
Then we similarly have $S_k^2=\gamma^2+\sum_{j=1}^{k-1}\|\widetilde{\gcal}_j\|^2$.

\begin{algorithm}[t]
    \caption{Adaptive Proximal Gradient Method for Solving \eqref{main-prob}}
    \label{Alg:adaprox}
    \begin{algorithmic}
    \REQUIRE $x_1 \in \dom(h)$; hyperparameters $\gamma, \eta>0$; batch sizes $\{b_k\}_{k\ge 1}$ (stochastic setting only).
    \STATE Initialize $S_1 := \gamma$.
    \FOR{$k = 1, 2, \dots$}
    \STATE Compute $g_k:=\nabla f(x_k)$ (deterministic) or $g_k:=\frac{1}{b_k}\sum_{j=1}^{b_k}\nabla \ell(x_k;\xi_{k_j})$ (stochastic) for randomly sampled $\{\xi_{k_j}\}_{j=1}^{b_k}$.
    \STATE Update $x_{k+1}:=\argmin_{x\in\rbb^d}\Big\{\langle g_k, x\rangle+\frac{S_k}{2 \eta}\|x-x_k\|^2+h(x)\Big\}$.
    \STATE Update $S^2_{k+1}=S^2_k\big(1+\|x_{k+1}-x_k\|^2/\eta^2\big)$.
    \ENDFOR
    \end{algorithmic}
\end{algorithm}

To analyze Algorithm \ref{Alg:adaprox}, we make the following boundedness assumption on the iterates $\{x_k\}_{k\ge 1}$. Similar boundedness assumptions are widely adopted for the analysis of AdaGrad-type methods in the literature (see, e.g., \citep{duchi2011adaptive,ene2020adaptive,joulani2020simpler,kavis2019unixgrad}). In the context of the composite problem \eqref{main-prob}, this assumption holds if, e.g., $\dom(h)$ is bounded or the operator $\argmin_{x\in\rbb^d}\{\langle u, x\rangle+h(x)\}$ is bounded for all $u\in \rbb^d$~\citep[Lemma 2]{ghadimi2016accelerated}.
\begin{assumption}\label{ass:seq-bd}
    There exists a constant $D>0$ such that (a) $\|x_{k+1}-x_k\|\leq D$ for all $k\ge 1$ (bounded steps), and, if $F$ is convex, (b) $\|x_k-x^*\|\leq D$ for all $k\ge 1$ (bounded iterates).
\end{assumption}
Note that while Assumption~\ref{ass:seq-bd} is used in the analysis, knowledge of the diameter $D$ is not required for implementing Algorithm~\ref{Alg:adaprox}.
Part (a), which holds automatically for smooth $f$ in the deterministic setting (see Appendix~\ref{app:discussion}), is needed due to the delayed accumulation in \eqref{eq:Sk-def}.
Since part (b) implies part (a) with $2D$ in place of $D$, our results require only part (a) for nonconvex $f$ and only part (b) for convex $f$.
For notational convenience, we define $\Delta_k:=F(x_k)-F^*$.

\section{Convergence Analysis in the Nonconvex Setting}
\label{sec:nc}
In this section, we study the convergence of Algorithm \ref{Alg:adaprox} when $f$ is $L$-smooth but possibly nonconvex. 
For ease of presentation, we use $\lesssim$ to hide constants depending only on fixed quantities (e.g., $L, G, \gamma, \eta, D$) hereafter, i.e., $A\lesssim B$ if there exists a constant $c>0$ such that $A\leq c B$.

\textbf{Deterministic Setting.}
We first consider $g_k:=\nabla f(x_k)$.
The following theorem, whose proof relies on the boundedness of $\{\Delta_k\}_{k\ge 1}$, establishes an $O(1/t)$ stationarity rate for Algorithm~\ref{Alg:adaprox}.
\begin{theorem}\label{thm:nc-deter}
Suppose $f$ is $L$-smooth. Let $\{x_k\}_{k\ge 1}$ be generated by Algorithm~\ref{Alg:adaprox} with $g_k:=\nabla f(x_k)$ and any $\gamma, \eta>0$. Then for any $t \geq 1$, we have $\frac{1}{t}\sum_{k=1}^t\|\gcal_k\|^2 \lesssim \frac{1}{t}$.
\end{theorem}

\textbf{Stochastic Setting.}
We proceed to study the case $g_k:=\frac{1}{b_k}\sum_{j=1}^{b_k}\nabla \ell(x_k;\xi_{k_j})$ for some batch size $b_k\ge 1$, where the samples $\{\xi_{k_j}\}$ are drawn independently of each other and of the past iterates.
We first make the following unbiasedness and bounded-variance assumption.
\begin{assumption}\label{ass:bd-stoch}
$\ebb_k[\nabla\ell(x_k;\xi)]=\nabla f(x_k)$ for all $k \ge 1$, where the expectation is taken w.r.t. the randomness of $\xi$.
Moreover, there exists $\sigma\ge 0$ such that $\ebb_k[\|\nabla\ell(x_k;\xi)-\nabla f(x_k)\|^2] \leq \sigma^2,~\forall k\ge1$.
\end{assumption}
Analysis of Algorithm \ref{Alg:adaprox} in such a stochastic setting is more challenging, because the adaptive step sizes $\{\eta_k\}_{k\ge 1}$ become random variables. A direct consequence is that our deterministic analysis relying on bounding objective gaps is not directly applicable. To see this, while we may still show that the objective gaps $\{\Delta_k\}_{k\ge 1}$ are uniformly bounded in expectation, it is hard to bound
\[
\ebb\bigg[\sum_{k=2}^t\Big(\frac{1}{\eta_k} - \frac{1}{\eta_{k-1}}\Big)\Delta_k\bigg]
\]
due to the coupling between the two random variables $\frac{1}{\eta_k} - \frac{1}{\eta_{k-1}}$ and, which is crucial for establishing Theorem~\ref{thm:nc-deter}. This issue regarding randomness in $\{\eta_k\}_{k\ge 1}$ is also identified as the main bottleneck for nonconvex analysis of AdaGrad in non-composite settings \citep{attia2023sgd,faw2022power,wang2023convergence,ward2020adagrad}. Most of these works handle it by introducing step-size surrogates that are measurable given all historical iterates, which leads to highly involved analyses.
Besides, the composite structure of \eqref{main-prob} introduces additional challenges because in this case $\langle\nabla f(x_k), x_{k+1} - x_k\rangle = -\eta_k\langle\nabla f(x_k), \widetilde{\gcal}_k\rangle \neq -\eta_k\langle\nabla f(x_k), g_k\rangle$, which makes existing nonconvex analyses for non-composite settings inapplicable.

To overcome this difficulty, we develop a new approach by carefully estimating the sum of inner products between the stochastic noise and iterate differences, which allows us to bound $\ebb[S_t]$ without introducing new surrogates and leads to a substantially simplified analysis. We present this technical ingredient in the following lemma. Note that it does not rely on any conditions on $f$ and hence is applicable across different settings of $f$.

\begin{lemma}\label{lem:stoc-ip-sum}
    Suppose Assumptions \ref{ass:seq-bd}(a) and \ref{ass:bd-stoch} hold.
Let $\{x_k\}_{k\ge 1}$ be generated by Algorithm \ref{Alg:adaprox} with batch sizes $\{b_k\}_{k\ge 1}$ and any $\gamma, \eta>0$.
Then for any $t \geq 1$, we have
    \[
    \ebb\Big[\sum_{k=1}^t\langle\nabla f(x_k)-g_k, x_{k+1}-x_k\rangle\Big] \leq \sigma \sqrt{2(\eta^2+D^2) \sum_{k=1}^t \frac{1}{b_k}} \sqrt{\log \ebb\bigg[\frac{S_{t+1}}{\gamma}\bigg]}.
    \]
\end{lemma}

Using Lemma \ref{lem:stoc-ip-sum}, we can establish convergence to stationary points for Algorithm \ref{Alg:adaprox}.
\begin{theorem}
\label{thm:nc-stoc}
Suppose $f$ is $L$-smooth and Assumptions \ref{ass:seq-bd}(a) and \ref{ass:bd-stoch} hold.
Let $\{x_k\}_{k\ge 1}$ be generated by Algorithm \ref{Alg:adaprox} with batch sizes $\{b_k\}_{k\ge 1}$ and any $\gamma, \eta>0$.
Then for any $t\ge 1$, we have

\begin{equation}\label{eq:nc-stoc-main}
\Bigg(\ebb\Bigg[\sqrt{\frac{1}{t}\sum_{k=1}^t\|\gcal_k\|^2}\,\Bigg]\Bigg)^2 \lesssim \frac{1}{t}\Bigg(1+\sigma^2\sum_{k=1}^t\frac{1}{b_k}\Bigg) \Bigg(1+\log^2 \Bigg(1+\sigma\sqrt{\sum_{k=1}^t\frac{1}{b_k}}\Bigg)\Bigg).
\end{equation}

\end{theorem}

\begin{remark}
    With unit batch size $b_k\equiv 1$, the bound in Theorem~\ref{thm:nc-stoc} reduces to $\widetilde{O}\big(\frac{1}{t}+\sigma^2\big)$, which only ensures convergence to a neighborhood of a stationary point. This marks a key difference from AdaGrad analysis in the non-composite setting, where unit batch size suffices for convergence to a stationary point \citep{attia2023sgd,faw2022power,wang2023convergence,ward2020adagrad}. Such behavior is not an artifact of our analysis, as it matches the behavior of PG in the stochastic nonconvex setting \citep{ghadimi2016mini}, whose convergence bound also contains a residual term proportional to $\frac{\sigma^2}{t}\sum_{k=1}^t\frac{1}{b_k}$. To ensure convergence to a stationary point, one may take increasing batch sizes, e.g., $b_k=\Omega(\sigma\sqrt{k})$, which yields the rate $\widetilde{O}\big(\frac{1}{t}+\frac{\sigma}{\sqrt{t}}\big)$, matching the result of \citet{ghadimi2016mini}.
\end{remark}

\begin{remark}\label{rmk:high-prob}
Applying Markov's inequality to Theorem~\ref{thm:nc-stoc} further gives a high-probability bound: for any $\epsilon\in(0,1)$, with probability at least $1-\epsilon$, $\frac{1}{t}\sum_{k=1}^t\|\gcal_k\|^2$ is at most $1/\epsilon^2$ times the right-hand side of \eqref{eq:nc-stoc-main} (see Appendix~\ref{app:thm-ncstoc}).
\end{remark}

\section{Convergence Analysis in the Convex Setting}
\label{sec:cvx}
In this section, we study the convergence of Algorithm \ref{Alg:adaprox} under the condition that $f$ is convex. Leveraging Lemma \ref{lem:stoc-ip-sum}, we show that Algorithm \ref{Alg:adaprox} automatically achieves convergence rates matching those of PG for both Lipschitz continuous and smooth $f$. For brevity, we focus on stochastic settings and fix $b_k\equiv 1$ (i.e., $g_k=\nabla \ell(x_k;\xi_k)$) throughout this section, and the established results easily extend to mini-batch and deterministic settings. Denote $\bar{x}_{t}:=\frac{1}{t}\sum_{k=1}^{t} x_{k+1}$.

\textbf{Lipschitz Continuous $f$.}
We establish the following convergence result for the averaged iterate of Algorithm~\ref{Alg:adaprox} when $f$ is convex and $G$-Lipschitz continuous.
\begin{theorem}
\label{thm:stoch-convex-nonsmooth}
Suppose $f$ is convex and $G$-Lipschitz continuous, and Assumptions \ref{ass:seq-bd}(b) and \ref{ass:bd-stoch} hold.
Let $\{x_k\}_{k\ge 1}$ be generated by Algorithm \ref{Alg:adaprox} with $b_k\equiv1$ and any $\gamma, \eta>0$.
Then for any $t \ge 1$, we have
\[
\ebb[F(\bar{x}_{t})-F^*] \lesssim \frac{(1+\sigma)\sqrt{\log \big(1+(1+\sigma^2)t\big)}}{\sqrt{t}}.
\]
\end{theorem}

Setting $\sigma=0$ in Theorem \ref{thm:stoch-convex-nonsmooth}, we obtain an $O\big(\sqrt{\log(1+t)/t}\big)$ deterministic rate for Algorithm \ref{Alg:adaprox}. %
\begin{corollary}\label{cor:cvx-ns-deter}
Suppose $f$ is convex and $G$-Lipschitz continuous, and Assumption \ref{ass:seq-bd}(b) holds.
Let $\{x_k\}_{k\ge 1}$ be generated by Algorithm \ref{Alg:adaprox} with $g_k=\nabla f(x_k)$ and any $\gamma, \eta>0$.
Then for any $t \ge 1$, it holds that $F(\bar{x}_{t})-F^* \lesssim \sqrt{\log(1+t)/t}$.
\end{corollary}

\textbf{Smooth $f$.}
When $f$ is convex and $L$-smooth, we derive the following convergence result.
\begin{theorem}
\label{thm:stoch-convex-smooth}
Suppose $f$ is convex and $L$-smooth, and Assumptions \ref{ass:seq-bd}(b) and \ref{ass:bd-stoch} hold.
Let $\{x_k\}_{k\ge 1}$ be generated by Algorithm \ref{Alg:adaprox} with $b_k\equiv1$ and any $\gamma, \eta>0$.
Then for any $t \ge 1$, we have
\[
\ebb[F(\bar{x}_{t})-F^*] \lesssim \frac{1+\log \big(1+\sigma\sqrt{t}\big)}{t}+\frac{\sigma\log \big(1+\sigma \sqrt{t}\big)}{\sqrt{t}}.
\]
\end{theorem}
Setting $\sigma=0$ in Theorem \ref{thm:stoch-convex-smooth} leads to an $O(1/t)$ deterministic convergence rate for Algorithm \ref{Alg:adaprox}.
\begin{corollary}
\label{coro:stoch-convex-smooth}
Suppose $f$ is convex and $L$-smooth, and Assumption \ref{ass:seq-bd}(b) holds.
Let $\{x_k\}_{k\ge 1}$ be generated by Algorithm \ref{Alg:adaprox} with $g_k=\nabla f(x_k)$ and any $\gamma, \eta>0$.
Then for any $t \ge 1$, we have $F(\bar{x}_{t})-F^* \lesssim 1/t$.
\end{corollary}

\section{Acceleration under Convexity and Known Diameter}
\label{sec:acc}
In this section, we study an accelerated version of Algorithm \ref{Alg:adaprox} by combining the accelerated stochastic approximation (AC-SA) algorithm \citep{lan2012optimal} with adaptive step sizes. The resulting method is summarized in Algorithm \ref{Alg:acc-adaprox}, which extends the scalar AdaACSA \citep{ene2020adaptive} to composite problems.

\begin{algorithm}[t]
    \caption{Adaptive Accelerated Proximal Gradient Method for Solving \eqref{main-prob}}
    \label{Alg:acc-adaprox}
    \begin{algorithmic}
    \REQUIRE initial points $y_1=z_1\in\dom(h)$; hyperparameters $\gamma,\eta>0$; batch sizes $\{b_k\}_{k\ge 1}$ (stochastic setting only).
    \STATE Initialize $S_1:=\gamma$ and $\alpha_0:= 0$.
    \FOR{$k=1,2,\dots$}
    \STATE Update weight $\alpha_k:=\frac{1+\sqrt{1+4\alpha_{k-1}^2}}{2}$; set $\theta_k:=1/\alpha_k$.
    \STATE Set $x_k=(1-\theta_k)y_k+\theta_k z_k$.
    \STATE Compute $g_k:=\nabla f(x_k)$ (deterministic) or $g_k:=\frac{1}{b_k}\sum_{j=1}^{b_k}\nabla \ell(x_k;\xi_{k_j})$ (stochastic) for randomly sampled $\{\xi_{k_j}\}_{j=1}^{b_k}$.
    \STATE Update the next iterates $y_{k+1}$ and $z_{k+1}$ via
    \begin{align}
    & z_{k+1}=\argmin_{x\in\rbb^d}\Big\{\langle g_k, x\rangle+\frac{\theta_k S_k}{2 \eta}\|x-z_k\|^2+h(x)\Big\},\label{alg:z-update} \\
    & y_{k+1}=x_k+\theta_k\big(z_{k+1}-z_k\big). \label{alg:y-update}
    \end{align}
    \STATE Update $S^2_{k+1}=S^2_k(1+\|z_{k+1}-z_k\|^2/\eta^2)$.
    \ENDFOR
    \end{algorithmic}
\end{algorithm}

Since the update of the intermediate point $z_{k+1}$ in Algorithm \ref{Alg:acc-adaprox} requires solving a proximal subproblem \eqref{alg:z-update}, we replace Assumption \ref{ass:seq-bd} with the boundedness of $\{z_k\}_{k\ge 1}$ throughout this section.
\begin{assumption}\label{ass:z-bd}
    There exists $D>0$ such that $\|z_k-x^*\|\leq D$ for all $k\ge 1$.
\end{assumption}
Under Assumption \ref{ass:z-bd}, we also have bounded steps since $\|z_{k+1}-z_k\|\leq\|z_{k+1}-x^*\|+\|z_k-x^*\|\leq2D$ for all $k\ge1$.
As in the previous section, we focus on convergence results for Algorithm \ref{Alg:acc-adaprox} in stochastic settings and fix the batch size $b_k\equiv 1$ for simplicity.
We demonstrate that, when $f$ is convex and the diameter $D$ in Assumption \ref{ass:z-bd} is known, Algorithm \ref{Alg:acc-adaprox} preserves a near-optimal convergence rate similar to that of Algorithm \ref{Alg:adaprox} if $f$ is Lipschitz continuous. 
Moreover, it achieves an improved convergence rate for $L$-smooth $f$, which is optimal up to logarithmic factors~\citep{lan2012optimal}.
A further difference is that while the convergence of Algorithm~\ref{Alg:adaprox} is measured at the averaged iterate $\bar{x}_t$, that of Algorithm~\ref{Alg:acc-adaprox} is at the last iterate $y_t$, which is preferred in practice as it requires no additional computation.

\textbf{Lipschitz Continuous $f$.}
We first consider $G$-Lipschitz continuous $f$. The following theorem shows that if $\eta> \sqrt{2}D/2$, then Algorithm \ref{Alg:acc-adaprox} achieves a near-optimal convergence rate $\widetilde{O}\big((1+\sigma)/\sqrt{t}\big)$.
\begin{theorem}\label{thm:cvx-lip-acc}
Suppose $f$ is convex and $G$-Lipschitz continuous, and Assumptions \ref{ass:bd-stoch} and \ref{ass:z-bd} hold. Let $\{y_k\}_{k\ge 1}$ be generated by Algorithm \ref{Alg:acc-adaprox} with $b_k\equiv 1$, $\gamma>0$, and $\eta> \sqrt{2}D/2$. Then for any $t \ge 1$, we have
\[
\ebb[F(y_{t+1})-F^*] \lesssim \frac{1+\sigma}{\sqrt{t}} \sqrt{\log \big(1+(1+\sigma^2) t^3\big)}.
\]
\end{theorem}
Setting $\sigma=0$ in Theorem \ref{thm:cvx-lip-acc}, we obtain a deterministic rate $O\big(\sqrt{\log(1+t)/t}\big)$ for Algorithm \ref{Alg:acc-adaprox}.
\begin{corollary}\label{cor:lip-acc-deter}
Suppose $f$ is convex and $G$-Lipschitz continuous, and Assumption~\ref{ass:z-bd} holds.
Let $\{y_k\}_{k\ge 1}$ be generated by Algorithm \ref{Alg:acc-adaprox} with $g_k=\nabla f(x_k)$, $\gamma>0$, and $\eta> \sqrt{2}D/2$.
Then for any $t \ge 1$, it holds that $F(y_{t+1})-F^* \lesssim \sqrt{\log(1+t)/t}$.
\end{corollary}

\textbf{Smooth $f$.}
When $f$ is convex and $L$-smooth, we show in the upcoming theorem that Algorithm \ref{Alg:acc-adaprox} achieves a near-optimal convergence rate if $\eta> \sqrt{2}D/2$. This rate can be significantly faster than that of Algorithm~\ref{Alg:adaprox}, particularly when the noise level $\sigma^2$ is small.
\begin{theorem}\label{thm:smooth-acc}
    Suppose $f$ is convex and $L$-smooth, and Assumptions \ref{ass:bd-stoch} and \ref{ass:z-bd} hold. Let $\{y_k\}_{k\ge 1}$ be generated by Algorithm \ref{Alg:acc-adaprox} with $b_k\equiv 1$, $\gamma>0$, and $\eta> \sqrt{2}D/2$. Then for any $t \ge 1$, we have
    \[
    \ebb[F(y_{t+1})-F^*] \lesssim \frac{1+\log \big(1+\sigma t^{3/2}\big)}{t^2}+\frac{\sigma\log \big(1+\sigma t^{3/2}\big)}{\sqrt{t}}.
    \]
\end{theorem}

Setting $\sigma=0$ in Theorem \ref{thm:smooth-acc}, we derive an optimal deterministic rate $O\big(1/t^2\big)$ for Algorithm \ref{Alg:acc-adaprox}.
\begin{corollary}\label{cor:smooth-acc-deter}
Suppose $f$ is convex and $L$-smooth, and Assumption~\ref{ass:z-bd} holds.
Let $\{y_k\}_{k\ge 1}$ be generated by Algorithm \ref{Alg:acc-adaprox} with $g_k=\nabla f(x_k)$, $\gamma>0$, and $\eta> \sqrt{2}D/2$.
Then for any $t \ge 1$, it holds that $F(y_{t+1})-F^* \lesssim 1/t^2$.
\end{corollary}

\begin{remark}\label{rmk:eta-cond}
The condition $\eta>\sqrt{2}D/2$ ensures that, in our bound on $\ebb[S_{t+1}]$, the term $-\eta S_{t+1}$ from the adaptive step sizes dominates the term $\frac{D^2}{2\eta}S_{t+1}$ from the telescoping distances (see Appendix~\ref{app:thm-cvxlipacc}). Note that accelerated AdaGrad-type methods including~\citep{ene2020adaptive,kavis2019unixgrad,levy2018online} also need $\eta$ proportional to $D$. This condition $\eta>\sqrt{2}D/2$ is only for theoretical analysis.
Empirically, Algorithm~\ref{Alg:acc-adaprox} also converges for $\eta$ less than $\sqrt{2}D/2$, as shown in Appendix~\ref{app:small-eta}.
\end{remark}

\begin{remark}\label{rmk:nc-alg2}
While Algorithm~\ref{Alg:acc-adaprox} achieves accelerated convergence guarantees for convex objectives, whether it converges to stationary points of smooth nonconvex objectives at the optimal rate remains an open question.
Establishing such guarantees while preserving the accelerated rates in the convex case is an interesting direction for future research.
\end{remark}

\section{Experiments}
\label{sec:experiments}
In this section, we conduct numerical experiments on the nonconvex support vector machine (SVM) problem~\citep{wang2017stochastic} with $\ell_1$ regularization and box constraints:
\begin{equation}\label{prob:svm}
\min_{x\in[-R,R]^d} F(x):=\frac{1}{n} \sum_{i=1}^n\left[1-\tanh \left(b_i\left\langle x, a_i\right\rangle\right)\right]+\frac{\mu}{2}\|x\|^2_2+\lambda\|x\|_1,
\end{equation}
where $(a_i, b_i)\in \rbb^d\times \{-1,1\}$ denotes the $i$-th sample with feature vector $a_i\in\rbb^d$ and label $b_i\in\{-1,1\}$. 
This problem fits into the composite formulation~\eqref{main-prob} by defining
\[
f:=\frac{1}{n} \sum_{i=1}^n\left[1-\tanh \left(b_i\left\langle \cdot, a_i\right\rangle\right)\right]+\frac{\mu}{2}\|\cdot\|^2_2,\quad h := \lambda\|\cdot\|_1 + \iota_{[-R, R]^d}(\cdot),
\]
where $\iota_{[-R, R]^d}(\cdot)$ denotes the indicator function of the set $[-R, R]^d$. Moreover, since $\dom(h)$ is bounded, Assumptions~\ref{ass:seq-bd} and~\ref{ass:z-bd} are satisfied for Algorithms~\ref{Alg:adaprox} and~\ref{Alg:acc-adaprox}, respectively. All experiments were conducted on an Apple M4 Pro CPU.

For all experiments on \eqref{prob:svm}, we set the box parameter to $R=50$ and the regularization parameters to $\mu=\lambda=10^{-3}$. Since the objective function $F$ in \eqref{prob:svm} is nonconvex, we measure the performance of algorithms using the norm of the gradient mapping $\gcal(\cdot):=\gcal_1(\cdot)=x-\p_{h}\big(x-\nabla f(x)\big)$. Furthermore, due to the separability of the $\ell_1$-norm and the box constraints, the proximal operator of $h$ can be computed efficiently by first applying the soft-thresholding operator and then projecting onto the box constraints.

\textbf{Comparison in Deterministic Settings.} We first compare the full-batch implementations of three methods on the MNIST dataset~\citep{lecun2002gradient}: Algorithm~\ref{Alg:adaprox}, Algorithm~\ref{Alg:acc-adaprox}, and the proximal gradient~\eqref{Alg:prox-grad} method.
We convert the dataset into a binary classification task by assigning digits $0$--$5$ to label $-1$ and digits $6$--$9$ to label $+1$.
Our goal is to assess the robustness of these methods under different step-size schemes.
For Algorithms~\ref{Alg:adaprox} and~\ref{Alg:acc-adaprox}, we fix $\gamma=1$ and vary $\eta \in \{1, 10, 100\}$.
For PG, we use constant step sizes $\eta_k \equiv \alpha \in \{0.1/L, 1/L, 10/L\}$ (classical theory suggests that $\alpha=1/L$ is a safe choice \citep{ghadimi2016mini}).
Each method is run for 10,000 iterations, using the same initial point for a fair comparison.

Figure~\ref{fig:svm_det} shows the evolution of the gradient mapping norm for all three methods under different step-size schemes. We observe that PG converges when the step size is $1/L$ or smaller, but may diverge for larger values such as $10/L$. In contrast, Algorithm~\ref{Alg:adaprox} exhibits fast and robust convergence to stationary points across all tested values of $\eta$, supporting Theorem~\ref{thm:nc-deter} and highlighting the advantage of adaptive step sizes. Moreover, Algorithm~\ref{Alg:acc-adaprox} shows slower decay of the gradient mapping norm across all tested values of $\eta$ compared to Algorithm \ref{Alg:adaprox}, which indicates that Algorithm~\ref{Alg:acc-adaprox} may be less effective in the nonconvex setting (see also Remark~\ref{rmk:nc-alg2}).

\begin{figure}[t]
    \centering
    \begin{subfigure}[b]{0.328\textwidth}
        \includegraphics[width=\textwidth]{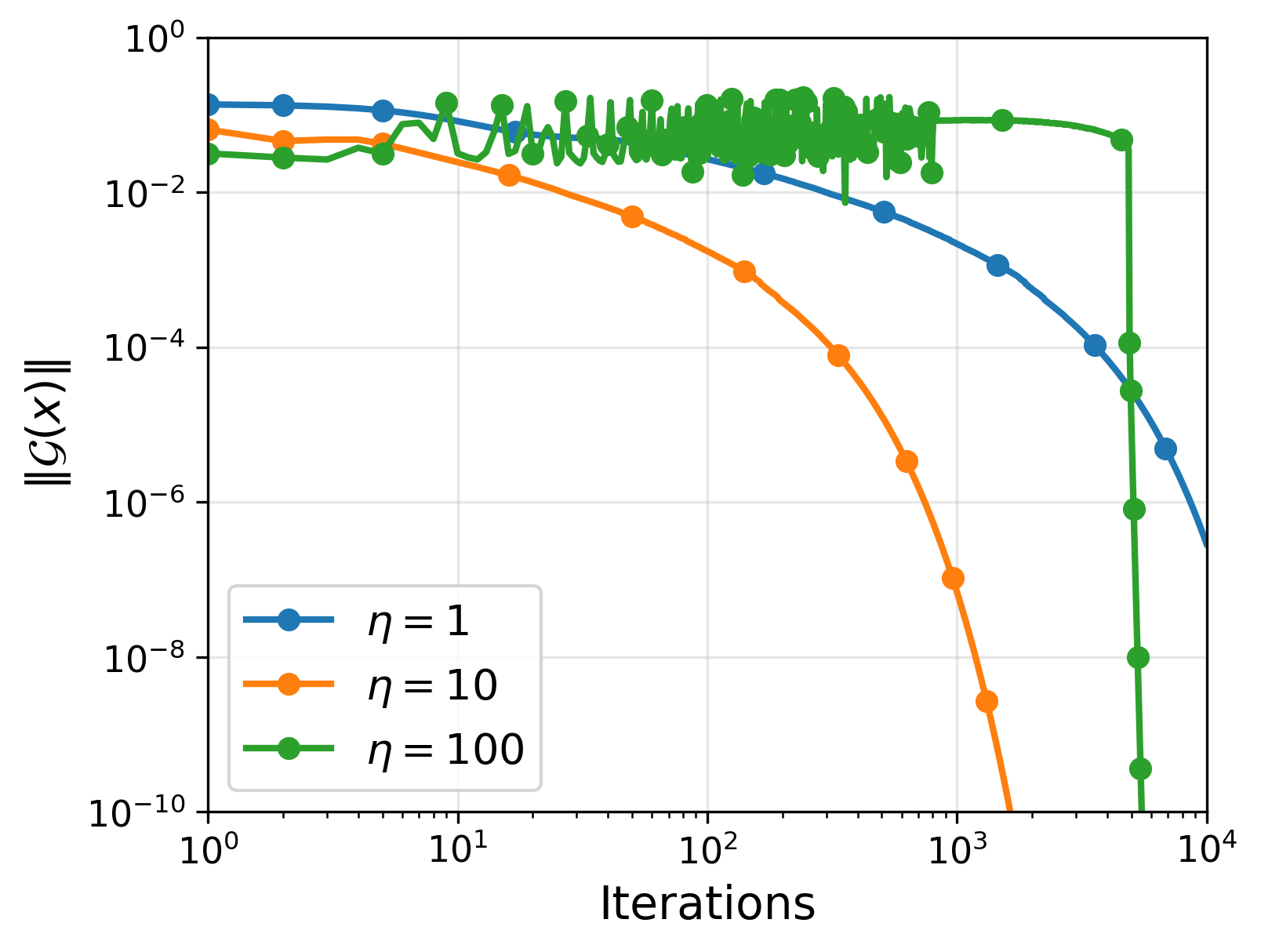}
        \caption{Algorithm \ref{Alg:adaprox}}
        \label{fig:svm-deter-alg1}
    \end{subfigure}
    \begin{subfigure}[b]{0.328\textwidth}
        \includegraphics[width=\textwidth]{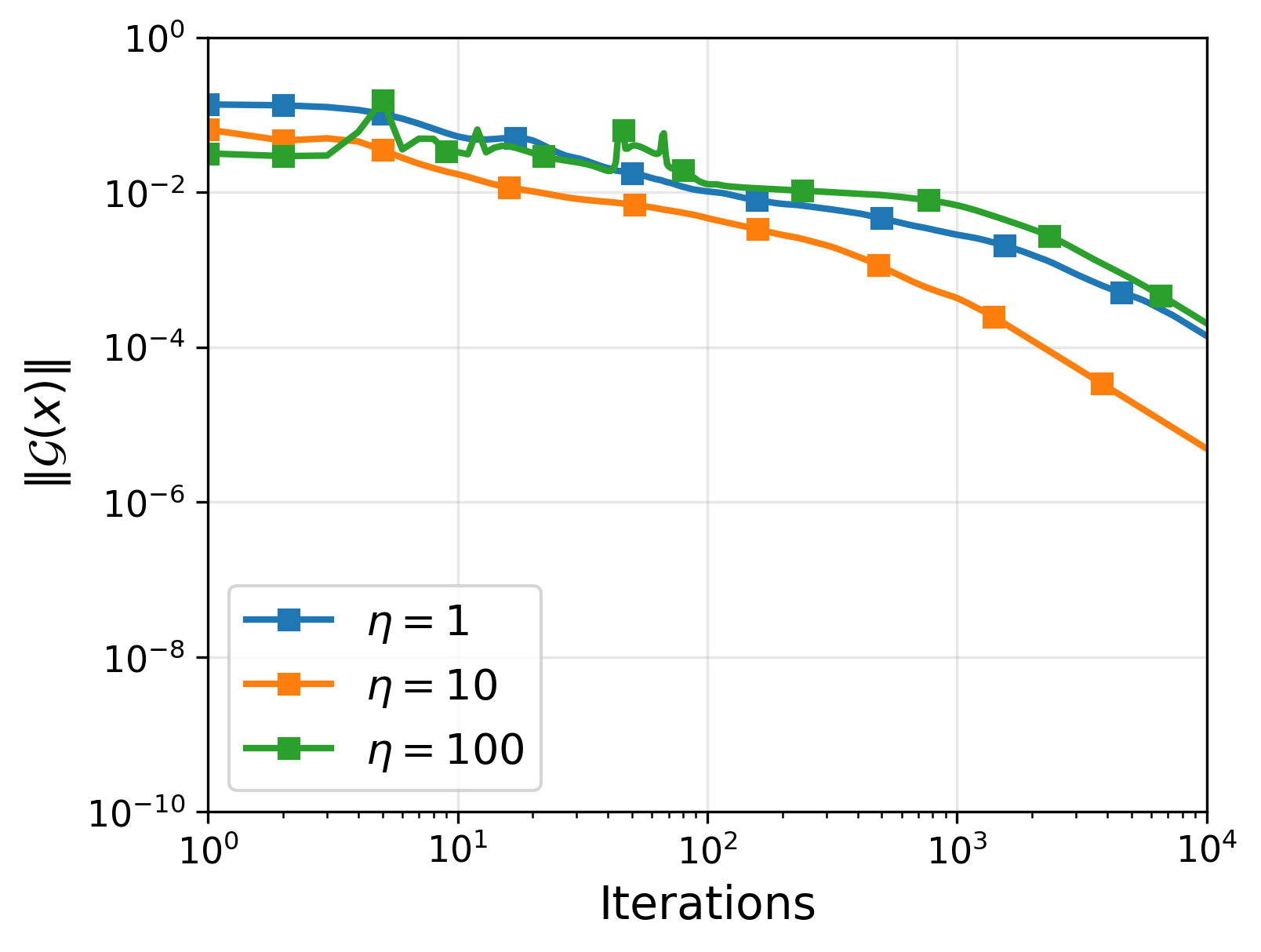}
        \caption{Algorithm \ref{Alg:acc-adaprox}}
        \label{fig:svm-deter-alg2}
    \end{subfigure}
    \begin{subfigure}[b]{0.328\textwidth}
        \includegraphics[width=\textwidth]{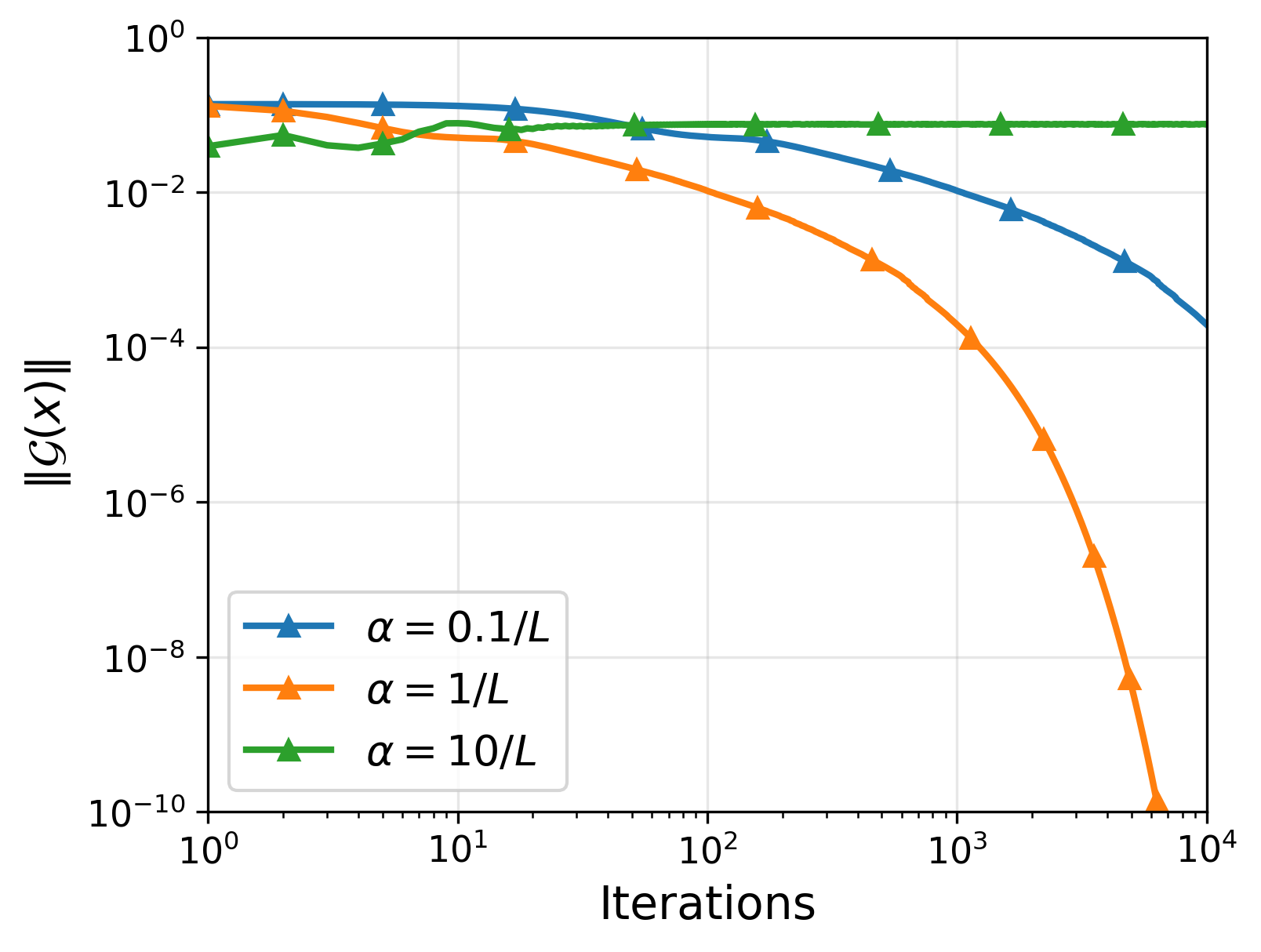}
        \caption{PG}
        \label{fig:svm-deter-pg}
    \end{subfigure}
    \caption{Evolution of $\|\gcal(x)\|$ for solving the nonconvex SVM problem \eqref{prob:svm} on the MNIST dataset in the deterministic setting. 
    }
    \label{fig:svm_det} 
\end{figure}

\textbf{Effect of Batch Size.} We next investigate the effect of batch size on the performance of Algorithm~\ref{Alg:adaprox} in the stochastic setting. Specifically, we use the a9a dataset~\citep{chang2011libsvm} and randomly split it into a training set ($80\%$ of the samples) and a test set (the remaining $20\%$). We fix the hyperparameters $\gamma=1$ and $\eta=10$, and consider five constant batch sizes $b\in\{0.1\%, 1\%, 10\%, 50\%, 100\%\}$ of the training set size. For fairness, we run Algorithm~\ref{Alg:adaprox} with each batch size for 1000 epochs (i.e., full passes over the training data), so that all configurations use the same total number of gradient evaluations.

Figure~\ref{fig:svm_batch} illustrates the evolution of the last-iterate and averaged gradient mapping norms, as well as the test accuracy, for different batch sizes. As shown in plots (a) and (b), both the last-iterate and averaged gradient mapping norms decrease faster (in terms of epochs) when larger batch sizes are used. In contrast, plot (c) indicates that smaller batch sizes yield faster improvements in test accuracy during the early stages of training. This suggests a trade-off between optimization efficiency and generalization performance when selecting the batch size.

\begin{figure}[t]
    \centering
    \begin{subfigure}[b]{0.33\textwidth}
        \includegraphics[width=\textwidth]{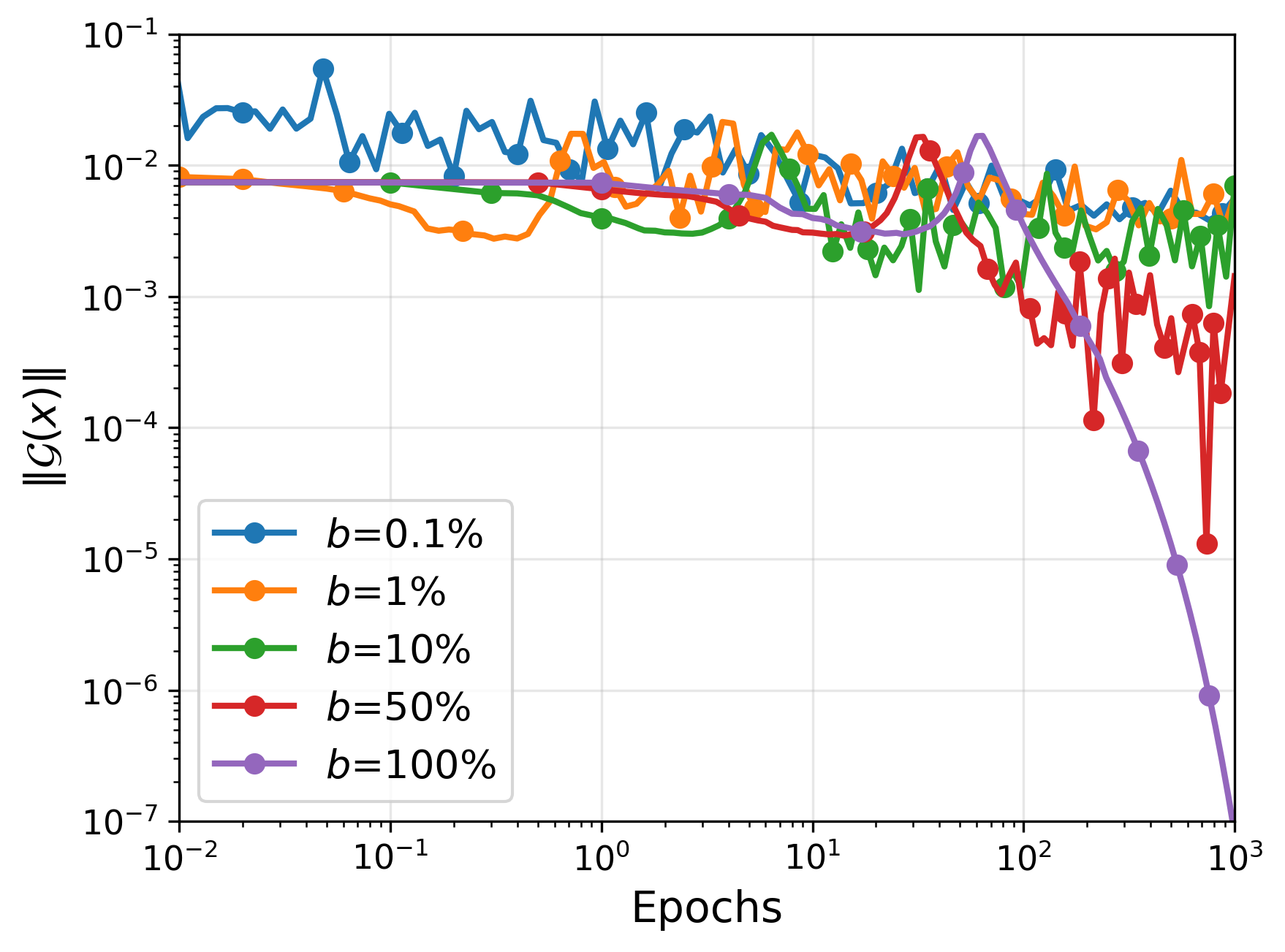}
        \caption{$\|\gcal(x_t)\|$}
        \label{fig:svm_batch_grad_last}
    \end{subfigure}
    \begin{subfigure}[b]{0.325\textwidth}
        \includegraphics[width=\textwidth]{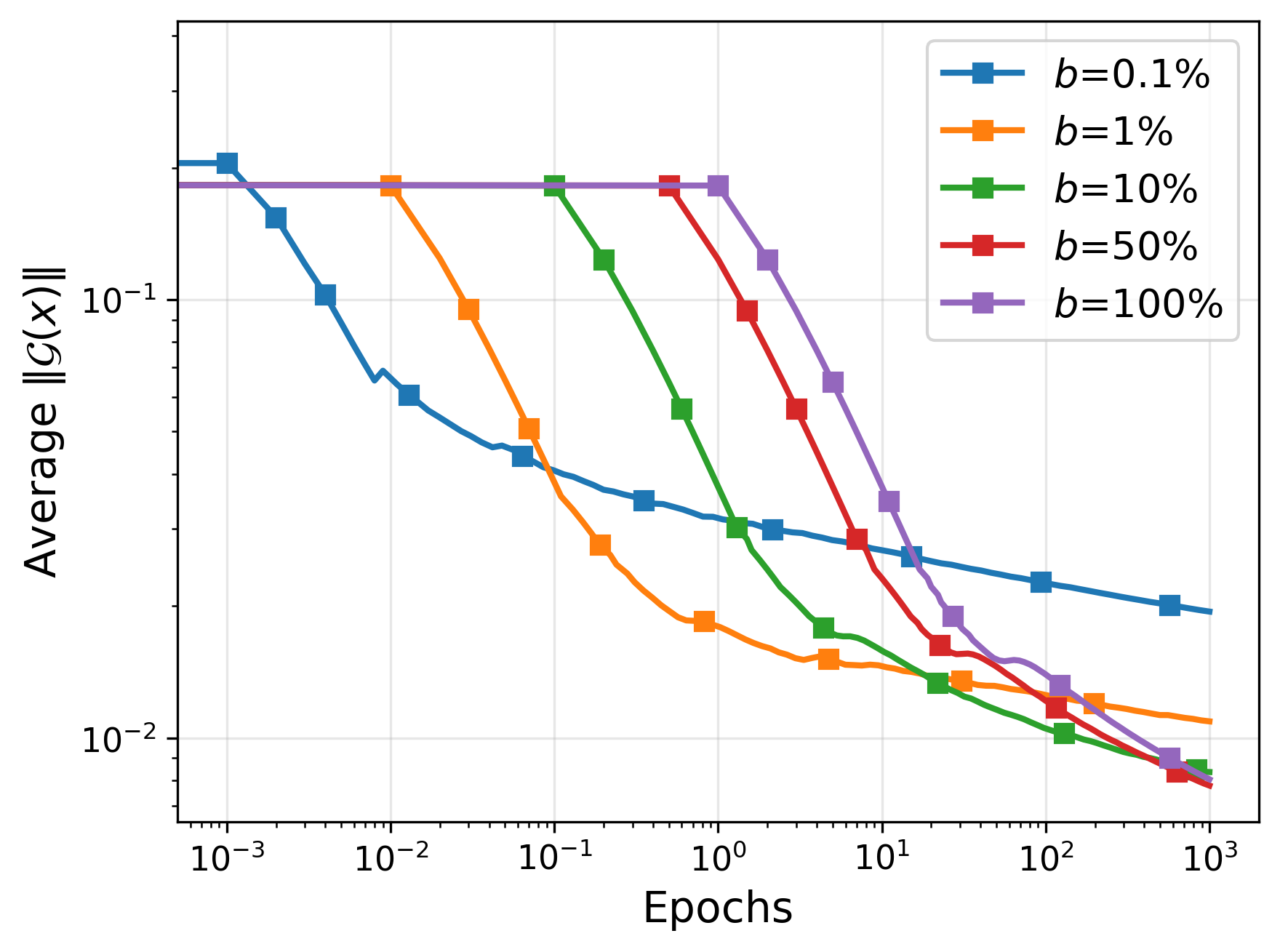}
        \caption{$\frac{1}{t}\sum_{k=1}^t\|\gcal(x_k)\|$}
        \label{fig:svm_batch_grad_ave}
    \end{subfigure}
    \begin{subfigure}[b]{0.33\textwidth}
        \includegraphics[width=\textwidth]{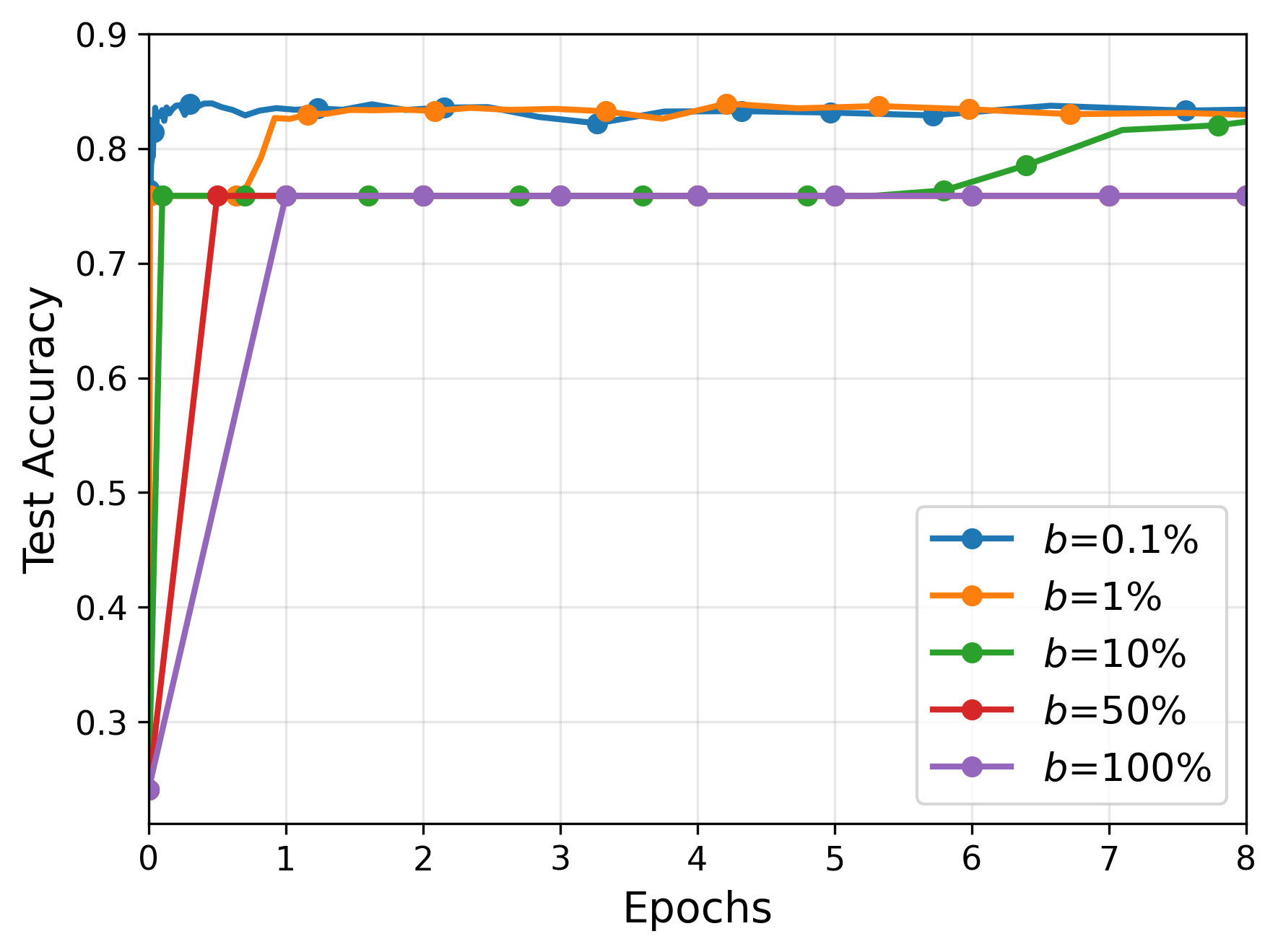}
        \caption{Test accuracy}
        \label{fig:svm_batch_acc}
    \end{subfigure}
    \caption{Effect of batch size on Algorithm~\ref{Alg:adaprox} for solving the nonconvex SVM problem~\eqref{prob:svm} on the a9a dataset. Here, $b=1\%$ means the gradient batch size is $1\%$ of the training set size.}
    \label{fig:svm_batch}
\end{figure}

\textbf{Additional Experiments.} Appendix~\ref{app:experiments} reports experiments on logistic regression; comparisons with AdaGrad-Norm, coordinate-wise AdaGrad, the methods in \citep{rodomanov2024universality}, DoG \citep{ivgi2023dog}, and DoWG \citep{khaled2023dowg} in both deterministic and mini-batch settings; and sensitivity tests on the hyperparameters.

\section{Conclusion}
\label{sec:conclusion}
We proposed an adaptive proximal gradient method for solving composite optimization problems, where one component is a simple convex function and the other belongs to one of the three classes: nonconvex smooth, convex nonsmooth, or convex smooth. The method is characterized by an adaptive step size that accumulates historical gradient mapping norms in the denominator, replacing the gradient norms used in standard AdaGrad to better accommodate the composite structure. We showed that the method converges across all three problem classes without any modification or knowledge of problem-related parameters, achieving rates matching those of proximal gradient methods up to logarithmic factors in both deterministic and stochastic settings. For the convex case, we further developed an accelerated variant that retains near-optimal rates for convex nonsmooth objectives and achieves an improved rate of $\widetilde{O}(1/t^2 + \sigma/\sqrt{t})$ for convex smooth objectives. To handle the randomness of the adaptive step sizes in the stochastic setting, we developed new technical tools to bound cumulative inner products between the stochastic noise and iterate differences, which are applicable across all three problem classes and enable simplified convergence analyses. Our experiments validate the theoretical findings and demonstrate the practical competitiveness of the proposed algorithms compared with classical proximal gradient methods.

Several directions remain open.
The logarithmic factors in our convergence rates arise from bounding the growth of the denominator in the adaptive step sizes, and it is unclear whether they can be removed by a sharper analysis or a modified algorithm.
Another challenging open direction is to develop an adaptive accelerated method that simultaneously achieves optimal convergence rates across all three problem classes without modifications or knowledge of problem-related parameters, thereby unifying the universality of the first algorithm we proposed with the accelerated rates of the second.
Finally, further potential improvements include relaxing the boundedness assumptions beyond deterministic smooth settings, and removing the condition $\eta>\sqrt{2}D/2$ in the analysis of the accelerated method.

\begin{ack}
This work was supported by the Wallenberg AI, Autonomous Systems and Software Program (WASP) funded by the Knut and Alice Wallenberg Foundation.
\end{ack}

\setlength{\bibsep}{0.111cm}
\bibliographystyle{abbrvnat}
\bibliography{reference}

\begin{thebibliography}{46}
\providecommand{\natexlab}[1]{#1}
\providecommand{\url}[1]{\texttt{#1}}
\expandafter\ifx\csname urlstyle\endcsname\relax
  \providecommand{\doi}[1]{doi: #1}\else
  \providecommand{\doi}{doi: \begingroup \urlstyle{rm}\Url}\fi

\bibitem[Antonakopoulos et~al.(2025)Antonakopoulos, Sabach, Viano, Hong, and
  Cevher]{antonakopoulos2025adaptive}
K.~Antonakopoulos, S.~Sabach, L.~Viano, M.~Hong, and V.~Cevher.
\newblock Adaptive bilevel optimization.
\newblock \emph{ACM/IMS Journal of Data Science}, 2\penalty0 (2):\penalty0
  1--29, 2025.
\newblock \doi{10.1145/3728478}.

\bibitem[Attia and Koren(2023)]{attia2023sgd}
A.~Attia and T.~Koren.
\newblock {SGD with AdaGrad stepsizes: Full adaptivity with high probability to
  unknown parameters, unbounded gradients and affine variance}.
\newblock In \emph{International Conference on Machine Learning}, pages
  1147--1171. PMLR, 2023.

\bibitem[Beck and Teboulle(2009)]{beck2009fast}
A.~Beck and M.~Teboulle.
\newblock A fast iterative shrinkage-thresholding algorithm for linear inverse
  problems.
\newblock \emph{SIAM Journal on Imaging Sciences}, 2\penalty0 (1):\penalty0
  183--202, 2009.

\bibitem[Berkson(1944)]{berkson1944application}
J.~Berkson.
\newblock Application of the logistic function to bio-assay.
\newblock \emph{Journal of the American Statistical Association}, 39\penalty0
  (227):\penalty0 357--365, 1944.

\bibitem[Cand\`es and Recht(2012)]{candes2012exact}
E.~Cand\`es and B.~Recht.
\newblock Exact matrix completion via convex optimization.
\newblock \emph{Communications of the ACM}, 55\penalty0 (6):\penalty0 111--119,
  2012.

\bibitem[Chang and Lin(2011)]{chang2011libsvm}
C.-C. Chang and C.-J. Lin.
\newblock {LIBSVM}: A library for support vector machines.
\newblock \emph{ACM Transactions on Intelligent Systems and Technology},
  2\penalty0 (3):\penalty0 27, 2011.

\bibitem[Cortes and Vapnik(1995)]{cortes1995support}
C.~Cortes and V.~Vapnik.
\newblock Support-vector networks.
\newblock \emph{Machine Learning}, 20\penalty0 (3):\penalty0 273--297, 1995.

\bibitem[Duchi et~al.(2011)Duchi, Hazan, and Singer]{duchi2011adaptive}
J.~Duchi, E.~Hazan, and Y.~Singer.
\newblock Adaptive subgradient methods for online learning and stochastic
  optimization.
\newblock \emph{Journal of Machine Learning Research}, 12:\penalty0 2121--2159,
  2011.

\bibitem[Ene et~al.(2020)Ene, Nguyen, and Vladu]{ene2020adaptive}
A.~Ene, H.~L. Nguyen, and A.~Vladu.
\newblock Adaptive gradient methods for constrained convex optimization and
  variational inequalities.
\newblock \emph{arXiv preprint arXiv:2007.08840}, 2020.

\bibitem[Faw et~al.(2022)Faw, Tziotis, Caramanis, Mokhtari, Shakkottai, and
  Ward]{faw2022power}
M.~Faw, I.~Tziotis, C.~Caramanis, A.~Mokhtari, S.~Shakkottai, and R.~Ward.
\newblock The power of adaptivity in {SGD}: Self-tuning step sizes with
  unbounded gradients and affine variance.
\newblock In \emph{Conference on Learning Theory}, pages 313--355. PMLR, 2022.

\bibitem[Ghadimi and Lan(2013)]{ghadimi2013stochastic}
S.~Ghadimi and G.~Lan.
\newblock Stochastic first- and zeroth-order methods for nonconvex stochastic
  programming.
\newblock \emph{SIAM Journal on Optimization}, 23\penalty0 (4):\penalty0
  2341--2368, 2013.

\bibitem[Ghadimi and Lan(2016)]{ghadimi2016accelerated}
S.~Ghadimi and G.~Lan.
\newblock Accelerated gradient methods for nonconvex nonlinear and stochastic
  programming.
\newblock \emph{Mathematical Programming}, 156\penalty0 (1-2):\penalty0 59--99,
  2016.

\bibitem[Ghadimi et~al.(2016)Ghadimi, Lan, and Zhang]{ghadimi2016mini}
S.~Ghadimi, G.~Lan, and H.~Zhang.
\newblock Mini-batch stochastic approximation methods for nonconvex stochastic
  composite optimization.
\newblock \emph{Mathematical Programming}, 155\penalty0 (1-2):\penalty0
  267--305, 2016.

\bibitem[Ivgi et~al.(2023)Ivgi, Hinder, and Carmon]{ivgi2023dog}
M.~Ivgi, O.~Hinder, and Y.~Carmon.
\newblock {DoG} is {SGD}'s best friend: A parameter-free dynamic step size
  schedule.
\newblock In \emph{International Conference on Machine Learning}, pages
  14465--14499. PMLR, 2023.

\bibitem[Joulani et~al.(2020)Joulani, Raj, Gy{\"o}rgy, and
  Szepesv{\'a}ri]{joulani2020simpler}
P.~Joulani, A.~Raj, A.~Gy{\"o}rgy, and C.~Szepesv{\'a}ri.
\newblock A simpler approach to accelerated optimization: iterative averaging
  meets optimism.
\newblock In \emph{International Conference on Machine Learning}, pages
  4984--4993. PMLR, 2020.

\bibitem[Kavis et~al.(2019)Kavis, Levy, Bach, and Cevher]{kavis2019unixgrad}
A.~Kavis, K.~Levy, F.~Bach, and V.~Cevher.
\newblock {UniXGrad}: A universal, adaptive algorithm with optimal guarantees
  for constrained optimization.
\newblock In \emph{Advances in Neural Information Processing Systems},
  volume~32, 2019.

\bibitem[Kavis et~al.(2022)Kavis, Levy, and Cevher]{kavis2022high}
A.~Kavis, K.~Levy, and V.~Cevher.
\newblock High probability bounds for a class of nonconvex algorithms with
  {AdaGrad} stepsize.
\newblock In \emph{International Conference on Learning Representations}, 2022.

\bibitem[Khaled et~al.(2023)Khaled, Mishchenko, and Jin]{khaled2023dowg}
A.~Khaled, K.~Mishchenko, and C.~Jin.
\newblock {DoWG} unleashed: An efficient universal parameter-free gradient
  descent method.
\newblock In \emph{Advances in Neural Information Processing Systems},
  volume~36, 2023.

\bibitem[Kingma and Ba(2015)]{kingma2015adam}
D.~Kingma and J.~Ba.
\newblock {Adam}: A method for stochastic optimization.
\newblock In \emph{International Conference on Learning Representations}, 2015.

\bibitem[Lan(2012)]{lan2012optimal}
G.~Lan.
\newblock An optimal method for stochastic composite optimization.
\newblock \emph{Mathematical Programming}, 133\penalty0 (1-2):\penalty0
  365--397, 2012.

\bibitem[Latafat et~al.(2024)Latafat, Themelis, and
  Patrinos]{latafat2024convergence}
P.~Latafat, A.~Themelis, and P.~Patrinos.
\newblock On the convergence of adaptive first order methods: proximal gradient
  and alternating minimization algorithms.
\newblock In \emph{6th Annual Learning for Dynamics \& Control Conference},
  pages 197--208. PMLR, 2024.

\bibitem[Latafat et~al.(2025)Latafat, Themelis, Stella, and
  Patrinos]{latafat2025adaptive}
P.~Latafat, A.~Themelis, L.~Stella, and P.~Patrinos.
\newblock Adaptive proximal algorithms for convex optimization under local
  {L}ipschitz continuity of the gradient.
\newblock \emph{Mathematical Programming}, 213\penalty0 (1-2):\penalty0
  433--471, 2025.

\bibitem[LeCun et~al.(1998)LeCun, Bottou, Bengio, and
  Haffner]{lecun2002gradient}
Y.~LeCun, L.~Bottou, Y.~Bengio, and P.~Haffner.
\newblock Gradient-based learning applied to document recognition.
\newblock \emph{Proceedings of the IEEE}, 86\penalty0 (11):\penalty0
  2278--2324, 1998.

\bibitem[Levitin and Polyak(1966)]{levitin1966constrained}
E.~Levitin and B.~Polyak.
\newblock Constrained minimization methods.
\newblock \emph{USSR Computational Mathematics and Mathematical Physics},
  6\penalty0 (5):\penalty0 1--50, 1966.

\bibitem[Levy(2017)]{levy2017online}
K.~Levy.
\newblock Online to offline conversions, universality and adaptive minibatch
  sizes.
\newblock In \emph{Advances in Neural Information Processing Systems},
  volume~30, 2017.

\bibitem[Levy et~al.(2018)Levy, Yurtsever, and Cevher]{levy2018online}
K.~Levy, A.~Yurtsever, and V.~Cevher.
\newblock Online adaptive methods, universality and acceleration.
\newblock In \emph{Advances in Neural Information Processing Systems},
  volume~31, 2018.

\bibitem[Lions and Mercier(1979)]{lions1979splitting}
P.-L. Lions and B.~Mercier.
\newblock Splitting algorithms for the sum of two nonlinear operators.
\newblock \emph{SIAM Journal on Numerical Analysis}, 16\penalty0 (6):\penalty0
  964--979, 1979.

\bibitem[Malitsky and Mishchenko(2020)]{malitsky2020adaptive}
Y.~Malitsky and K.~Mishchenko.
\newblock Adaptive gradient descent without descent.
\newblock In \emph{International Conference on Machine Learning}, pages
  6702--6712. PMLR, 2020.

\bibitem[Malitsky and Mishchenko(2024)]{malitsky2024adaptive}
Y.~Malitsky and K.~Mishchenko.
\newblock Adaptive proximal gradient method for convex optimization.
\newblock In \emph{Advances in Neural Information Processing Systems},
  volume~37, 2024.

\bibitem[Nesterov(1983)]{nesterov1983method}
Y.~Nesterov.
\newblock {A method for solving the convex programming problem with convergence
  rate $O(1/k^2)$}.
\newblock \emph{Dokl. Akad. Nauk SSSR}, 269:\penalty0 543--547, 1983.

\bibitem[Nesterov(2003)]{nesterov2003introductory}
Y.~Nesterov.
\newblock \emph{Introductory Lectures on Convex Optimization: A Basic Course}.
\newblock Springer Science \& Business Media, 2003.

\bibitem[Nesterov(2013)]{nesterov2013gradient}
Y.~Nesterov.
\newblock Gradient methods for minimizing composite functions.
\newblock \emph{Mathematical Programming}, 140\penalty0 (1):\penalty0 125--161,
  2013.

\bibitem[Nesterov(2015)]{nesterov2015universal}
Y.~Nesterov.
\newblock Universal gradient methods for convex optimization problems.
\newblock \emph{Mathematical Programming}, 152\penalty0 (1-2):\penalty0
  381--404, 2015.

\bibitem[Orabona(2023)]{orabona2023normalized}
F.~Orabona.
\newblock Normalized gradients for all.
\newblock \emph{arXiv preprint arXiv:2308.05621}, 2023.

\bibitem[Parikh and Boyd(2014)]{parikh2014proximal}
N.~Parikh and S.~Boyd.
\newblock Proximal algorithms.
\newblock \emph{Foundations and Trends in Optimization}, 1\penalty0
  (3):\penalty0 127--239, 2014.

\bibitem[Passty(1979)]{passty1979ergodic}
G.~Passty.
\newblock Ergodic convergence to a zero of the sum of monotone operators in
  {Hilbert} space.
\newblock \emph{Journal of Mathematical Analysis and Applications}, 72\penalty0
  (2):\penalty0 383--390, 1979.

\bibitem[Reddi et~al.(2018)Reddi, Kale, and Kumar]{reddi2018convergence}
S.~Reddi, S.~Kale, and S.~Kumar.
\newblock On the convergence of {Adam} and beyond.
\newblock In \emph{International Conference on Learning Representations}, 2018.

\bibitem[Rodomanov et~al.(2024)Rodomanov, Jiang, and
  Stich]{rodomanov2024universality}
A.~Rodomanov, X.~Jiang, and S.~Stich.
\newblock Universality of {AdaGrad} stepsizes for stochastic optimization:
  Inexact oracle, acceleration and variance reduction.
\newblock In \emph{Advances in Neural Information Processing Systems},
  volume~37, 2024.

\bibitem[Tibshirani(1996)]{tibshirani1996regression}
R.~Tibshirani.
\newblock Regression shrinkage and selection via the lasso.
\newblock \emph{Journal of the Royal Statistical Society Series B: Statistical
  Methodology}, 58\penalty0 (1):\penalty0 267--288, 1996.

\bibitem[Tieleman and Hinton(2012)]{tieleman2012rmsprop}
T.~Tieleman and G.~Hinton.
\newblock {Lecture 6.5---RMSProp}: Divide the gradient by a running average of
  its recent magnitude.
\newblock Coursera: Neural Networks for Machine Learning, 2012.

\bibitem[Tseng(2000)]{tseng2000modified}
P.~Tseng.
\newblock A modified forward-backward splitting method for maximal monotone
  mappings.
\newblock \emph{SIAM Journal on Control and Optimization}, 38\penalty0
  (2):\penalty0 431--446, 2000.

\bibitem[Wang et~al.(2023)Wang, Zhang, Ma, and Chen]{wang2023convergence}
B.~Wang, H.~Zhang, Z.~Ma, and W.~Chen.
\newblock Convergence of {AdaGrad} for non-convex objectives: Simple proofs and
  relaxed assumptions.
\newblock In \emph{Conference on Learning Theory}, pages 161--190. PMLR, 2023.

\bibitem[Wang et~al.(2017)Wang, Ma, Goldfarb, and Liu]{wang2017stochastic}
X.~Wang, S.~Ma, D.~Goldfarb, and W.~Liu.
\newblock Stochastic quasi-{Newton} methods for nonconvex stochastic
  optimization.
\newblock \emph{SIAM Journal on Optimization}, 27\penalty0 (2):\penalty0
  927--956, 2017.

\bibitem[Ward et~al.(2020)Ward, Wu, and Bottou]{ward2020adagrad}
R.~Ward, X.~Wu, and L.~Bottou.
\newblock {AdaGrad} stepsizes: Sharp convergence over nonconvex landscapes.
\newblock \emph{Journal of Machine Learning Research}, 21:\penalty0 1--30,
  2020.

\bibitem[Yun et~al.(2021)Yun, Lozano, and Yang]{yun2021adaptive}
J.~Yun, A.~Lozano, and E.~Yang.
\newblock Adaptive proximal gradient methods for structured neural networks.
\newblock In \emph{Advances in Neural Information Processing Systems},
  volume~34, 2021.

\bibitem[Zeiler(2012)]{zeiler2012adadelta}
M.~Zeiler.
\newblock {ADADELTA}: An adaptive learning rate method.
\newblock \emph{arXiv preprint arXiv:1212.5701}, 2012.

\end{thebibliography}

\newpage

\appendix

\section{Auxiliary Lemmas}
\label{app:aux-lemmas}
Here, we introduce some auxiliary lemmas used throughout our analyses. These results are technical tools for solving certain classes of scalar inequalities or bounding sums of specific forms.

\begin{lemma}\label{lem:quadratic}
Let $a, b\geq0$. If $x^2\leq ax+b$, then $x^2\leq a^2+2b$.
\end{lemma}

\begin{lemma}\label{lem:log}
Let $a, x>0$ and $b\in\rbb$. If $x \leq a \log x+b$, then $x \leq 2 a \log 2a - 2a+2 b.$
\end{lemma}
\begin{proof}
Applying the basic inequality $\log u \leq u-1$ with $u=x/(2a)$ gives 
\[
\log x \leq \frac{x}{2a}+\log 2a -1,
\]
which together with $x \leq a \log x+b$ implies that
\[
x \leq \frac{x}{2}+a\log 2a -a +b.
\]
We obtain the desired result after rearranging the terms.
\end{proof}

\begin{lemma}\label{lem:sqrt-log}
Let $a>0$, $x\ge 1$ and $b\ge 0$. If $x \leq a \sqrt{\log x}+b$, then $x \leq 2 a \sqrt{\log (a^2+1)}+2 b+1$.
\end{lemma}
\begin{proof}
Plugging $u=x/a^2$ into $\log u \leq u-1$ gives 
\[
\log x \leq \frac{x}{a^2}+\log a^2-1 \leq \frac{x}{a^2}+\log (a^2+1),
\]
which, combined with $x \leq a \sqrt{\log x}+b$ and $\sqrt{A+B}\leq\sqrt{A}+\sqrt{B}$ for any $A, B>0$, implies
\[
x \leq a\sqrt{\frac{x}{a^2}+\log(a^2+1)}+b \leq \sqrt{x}+a \sqrt{\log (a^2+1)}+b.
\]
Solving the above inequality with Lemma \ref{lem:quadratic} gives the desired result.
\end{proof}

\begin{lemma}\label{lem:summation}
Let $\{a_k\}_{k\ge 1}$ be a nonnegative sequence and define $A_k = A_0+\sum_{i=1}^k a_i$ with $A_0>0$.
Suppose there exists $C > 0$ such that $a_k \leq C A_{k-1}$ for all $k \geq 1$. Then the following inequalities hold for all $t\ge 1$:
\begin{equation}\label{eq:sqrt-bound}
2(\sqrt{A_t} - \sqrt{A_0}) \leq \sum_{k=1}^t \frac{a_k}{\sqrt{A_{k-1}}} \leq (1+\sqrt{1+C})(\sqrt{A_t} - \sqrt{A_0}),
\end{equation}
\begin{equation}\label{eq:log-bound}
\sum_{k=1}^t \frac{a_k}{A_{k-1}} \leq (1+C)\log\frac{A_t}{A_0}.
\end{equation}
\end{lemma}

\begin{proof}
Firstly, we know from $A_0>0$ that the two summations in \eqref{eq:sqrt-bound} and \eqref{eq:log-bound} are well-defined.
Using the identity $a_k = A_k - A_{k-1} = (\sqrt{A_k} + \sqrt{A_{k-1}})(\sqrt{A_k} - \sqrt{A_{k-1}})$ and $\sqrt{A_k} \geq \sqrt{A_{k-1}}$, we obtain
\begin{equation}\label{eq:sqrt-bd-1}
\frac{a_k}{\sqrt{A_{k-1}}} = \Big(1 + \frac{\sqrt{A_k}}{\sqrt{A_{k-1}}}\Big)(\sqrt{A_k} - \sqrt{A_{k-1}}) \geq 2(\sqrt{A_k} - \sqrt{A_{k-1}}).
\end{equation}
On the other hand, under the assumption $a_k \leq C A_{k-1}$, we have
\[
\frac{\sqrt{A_k}}{\sqrt{A_{k-1}}} = \sqrt{1 + \frac{a_k}{A_{k-1}}} \leq \sqrt{1 + C},
\]
which implies that
\begin{equation}\label{eq:sqrt-bd-2}
\frac{a_k}{\sqrt{A_{k-1}}} \leq (1 + \sqrt{1+C})(\sqrt{A_k} - \sqrt{A_{k-1}}).
\end{equation}
Summing \eqref{eq:sqrt-bd-1} and \eqref{eq:sqrt-bd-2} over $k=1,\dots, t$ gives \eqref{eq:sqrt-bound}.

We proceed to prove \eqref{eq:log-bound}. Firstly, it follows from $a_k \leq C A_{k-1}$ that
\[
   A_k=A_{k-1}+a_k\leq (1+C)A_{k-1}, 
\]
which implies that
\begin{equation}\label{eq:log-bd-1}
    \frac{a_k}{A_{k-1}} \leq (1+C)\frac{a_k}{A_k}.
\end{equation}
Since the function $\log x$ is concave, we know that
\[
\log A_k-\log A_{k-1}\ge \log'(A_k)(A_k-A_{k-1})=\frac{a_k}{A_k}.
\]
Applying the above inequality to \eqref{eq:log-bd-1} and summing over $k=1,\dots, t$ leads to
\begin{align*}
    \sum_{k=1}^t \frac{a_k}{A_{k-1}}\leq& (1+C)\sum_{k=1}^t \frac{a_k}{A_k}\leq (1+C)\sum_{k=1}^t \bigl(\log A_k-\log A_{k-1}\bigr)\\
    =&(1+C)(\log A_t-\log A_0)=(1+C)\log \frac{A_{t}}{A_0}.
\end{align*}
The proof is complete.
\end{proof}

\section{Proofs for Section~\ref{sec:nc}}
\label{app:proofs}
We first introduce the following classical descent lemma for PG-type updates, which serves as a starting point for our analysis in this section. Its proof can be found in the literature (see, e.g., \citep{parikh2014proximal}). We provide a proof here for completeness.
\begin{lemma}\label{lem:F-descent}
Suppose $f$ is $L$-smooth. Let $\{x_k\}_{k\ge 1}$ be the sequence generated by the iterative scheme \eqref{Alg:prox-grad}. Then for any $k\ge 1$, it holds that
\[
F(x_{k+1}) \leq F(x_k)+\langle\nabla f(x_k)-g_k, x_{k+1}-x_k\rangle-\frac{1}{\eta_k}\|x_{k+1}-x_k\|^2+\frac{L}{2}\|x_{k+1}-x_k\|^2.
\]
\end{lemma}
\begin{proof}
From the optimality condition of \eqref{Alg:prox-grad}, we have
\[
0\in g_k + \frac{1}{\eta_k}(x_{k+1} - x_k)+\partial h(x_{k+1}),
\]
which, together with the convexity of $h$, implies that
\begin{equation}\label{eq:convex-h}
h(x) \geq h(x_{k+1}) + \Big\langle g_k + \frac{1}{\eta_k}(x_{k+1} - x_k), x_{k+1}-x\Big\rangle,~\forall x \in \mathbb{R}^d.
\end{equation}
Setting $x=x_k$ in \eqref{eq:convex-h} and rearranging the terms gives
\begin{equation}\label{eq:stoch-h-descent}
h(x_{k+1}) \leq h(x_k)-\langle g_k, x_{k+1}-x_k\rangle-\frac{1}{\eta_k}\|x_{k+1}-x_k\|^2.
\end{equation}
On the other hand, it follows from the $L$-smoothness of $f$ that
\begin{equation}\label{eq:smooth-f}
f(x_{k+1}) \leq f(x_k) + \langle \nabla f(x_k), x_{k+1} - x_k\rangle + \frac{L}{2}\|x_{k+1} - x_k\|^2.
\end{equation}
Combining \eqref{eq:stoch-h-descent} and \eqref{eq:smooth-f} leads to the desired result, which completes the proof.
\end{proof}

To prove Theorem \ref{thm:nc-deter}, we also need the following lemma, which shows that the objective values at the iterates produced by Algorithm \ref{Alg:adaprox} are uniformly upper bounded.
Recall that $\Delta_t:=F(x_t)-F^*$.
\begin{lemma}\label{lem:bd-F}
    Suppose $f$ is $L$-smooth and Assumption \ref{ass:seq-bd} holds. Let $\{x_k\}_{k\ge 1}$ be the iterates generated by Algorithm \ref{Alg:adaprox} with full gradient $g_k:=\nabla f(x_k)$ and any $\gamma, \eta>0$. Then for any $t \geq 1$, we have
    \[
        \Delta_t\leq M:= \Delta_1+ \Big[L(\eta^2+D^2)\log\frac{L\eta}{2\gamma}+\frac{LD^2}{2}\Big]_+.
    \]
\end{lemma}
\begin{proof}
Firstly, since $g_k=\nabla f(x_k)$ and $\|x_{k+1} - x_k\| = \eta_k\|\gcal_k\|$, we obtain from Lemma \ref{lem:F-descent} that
\begin{align}
F(x_{k+1}) &\leq F(x_k) - \frac{1}{\eta_k}\|x_{k+1} - x_k\|^2 + \frac{L}{2}\|x_{k+1} - x_k\|^2 \nonumber\\
& = F(x_k) - \eta_k\|\gcal_k\|^2 + \frac{L\eta_k^2}{2}\|\gcal_k\|^2\label{eq:descent-F}.
\end{align}
To proceed, summing \eqref{eq:descent-F} over $k = 1, \ldots, t$ and subtracting $F^*$ from both sides yields
\begin{equation}\label{eq:F-bd-1}
\Delta_{t+1} \leq \Delta_1 + \sum_{k=1}^t \Big(\frac{L}{2} - \frac{1}{\eta_k}\Big)\eta_k^2\|\gcal_k\|^2.
\end{equation}

We consider two cases. Firstly, if $S_1=\gamma>\frac{L\eta}{2}$, then it follows that
\[
\frac{1}{\eta_k}=\frac{S_k}{\eta}\ge \frac{S_1}{\eta}>\frac{L}{2},~\forall k\ge 1,
\]
which implies that $\Delta_{t+1} \leq \Delta_1\leq M$ for all $t\ge 1$. Otherwise, if $\gamma\leq\frac{L\eta}{2}$, then $t^* := \max\{k \in [t] : L/2 \geq 1/\eta_k\}$ is well-defined and hence $S_{t^*} \leq \frac{L\eta}{2}$.
For $k > t^*$, we know from the definition of $t^*$ that $L/2 \leq 1/\eta_k$. Therefore, it follows from \eqref{eq:F-bd-1} that
\begin{equation}\label{eq:F-bd-3}
\Delta_{t+1} \leq \Delta_1 + \sum_{k=1}^{t^*} \eta_k^2\Big(\frac{L}{2} - \frac{1}{\eta_k}\Big)\|\gcal_k\|^2\leq \Delta_1 + \frac{L}{2}\sum_{k=1}^{t^*} \eta_k^2\|\gcal_k\|^2 = \Delta_1 + \frac{L\eta^2}{2}\sum_{k=1}^{t^*} \frac{\|\gcal_k\|^2}{S_k^2}.
\end{equation}
From Assumption \ref{ass:seq-bd}, we know that
\[
\|\gcal_k\|^2 = \frac{\|x_{k+1}-x_k\|^2}{\eta_k^2}\leq \frac{D^2}{\eta_k^2} = \frac{D^2}{\eta^2}S_k^2,~\forall k\ge 1.
\]
Applying \eqref{eq:log-bound} in Lemma \ref{lem:summation} with $a_k:=\|\gcal_k\|^2$, $A_k:=S_{k+1}^2$, and $C:=D^2/\eta^2$, we derive from \eqref{eq:F-bd-3} that
\begin{align*}
\Delta_{t+1} \leq&\Delta_1 + \frac{L\eta^2}{2}\Big(\sum_{k=1}^{t^*-1} \frac{\|\gcal_k\|^2}{S_k^2}+\frac{\|\gcal_{t^*}\|^2}{S_{t^*}^2}\Big)\\
\leq& \Delta_1 + \frac{L\eta^2}{2}\Big(1+\frac{D^2}{\eta^2}\Big)\log\frac{S_{t^*}^2}{S_1^2}+\frac{LD^2}{2}\\
\leq& \Delta_1 + L(\eta^2+D^2)\log\frac{L\eta}{2\gamma}+\frac{LD^2}{2},
\end{align*}
where the last inequality follows from $S_{t^*} \leq \frac{L\eta}{2}$ and $S_1=\gamma$.
The proof is thus complete.
\end{proof}

In the stochastic setting, for notational convenience, we denote the stochastic noises as 
\[
\delta_{k, j}:=\nabla f(x_k)-\nabla\ell(x_k;\xi_{k_j}), \quad\delta_k:=\frac{1}{b_k}\sum_{j=1}^{b_k}\delta_{k, j} = \nabla f(x_k)-g_k.
\]
Then it follows from the independence among the batch samples $\{\xi_{k_j}\}_{j=1}^{b_k}$ and Assumption \ref{ass:bd-stoch} that 
\[
\ebb_k\bigg[\bigg\|\sum_{j=1}^{b_k}\delta_{k,j}\bigg\|^2\bigg]=\sum_{i,j=1}^{b_k}\ebb_k\big[\langle \delta_{k, i}, \delta_{k, j}\rangle\big] = \sum_{j=1}^{b_k}\ebb_k\big[\|\delta_{k, j}\|^2\big]\leq \sigma^2 b_k,
\]
where to obtain the second equality, we used $\ebb_k\big[\langle \delta_{k, i}, \delta_{k, j}\rangle\big]=\big\langle \ebb_k\big[\delta_{k, i}\big], \ebb_k\big[\delta_{k, j}\big]\big\rangle=0$ for $i\neq j$ due to the independence between $\delta_{k, i}$ and $\delta_{k, j}$.
This further implies that
\begin{equation}\label{eq:var-bd}
    \ebb_k\big[\|\delta_k\|^2\big] = \ebb_k\bigg[\bigg\|\frac{1}{b_k}\sum_{j=1}^{b_k}\delta_{k,j}\bigg\|^2\bigg] =\frac{1}{b_k^2} \ebb_k\bigg[\bigg\|\sum_{j=1}^{b_k}\delta_{k,j}\bigg\|^2\bigg]\leq \frac{\sigma^2}{b_k}.
\end{equation}

\subsection{Proof of Theorem~\ref{thm:nc-deter}}
\label{app:thm-ncdeter}

As shown in Appendix~\ref{app:bd-discussion}, the iterates of Algorithm~\ref{Alg:adaprox} satisfy Assumption~\ref{ass:seq-bd}(a) with a constant $D$ that depends only on $\Delta_1$, $L$, $\gamma$, and $\eta$. Therefore, we use this $D$ below and apply Lemma~\ref{lem:bd-F}.
We start by rearranging \eqref{eq:descent-F} to obtain
\begin{align}
\|\gcal_k\|^2 \leq \frac{1}{\eta_k}(F(x_k) - F(x_{k+1})) + \frac{L}{2}\eta_k\|\gcal_k\|^2= \frac{1}{\eta_k}(\Delta_k - \Delta_{k+1}) + \frac{L}{2}\eta_k\|\gcal_k\|^2\label{eq:sum-G2-1}.
\end{align}
Summing \eqref{eq:sum-G2-1} over $k = 1, \ldots, t$ leads to
\begin{align}
\sum_{k=1}^t \|\gcal_k\|^2 &\leq \frac{\Delta_1}{\eta_1} + \sum_{k=2}^t\Big(\frac{1}{\eta_k} - \frac{1}{\eta_{k-1}}\Big)\Delta_k - \frac{\Delta_{t+1}}{\eta_t} + \frac{L}{2}\sum_{k=1}^t \eta_k\|\gcal_k\|^2 \nonumber\\
&\leq \frac{M}{\eta_1} + M\sum_{k=2}^t\Big(\frac{1}{\eta_k} - \frac{1}{\eta_{k-1}}\Big) + \frac{L}{2}\sum_{k=1}^t \eta_k\|\gcal_k\|^2 \nonumber\\
&= \frac{M}{\eta_t} + \frac{L}{2}\sum_{k=1}^t \eta_k\|\gcal_k\|^2 = \frac{M}{\eta}S_t + \frac{L\eta}{2}\sum_{k=1}^t \frac{\|\gcal_k\|^2}{S_k}\label{eq:sum-G2-2},
\end{align}
where we used $\Delta_k\leq M$ and $\frac{1}{\eta_k} - \frac{1}{\eta_{k-1}}\ge 0$ since $\eta_k$ is non-increasing.
Applying \eqref{eq:sqrt-bound} in Lemma \ref{lem:summation} with $a_k:=\|\gcal_k\|^2$, $A_k:=S_{k+1}^2$, and $C:=D^2/\eta^2$, we further derive from \eqref{eq:sum-G2-2} that
\begin{align}
\sum_{k=1}^t \|\gcal_k\|^2 &\leq \frac{M}{\eta} \sqrt{\gamma^2+\sum_{k=1}^t\|\gcal_k\|^2} + \frac{L\eta}{2}\Big(1+\sqrt{1+D^2/\eta^2}\Big)\Bigg(\sqrt{\gamma^2+\sum_{k=1}^t\|\gcal_k\|^2} -\gamma\Bigg)\nonumber\\
&\leq \frac{M\gamma}{\eta} + \frac{M}{\eta}\sqrt{\sum_{k=1}^t\|\gcal_k\|^2}+\frac{L\eta}{2}\big(2+D/\eta\big)\sqrt{\sum_{k=1}^t\|\gcal_k\|^2} \nonumber\\
&= \frac{M\gamma}{\eta} + \Big(\frac{M}{\eta} + L\eta + \frac{LD}{2}\Big)\sqrt{\sum_{k=1}^t\|\gcal_k\|^2}\label{eq:sum-G2-3},
\end{align}
where we used the basic inequality $\sqrt{a+b} \leq \sqrt{a} + \sqrt{b}$ to obtain the second inequality. Solving \eqref{eq:sum-G2-3} with Lemma \ref{lem:quadratic} gives
\[
  \sum_{k=1}^t \|\gcal_k\|^2 \leq \Big(\frac{M}{\eta} + L\eta + \frac{LD}{2}\Big)^2 + \frac{2M\gamma}{\eta},
\]
which completes the proof after dividing by $t$.

\subsection{Proof of Lemma~\ref{lem:stoc-ip-sum}}
\label{app:lem-stocipsum}

Firstly, note from the Cauchy--Schwarz inequality and the definition of $\widetilde{\gcal}_k$ \eqref{def:tilde-Gt} that
\[
\sum_{k=1}^t\langle\delta_k, x_{k+1}-x_k\rangle\leq\sum_{k=1}^t\|\delta_k\| \|x_{k+1}-x_k\|=\eta\sum_{k=1}^t\|\delta_k\| \frac{\|\widetilde{\gcal}_k\|}{S_k}.
\]
Taking expectations on both sides of the above inequality and applying the Cauchy--Schwarz inequality for expectations $|\mathbb{E}[XY]| \leq \sqrt{\mathbb{E}[X^2] \mathbb{E}[Y^2]}$, we further obtain
\begin{equation}\label{eq:stoch-error-1}
\begin{aligned}
\ebb\Big[\sum_{k=1}^t\langle\delta_k, x_{k+1}-x_k\rangle\Big]\leq & \eta\sum_{k=1}^t\ebb\Big[\|\delta_k\| \frac{\|\widetilde{\gcal}_k\|}{S_k}\Big]\leq \eta\sum_{k=1}^t \sqrt{\ebb[\|\delta_k\|^2] \ebb\Big[\frac{\|\widetilde{\gcal}_k\|^2}{S_k^2}\Big]}\\
\leq& \sigma \eta\sum_{k=1}^t \frac{1}{\sqrt{b_k}}\sqrt{\ebb\Big[\frac{\|\widetilde{\gcal}_k\|^2}{S_k^2}\Big]}\leq \sigma \eta\sqrt{\sum_{k=1}^t \frac{1}{b_k}} \sqrt{\ebb\Big[\sum_{k=1}^t \frac{\|\widetilde{\gcal}_k\|^2}{S_k^2}\Big]},
\end{aligned}
\end{equation}
where we used \eqref{eq:var-bd} to obtain the third inequality, and the last one follows from the Cauchy--Schwarz inequality. Applying \eqref{eq:log-bound} in Lemma \ref{lem:summation} with $a_k:=\|\widetilde{\gcal}_k\|^2$, $A_k:=S_{k+1}^2$, and $C:=D^2/\eta^2$ to bound $\sum_{k=1}^t \frac{\|\widetilde{\gcal}_k\|^2}{S_k^2}$, we derive from \eqref{eq:stoch-error-1} that
\begin{align*}
     \ebb\Big[\sum_{k=1}^t\langle\delta_k, x_{k+1}-x_k\rangle\Big] \leq & \sigma\eta \sqrt{\sum_{k=1}^t \frac{1}{b_k}} \sqrt{\Big(1+\frac{D^2}{\eta^2}\Big) \ebb\Big[\log \frac{S_{t+1}^2}{\gamma^2}\Big]}\\
     \leq & \sigma \sqrt{2(\eta^2+D^2) \sum_{k=1}^t \frac{1}{b_k}} \sqrt{\log \ebb\bigg[\frac{S_{t+1}}{\gamma}\bigg]},
\end{align*}
where the last inequality follows from Jensen's inequality applied to the concave function $\log(\cdot)$. The proof is now complete.

\subsection{Proof of Theorem~\ref{thm:nc-stoc}}
\label{app:thm-ncstoc}

We start by applying Lemma \ref{lem:F-descent} and $\|x_{k+1}-x_k\|=\eta_k\|\widetilde{\gcal}_k\|$ to obtain
\begin{equation}\label{eq:stoch-F-descent}
F(x_{k+1}) \leq F(x_k)+\langle \delta_k, x_{k+1}-x_k\rangle-\eta_k\|\widetilde{\gcal}_k\|^2+\frac{L\eta_k^2}{2}\|\widetilde{\gcal}_k\|^2.
\end{equation}
Subtracting $F^*$ from both sides of \eqref{eq:stoch-F-descent} and rearranging the terms, we have
\begin{equation}\label{eq:stoch-Gk-bound}
\eta_k\|\widetilde{\gcal}_k\|^2 \leq \Delta_k-\Delta_{k+1}+\frac{L\eta_k^2}{2}\|\widetilde{\gcal}_k\|^2+\langle\delta_k, x_{k+1}-x_k\rangle.
\end{equation}
Summing \eqref{eq:stoch-Gk-bound} over $k=1,\ldots,t$ yields
\begin{align}
\eta \sum_{k=1}^t \frac{\|\widetilde{\gcal}_k\|^2}{S_k} \leq& \Delta_1-\Delta_{t+1}+\frac{L \eta^2}{2} \sum_{k=1}^t \frac{\|\widetilde{\gcal}_k\|^2}{S_k^2}+ \sum_{k=1}^t\langle\delta_k, x_{k+1}-x_k\rangle\notag\\
\leq & \Delta_1+\frac{L}{2}(\eta^2+ D^2) \log \frac{S_{t+1}^2}{\gamma^2}+\sum_{k=1}^t\langle\delta_k, x_{k+1}-x_k\rangle\label{eq:stoch-sum-Gk},
\end{align}
where the second inequality follows by applying \eqref{eq:log-bound} in Lemma \ref{lem:summation} with $a_k:=\|\widetilde{\gcal}_k\|^2$, $A_k:=S_{k+1}^2$, and $C:=D^2/\eta^2$.

Next, we bound $\ebb[S_{t+1}]$. Applying \eqref{eq:sqrt-bound} in Lemma \ref{lem:summation} to the LHS of \eqref{eq:stoch-sum-Gk} gives
\begin{equation}\label{eq:stoch-S-bound}
2 \eta (S_{t+1}-\gamma) \leq \eta \sum_{k=1}^t \frac{\|\widetilde{\gcal}_k\|^2}{S_k}\leq \Delta_1+L(\eta^2+D^2) \log \frac{S_{t+1}}{\gamma}+ \sum_{k=1}^t\langle\delta_k, x_{k+1}-x_k\rangle.
\end{equation}
For brevity, let $\kappa_t:=\frac{\sigma}{4\gamma\eta}\sqrt{2(\eta^2+D^2)\sum_{k=1}^t 1/b_k}$.
Taking expectation on both sides of \eqref{eq:stoch-S-bound} and applying Lemma \ref{lem:stoc-ip-sum}, we derive after rearranging the terms that
\begin{align}
\ebb\bigg[\frac{S_{t+1}}{\gamma}\bigg] &\leq 1+\frac{\Delta_1}{2 \gamma \eta}+\frac{L}{2 \gamma \eta}(\eta^2+D^2) \ebb\bigg[\log \frac{S_{t+1}}{\gamma}\bigg]+2\kappa_t\sqrt{\log \ebb\bigg[\frac{S_{t+1}}{\gamma}\bigg]}\notag\\
& \leq 1+\frac{\Delta_1}{2 \gamma \eta}+\frac{L}{2 \gamma \eta}(\eta^2+D^2) \log \ebb\bigg[\frac{S_{t+1}}{\gamma}\bigg]+\kappa_t\bigg(1+\log \ebb\bigg[\frac{S_{t+1}}{\gamma}\bigg]\bigg)\notag\\
& = 1+\frac{\Delta_1}{2 \gamma \eta}+\kappa_t+\bigg(\frac{L}{2 \gamma \eta}(\eta^2+D^2)+\kappa_t\bigg) \log \ebb\bigg[\frac{S_{t+1}}{\gamma}\bigg],\label{eq:stoch-S-bound-2}
\end{align}
where the second inequality follows from the concavity of $\log(\cdot)$ and the basic inequality $\sqrt{x}\leq \frac{1}{2}(1+ x)$.
Solving \eqref{eq:stoch-S-bound-2} via Lemma \ref{lem:log}, we conclude that
\begin{equation*}
\ebb\bigg[\frac{S_{t+1}}{\gamma}\bigg] \leq \bigg(\frac{L}{\gamma \eta}(\eta^2+D^2)+2\kappa_t\bigg)\log \bigg(\frac{L}{\gamma \eta}(\eta^2+D^2)+2\kappa_t\bigg)+2+\frac{\Delta_1}{\gamma \eta}-\frac{L}{\gamma \eta}(\eta^2+D^2),
\end{equation*}
which leads to the following after omitting the constants $\{\gamma, \eta, D, L, \Delta_1\}$:
\begin{equation}\label{eq:stoch-S-final}
    \ebb[S_{t+1}]\lesssim 1+\Bigg(1+\sigma\sqrt{\sum_{k=1}^t\frac{1}{b_k}}\Bigg) \log \Bigg(1+\sigma\sqrt{\sum_{k=1}^t\frac{1}{b_k}}\Bigg).
\end{equation}

From the nonexpansiveness of the proximal operator \eqref{eq:prox-nonexpansive}, we have
\begin{equation}\label{eq:stoch-G-tilde-G}
\|\gcal_k-\widetilde{\gcal}_k\|^2=\frac{1}{\eta_k^2}\big\|\p_{\eta_k h}(x_k-\eta_k g_k)-\p_{\eta_k h}\big(x_k-\eta_k \nabla f(x_k)\big)\big\|^2 \leq \|g_k-\nabla f(x_k)\|^2=\|\delta_k\|^2.
\end{equation}
Therefore, we obtain from the basic inequalities $\|u+v\|^2\leq 2\|u\|^2 + 2\|v\|^2$ and $\sqrt{a+b}\leq \sqrt{a}+\sqrt{b}$, together with \eqref{eq:stoch-G-tilde-G}, that
\begin{equation}\label{eq:nc-stoc-1}
\begin{aligned}
\ebb\Bigg[\sqrt{\sum_{k=1}^t\|\gcal_k\|^2}\Bigg]\leq &\ebb\Bigg[\sqrt{2 \sum_{k=1}^t\|\widetilde{\gcal}_k\|^2+2 \sum_{k=1}^t\|\gcal_k-\widetilde{\gcal}_k\|^2}\Bigg]\\
\leq & \sqrt{2} \ebb\Bigg[\sqrt{\sum_{k=1}^t\|\widetilde{\gcal}_k\|^2}\Bigg]+\ebb\Bigg[\sqrt{2 \sum_{k=1}^t\|\delta_k\|^2}\Bigg] \\
\leq & \sqrt{2} \ebb[S_{t+1}]+\sqrt{2 \sum_{k=1}^t \ebb[\|\delta_k\|^2]} \leq \sqrt{2} \ebb[S_{t+1}]+\sigma\sqrt{2\sum_{k=1}^{t}\frac{1}{b_k}},
\end{aligned}
\end{equation}
where we used $S_k^2=\gamma^2+\sum_{j=1}^{k-1}\|\widetilde{\gcal}_j\|^2$ and the concavity of $\sqrt{\cdot}$ to derive the third inequality, and the last one follows from \eqref{eq:var-bd}. 

Squaring both sides of \eqref{eq:nc-stoc-1} and using $(a+b)^2\le 2a^2+2b^2$ together with \eqref{eq:stoch-S-final}, we obtain
\begin{align*}
\Bigg(\ebb\Bigg[\sqrt{\sum_{k=1}^t\|\gcal_k\|^2}\,\Bigg]\Bigg)^2&\le 4\big(\ebb[S_{t+1}]\big)^2+4\sigma^2\sum_{k=1}^t\frac{1}{b_k}\\
&\lesssim \Bigg(1+\sigma^2\sum_{k=1}^t\frac{1}{b_k}\Bigg) \Bigg(1+\log^2 \Bigg(1+\sigma\sqrt{\sum_{k=1}^t\frac{1}{b_k}}\Bigg)\Bigg).
\end{align*}
Then Theorem~\ref{thm:nc-stoc} follows by dividing both sides by $t$.

Finally, we establish the high-probability bound in Remark~\ref{rmk:high-prob} through Markov's inequality $\mathbb{P}(X \geq a) \leq \ebb[X]/a$ for any nonnegative random variable $X$ and $a>0$, which is equivalent to
\[
\mathbb{P}\big(X^2 \leq a^2\big) \geq 1-\ebb[X]/a.
\]
Given $\epsilon\in (0, 1)$, setting $a:=\ebb[X]/\epsilon$ in the above inequality gives 
\begin{equation}\label{eq:markov}
\mathbb{P}\big(X^2 \leq (\ebb[X])^2/\epsilon^2\big)\ge 1-\epsilon.
\end{equation}
Letting $X:=\sqrt{\sum_{k=1}^t\|\gcal_k\|^2}$ in \eqref{eq:markov}, we obtain from \eqref{eq:nc-stoc-1} that, with probability at least $1- \epsilon$,
\begin{align*}
\sum_{k=1}^t\|\gcal_k\|^2 \leq&\frac{\Big(\ebb\Big[\sqrt{\sum_{k=1}^t\|\gcal_k\|^2}\Big]\Big)^2}{\epsilon^2}\leq \frac{4\big(\ebb\big[S_{t+1}\big]\big)^2+4 \sigma^2 \sum_{k=1}^t \frac{1}{b_k}}{\epsilon^2}\\
\lesssim& \frac{1}{\epsilon^2}\Bigg[1+\log^2 \Bigg(1+\sigma\sqrt{\sum_{k=1}^t\frac{1}{b_k}}\Bigg)+\sigma^2\sum_{k=1}^t\frac{1}{b_k} \Bigg(1+\log^2 \Bigg(1+\sigma\sqrt{\sum_{k=1}^t\frac{1}{b_k}}\Bigg)\Bigg)\Bigg],
\end{align*}
where we used \eqref{eq:stoch-S-final} and the basic inequality $(a+b)^2\leq 2a^2+2b^2$ to obtain the last inequality.

\section{Proofs for Section~\ref{sec:cvx}}
\label{app:proofs-cvx}

The following lemma provides descent properties for PG-type algorithms under the convexity of $f$, which serve as starting points for our analysis in this section. Recall that $\Delta_k:=F(x_k)-F^*$.

\begin{lemma}\label{lem:F-descent-cvx}
Suppose $f$ in \eqref{main-prob} is convex. Let $\{x_k\}_{k\ge 1}$ be the sequence generated by the iterative scheme \eqref{Alg:prox-grad}. Then for any $k\ge 1$, we have the following results:
\begin{itemize}[leftmargin=2em]
    \item[(i).] If $f$ is $G$-Lipschitz continuous, then
    \begin{equation}\label{eq:descent-lip}
    \begin{aligned}
    \Delta_{k+1} \leq & \frac{1}{2 \eta_k}\big(\|x_k-x^*\|^2-\|x_{k+1}-x^*\|^2\big)-\frac{1}{2 \eta_k}\|x_{k+1}-x_k\|^2\\
    &+2G \|x_k-x_{k+1}\| +\langle\delta_k, x_k-x^*\rangle+\langle\delta_k, x_{k+1}-x_k\rangle.
    \end{aligned}
    \end{equation}
    \item[(ii).] If $f$ is $L$-smooth, then
    \begin{equation}\label{eq:descent-smooth}
    \begin{aligned}
    \Delta_{k+1} \leq & \frac{1}{2 \eta_k}\big(\|x_k-x^*\|^2-\|x_{k+1}-x^*\|^2\big)+\frac{1}{2}\Big(L-\frac{1}{\eta_k}\Big)\|x_{k+1}-x_k\|^2\\
    & +\langle\delta_k, x_k-x^*\rangle+\langle\delta_k, x_{k+1}-x_k\rangle.
    \end{aligned}
    \end{equation}
\end{itemize}
\end{lemma}

\begin{proof}
We start by using the convexity of $f$ to obtain
\begin{equation}\label{eq:convex-descent-1}
\begin{aligned}
\Delta_k &= f(x_k)-f(x^*)+h(x_k)-h(x^*)\\
&\leq\langle \nabla f(x_k), x_k-x^*\rangle+h(x_k)-h(x^*) \\
& =\langle\nabla f(x_k), x_k-x_{k+1}\rangle+\langle g_k, x_{k+1}-x^*\rangle+\langle\delta_k, x_{k+1}-x^*\rangle+h(x_k)-h(x^*).
\end{aligned}
\end{equation}

We first prove the Lipschitz continuous case (i). To this end, we bound the first term on the RHS of \eqref{eq:convex-descent-1} via the convexity of $f$ and $\|\nabla f(x_{k+1})\|\leq G$ as follows 
\begin{equation}\label{eq:stoch-fk-convex}
\begin{aligned}
\langle\nabla f(x_k), x_k-x_{k+1}\rangle & =\langle \nabla f(x_{k+1}), x_k-x_{k+1}\rangle+\langle\nabla f(x_k)-\nabla f(x_{k+1}), x_k-x_{k+1}\rangle \\
& \leq f(x_k)-f(x_{k+1})+\|\nabla f(x_k)-\nabla f(x_{k+1})\| \|x_k-x_{k+1}\|\\
& \leq f(x_k)-f(x_{k+1})+2G \|x_k-x_{k+1}\|.
\end{aligned}
\end{equation}
For the second term in \eqref{eq:convex-descent-1}, plugging $x=x^*$ into \eqref{eq:convex-h} and rearranging the terms yields
\begin{equation}\label{eq:convex-trick}
\begin{aligned}
&\langle g_k, x_{k+1}-x^*\rangle \leq \frac{1}{\eta_k}\langle x_k-x_{k+1}, x_{k+1}-x^*\rangle+h(x^*)-h(x_{k+1}) \\
=&\frac{1}{2\eta_k}\big(\|x_k-x^*\|^2-\|x_{k+1}-x^*\|^2-\|x_{k+1}-x_k\|^2\big)+h(x^*)-h(x_{k+1}),
\end{aligned}
\end{equation}
where we used the identity $2\langle a-b, b-c\rangle = \|a-c\|^2-\|a-b\|^2-\|b-c\|^2$ for the equality.
To bound the third term in \eqref{eq:convex-descent-1}, we decompose
\begin{equation}\label{eq:stoch-delta-decompose}
\langle\delta_k, x_{k+1}-x^*\rangle =\langle\delta_k, x_k-x^*\rangle+\langle\delta_k, x_{k+1}-x_k\rangle.
\end{equation}
Applying \eqref{eq:stoch-fk-convex}, \eqref{eq:convex-trick}, and \eqref{eq:stoch-delta-decompose} to \eqref{eq:convex-descent-1}, we obtain
\begin{equation*}
\begin{aligned}
\Delta_{k} \leq & \frac{1}{2 \eta_k}\big(\|x_k-x^*\|^2-\|x_{k+1}-x^*\|^2\big)-\frac{1}{2 \eta_k}\|x_{k+1}-x_k\|^2+2G \|x_k-x_{k+1}\|\\
& +\langle\delta_k, x_k-x^*\rangle+\langle\delta_k, x_{k+1}-x_k\rangle+\Delta_{k} -\Delta_{k+1} ,
\end{aligned}
\end{equation*}
which simplifies to \eqref{eq:descent-lip}.

We proceed to prove (ii) for $L$-smooth $f$. In this case, we can obtain a tighter bound for the first term on the RHS of \eqref{eq:convex-descent-1} using
\begin{equation}\label{eq:descent-cvx-term1}
\langle\nabla f(x_k), x_k-x_{k+1}\rangle \leq f(x_k)-f(x_{k+1})+ \frac{L}{2}\|x_k-x_{k+1}\|^2.
\end{equation}
Then \eqref{eq:descent-smooth} follows after plugging \eqref{eq:descent-cvx-term1}, \eqref{eq:convex-trick}, and \eqref{eq:stoch-delta-decompose} into \eqref{eq:convex-descent-1}. The proof is now complete.
\end{proof}

In both theorem proofs below, Assumption~\ref{ass:seq-bd}(b) implies $\|x_{k+1}-x_k\|\leq\|x_{k+1}-x^*\|+\|x_k-x^*\|\leq2D$, so Lemma~\ref{lem:stoc-ip-sum} applies with $2D$ in place of $D$.

\subsection{Proof of Theorem~\ref{thm:stoch-convex-nonsmooth}}
\label{app:thm-cvxns}

We first establish a bound for $\ebb[S_{t+1}]$. To this end, note from the $G$-Lipschitz continuity of $f$ that
\begin{align*}
    f(x_{k+1}) &\leq f(x_k)+G\|x_{k+1} - x_k\|,\\
    -\langle \nabla f(x_k), x_{k+1} - x_k\rangle &\leq \|\nabla f(x_k)\|\|x_{k+1} - x_k\|\leq G\|x_{k+1} - x_k\|.
\end{align*}
Summing up the above two inequalities gives 
\begin{equation}\label{eq:Lip-f}
f(x_{k+1}) \leq f(x_k) + \langle \nabla f(x_k), x_{k+1} - x_k\rangle + 2G\|x_{k+1} - x_k\|.
\end{equation}
Combining \eqref{eq:Lip-f} and \eqref{eq:convex-h} with $x = x_k$, we obtain
\begin{equation}\label{eq:descent-F-Lip}
\begin{aligned}
&F(x_{k+1}) = f(x_{k+1}) + h(x_{k+1})\\
\leq & f(x_k)+h(x_k) -\frac{1}{\eta_k}\|x_{k+1} - x_k\|^2 + 2G\|x_{k+1} - x_k\| + \langle\delta_k, x_{k+1}-x_k\rangle\\
=& F(x_k) - \eta_k\|\widetilde{\gcal}_k\|^2 + 2G\eta_k\|\widetilde{\gcal}_k\|+\langle\delta_k, x_{k+1}-x_k\rangle,
\end{aligned}
\end{equation}
where we used $\|x_{k+1} - x_k\| = \eta_k\|\widetilde{\gcal}_k\|$ to obtain the last equality. Rearranging the terms in \eqref{eq:descent-F-Lip} and noting $\eta_k=\eta/S_k$ gives
\begin{equation}\label{eq:stoch-F-Lip-2}
\eta\frac{\|\widetilde{\gcal}_k\|^2}{S_k} \leq F(x_k) - F(x_{k+1}) + 2G\eta\frac{\|\widetilde{\gcal}_k\|}{S_k} + \langle\delta_k, x_{k+1}-x_k\rangle.
\end{equation}
Summing \eqref{eq:stoch-F-Lip-2} over $k=1,\ldots,t$ and applying \eqref{eq:sqrt-bound} in Lemma \ref{lem:summation} yields
\begin{equation}\label{eq:stoch-sum-S-ns}
\begin{aligned}
2\eta(S_{t+1}-\gamma) & \leq \eta \sum_{k=1}^t \frac{\|\widetilde{\gcal}_k\|^2}{S_k}\\ 
& \leq F(x_1)-F(x_{t+1})+2 G\eta \sum_{k=1}^t \frac{\|\widetilde{\gcal}_k\|}{S_k}+ \sum_{k=1}^t\langle\delta_k, x_{k+1}-x_k\rangle \\
& \leq \Delta_1+2 G\eta \sqrt{t} \sqrt{\sum_{k=1}^t \frac{\|\widetilde{\gcal}_k\|^2}{S_k^2}}+\sum_{k=1}^t\langle\delta_k, x_{k+1}-x_k\rangle,
\end{aligned}
\end{equation}
where the last inequality follows from $F(x_{t+1})\ge F^*$ and the Cauchy--Schwarz inequality.
Taking expectation on both sides of \eqref{eq:stoch-sum-S-ns} and using the concavity of $\sqrt{\cdot}$, we obtain
\begin{equation}\label{eq:stoch-E-S-ns}
\begin{aligned}
\ebb\bigg[\frac{S_{t+1}}{\gamma}\bigg] & \leq 1+\frac{\Delta_1}{2 \gamma\eta}+ \frac{G \sqrt{t}}{\gamma} \ebb\Bigg[\sqrt{\sum_{k=1}^t \frac{\|\widetilde{\gcal}_k\|^2}{S_k^2}}\Bigg]+\frac{1}{2\gamma\eta}\ebb\Big[\sum_{k=1}^t\langle\delta_k, x_{k+1}-x_k\rangle\Big] \\
& \leq 1+\frac{\Delta_1}{2 \gamma\eta}+ \frac{G \sqrt{t}}{\gamma} \sqrt{\ebb\bigg[\sum_{k=1}^t \frac{\|\widetilde{\gcal}_k\|^2}{S_k^2}\bigg]}+\frac{1}{2\gamma\eta}\ebb\Big[\sum_{k=1}^t\langle\delta_k, x_{k+1}-x_k\rangle\Big].
\end{aligned}
\end{equation}
Applying \eqref{eq:log-bound} in Lemma \ref{lem:summation} and Lemma \ref{lem:stoc-ip-sum} with $b_k=1$ to \eqref{eq:stoch-E-S-ns} yields
\begin{equation*}
\begin{aligned}
\ebb\bigg[\frac{S_{t+1}}{\gamma}\bigg] & \leq 1+\frac{\Delta_1}{2 \gamma\eta}+\frac{G\sqrt{t}}{\gamma} \sqrt{\Big(1+\frac{4D^2}{\eta^2}\Big) \ebb\bigg[\log \frac{S_{t+1}^2}{\gamma^2}\bigg]} +\frac{\sigma}{2\gamma\eta}\!\sqrt{2(\eta^2+4D^2) t} \sqrt{\log \ebb\bigg[\frac{S_{t+1}}{\gamma}\bigg]}\\
& \leq 1+\frac{\Delta_1}{2 \gamma\eta}+\frac{2G+\sigma}{2\gamma\eta} \sqrt{2(\eta^2+4D^2)t} \sqrt{\log \ebb\bigg[\frac{S_{t+1}}{\gamma}\bigg]},
\end{aligned}
\end{equation*}
where the last inequality follows from the concavity of $\log(\cdot)$. 
Solving the above inequality via Lemma \ref{lem:sqrt-log}, we obtain
\begin{equation*}
\ebb\bigg[\frac{S_{t+1}}{\gamma}\bigg] \leq \frac{2 G+\sigma}{\gamma\eta}\sqrt{2(\eta^2+4D^2)t} \sqrt{\log \bigg(\frac{(2 G+\sigma)^2(\eta^2+4D^2) t}{2\gamma^2\eta^2}+1\bigg)}+3 +\frac{\Delta_1}{\gamma\eta},
\end{equation*}
which leads to the following after omitting the constants $\{\gamma, \eta, D, G, \Delta_1\}$:
\begin{equation}\label{eq:stoc-S-lip}
    \ebb[S_{t+1}]\lesssim 1+(1+\sigma)\sqrt{t\log \big(1+(1+\sigma^2)t\big)}\lesssim (1+\sigma)\sqrt{t\log \big(1+(1+\sigma^2)t\big)}.
\end{equation}

Next, we estimate $\ebb[\sum_{k=1}^t \Delta_{k+1}]$ by using the bound on $\ebb[S_{t+1}]$. Firstly, by omitting the negative term in \eqref{eq:descent-lip}, we obtain
\begin{equation}\label{eq:descent-lip-1}
\begin{aligned}
\Delta_{k+1} \leq & \frac{1}{2 \eta_k}\big(\|x_k-x^*\|^2-\|x_{k+1}-x^*\|^2\big)+2G\eta \frac{\|\widetilde{\gcal}_k\|}{S_k}\\
&+\langle\delta_k, x_k-x^*\rangle+\langle\delta_k, x_{k+1}-x_k\rangle.
\end{aligned}
\end{equation}
Summing \eqref{eq:descent-lip-1} over $k=1,\ldots,t$ and noting that
\begin{equation}\label{eq:bound-x-diff}
    \begin{aligned}
    &\sum_{k=1}^t \frac{1}{2 \eta_k}\big(\|x_k-x^*\|^2-\|x_{k+1}-x^*\|^2\big) \\
    =& \frac{\|x_1-x^*\|^2}{2\eta_1}+\sum_{k=2}^t\Big(\frac{1}{2\eta_k} - \frac{1}{2\eta_{k-1}}\Big)\|x_k-x^*\|^2 - \frac{\|x_{t+1}-x^*\|^2}{2\eta_t}\\
    \leq & \frac{D^2}{2\eta_1} + D^2\sum_{k=2}^t\Big(\frac{1}{2\eta_k} - \frac{1}{2\eta_{k-1}}\Big)=\frac{D^2}{2\eta_t},
    \end{aligned}
\end{equation}
we obtain
\begin{equation}\label{eq:stoch-sum-Delta-ns}
\sum_{k=1}^t\Delta_{k+1} \leq \frac{D^2}{2 \eta} S_t+2 G \eta \sum_{k=1}^t \frac{\|\widetilde{\gcal}_k\|}{S_k}+\sum_{k=1}^t\langle\delta_k, x_k-x^*\rangle+\sum_{k=1}^t\langle\delta_k, x_{k+1}-x_k\rangle.
\end{equation}
Taking expectation on both sides of \eqref{eq:stoch-sum-Delta-ns} and noting that $\ebb_k[\langle\delta_k, x_k-x^*\rangle]=0$ for all $k\ge 1$ due to the unbiasedness of $g_k$, we have
\begin{equation}\label{eq:stoch-E-sum-Delta-ns}
\begin{aligned}
\ebb\Big[\sum_{k=1}^t \Delta_{k+1}\Big] & \leq \frac{D^2}{2 \eta} \ebb[S_t]+2 G \eta \ebb\bigg[\sum_{k=1}^t \frac{\|\widetilde{\gcal}_k\|}{S_k}\bigg]+\ebb\Big[\sum_{k=1}^t\langle\delta_k, x_{k+1}-x_k\rangle\Big] \\
& \leq \frac{D^2}{2 \eta} \ebb[S_{t+1}]+2 G \eta \sqrt{t} \sqrt{\ebb\bigg[\sum_{k=1}^t \frac{\|\widetilde{\gcal}_k\|^2}{S_k^2}}\bigg]+\ebb\Big[\sum_{k=1}^t\langle\delta_k, x_{k+1}-x_k\rangle\Big] \\
& \leq \frac{D^2}{2 \eta} \ebb[S_{t+1}]+2 G\eta \sqrt{t} \sqrt{\Big(1+\frac{4D^2}{\eta^2}\Big) \ebb\bigg[\log \frac{S_{t+1}^2}{\gamma^2}\bigg]}\\
&\quad + \sigma  \sqrt{2(\eta^2+4D^2) t} \sqrt{\log \ebb\bigg[\frac{S_{t+1}}{\gamma}\bigg]}\\
& \leq \frac{D^2}{2 \eta} \ebb[S_{t+1}]+(2 G+\sigma)\sqrt{2(\eta^2+4D^2)t} \sqrt{\log \ebb\bigg[\frac{S_{t+1}}{\gamma}\bigg]},
\end{aligned}
\end{equation}
where the second inequality follows from the Cauchy--Schwarz inequality and the concavity of $\sqrt{\cdot}$, the third one follows from \eqref{eq:log-bound} in Lemma \ref{lem:summation} and Lemma \ref{lem:stoc-ip-sum} with $b_k\equiv 1$, and the last one follows from the concavity of $\log(\cdot)$.
Finally, applying \eqref{eq:stoc-S-lip} to \eqref{eq:stoch-E-sum-Delta-ns} and hiding fixed constants, we obtain
\begin{align*}
\ebb\Big[\sum_{k=1}^t \Delta_{k+1}\Big] \lesssim&(1+\sigma)\sqrt{t\log \big(1+(1+\sigma^2)t\big)}+(1+\sigma)\sqrt{t\log\Big(1+(1+\sigma)\sqrt{t\log \big(1+(1+\sigma^2)t\big)}\Big)}\\
\lesssim& (1+\sigma)\sqrt{t\log \big(1+(1+\sigma^2)t\big)},
\end{align*}
where the last inequality follows from $\log \big(1+(1+\sigma^2)t\big)\leq (1+\sigma^2)t\leq (1+\sigma)^2t$.
This together with the convexity of $F$ implies that
\[
\ebb[F(\bar{x}_{t})-F^*] \leq \frac{1}{t}\ebb\Big[\sum_{k=1}^t \Delta_{k+1}\Big] \lesssim \frac{(1+\sigma)\sqrt{\log \big(1+(1+\sigma^2)t\big)}}{\sqrt{t}}.
\]
The proof is now complete.

\subsection{Proof of Theorem~\ref{thm:stoch-convex-smooth}}
\label{app:thm-cvxsmt}

We start by summing \eqref{eq:descent-smooth} over $k=1,\dots, t$ and noting $\|x_{k+1}-x_k\|=\eta\|\widetilde{\gcal}_k\|/S_k$ to obtain
\begin{equation}\label{eq:stoch-sum-Delta-smooth}
    \begin{aligned}
    \sum_{k=1}^t\Delta_{k+1} \leq & \sum_{k=1}^t\frac{1}{2 \eta_k}\big(\|x_k-x^*\|^2-\|x_{k+1}-x^*\|^2\big)+\frac{\eta^2}{2}\sum_{k=1}^t\Big(L-\frac{1}{\eta_k}\Big)\frac{\|\widetilde{\gcal}_k\|^2}{S_k^2}\\
    & +\sum_{k=1}^t\langle\delta_k, x_k-x^*\rangle+\sum_{k=1}^t\langle\delta_k, x_{k+1}-x_k\rangle\\
    \leq & \frac{D^2}{2 \eta} S_t+\frac{\eta^2}{2}\sum_{k=1}^t\Big(L-\frac{1}{\eta_k}\Big)\frac{\|\widetilde{\gcal}_k\|^2}{S_k^2}+\sum_{k=1}^t\langle\delta_k, x_k-x^*\rangle+\sum_{k=1}^t\langle\delta_k, x_{k+1}-x_k\rangle,
    \end{aligned}
\end{equation}
where the second inequality follows from \eqref{eq:bound-x-diff}.

To bound the second term on the RHS of \eqref{eq:stoch-sum-Delta-smooth}, we consider two cases. Firstly, if $S_1:=\gamma> L\eta$, then it follows that
\[
\frac{1}{\eta_k}=\frac{S_k}{\eta}\ge \frac{S_1}{\eta}>L,~\forall k\ge 1,
\]
which further implies
\[
\frac{\eta^2}{2}\sum_{k=1}^t\Big(L-\frac{1}{\eta_k}\Big)\frac{\|\widetilde{\gcal}_k\|^2}{S_k^2}\leq 0.
\]
Otherwise, if $\gamma\leq L\eta$, then $t^* := \max\{k \in [t] : L \geq 1/\eta_k\}$ is well-defined. Since $\eta_k=\eta/S_k$ is non-increasing, we have $L \geq S_k/\eta$ for all $1\leq k \leq t^*$ and hence $S_{t^*} \leq L\eta$. For $k > t^*$, we know from the definition of $t^*$ that $L\leq 1/\eta_k$. Therefore, by applying \eqref{eq:log-bound} in Lemma \ref{lem:summation} and noting that $\|\widetilde{\gcal}_k\|\leq 2D S_k/\eta$ for all $k\ge 1$, we obtain
\[
\frac{\eta^2}{2}\sum_{k=1}^t\Big(L-\frac{1}{\eta_k}\Big)\frac{\|\widetilde{\gcal}_k\|^2}{S_k^2} \leq \frac{L\eta^2}{2} \bigg(\sum_{k=1}^{t^*-1} \frac{\|\widetilde{\gcal}_k\|^2}{S_k^2}+ \frac{\|\widetilde{\gcal}_{t^*}\|^2}{S_{t^*}^2}\bigg)\leq \frac{L \eta^2}{2}\Big(1+\frac{4D^2}{\eta^2}\Big) \log \frac{L^2 \eta^2}{\gamma^2}+2L D^2.
\]
Combining the above two cases, we have
\begin{equation}\label{eq:stoch-smooth-sum-L}
    \frac{\eta^2}{2}\sum_{k=1}^t\Big(L-\frac{1}{\eta_k}\Big)\frac{\|\widetilde{\gcal}_k\|^2}{S_k^2} \leq \Big[L(\eta^2+4D^2) \log \frac{L \eta}{\gamma}+2L D^2\Big]_+.
\end{equation}

To proceed, we take expectation on both sides of \eqref{eq:stoch-sum-Delta-smooth} and apply $\ebb_k[\langle\delta_k, x_k-x^*\rangle]=0$ for all $k\ge 1$, \eqref{eq:stoch-smooth-sum-L}, and Lemma \ref{lem:stoc-ip-sum} with $b_k=1$ to obtain
\begin{equation}\label{eq:stoch-E-sum-Delta-smooth}
\begin{aligned}
\ebb\Big[\sum_{k=1}^t \Delta_{k+1}\Big] & \leq \frac{D^2}{2 \eta} \ebb[S_t]+\Big[L(\eta^2+4D^2) \log \frac{L \eta}{\gamma}+2L D^2\Big]_+ +\ebb\bigg[\sum_{k=1}^t\langle\delta_k, x_{k+1}-x_k\rangle\bigg] \\
& \leq \frac{D^2}{2 \eta} \ebb[S_{t+1}]+\Big[L(\eta^2+4D^2) \log \frac{L \eta}{\gamma}+2L D^2\Big]_+\\
&\quad +\sigma \sqrt{2(\eta^2+4D^2) t} \sqrt{\log \ebb\bigg[\frac{S_{t+1}}{\gamma}\bigg]}.
\end{aligned}
\end{equation}
Moreover, we have established from the analysis in the nonconvex smooth case (cf. \eqref{eq:stoch-S-final}) and $b_k=1$ that
\begin{equation}\label{eq:stoch-S-smooth}
\ebb[S_{t+1}] \lesssim 1+\big(1+\sigma\sqrt{t}\big) \log \big(1+\sigma\sqrt{t}\big).
\end{equation}
Plugging \eqref{eq:stoch-S-smooth} into \eqref{eq:stoch-E-sum-Delta-smooth} and using $\sqrt{x}\leq \frac{1}{2}(1+x)$ yields
\begin{align*}
\ebb\Big[\sum_{k=1}^t \Delta_{k+1}\Big] \lesssim & 1+\big(1+\sigma\sqrt{t}\big) \log \big(1+\sigma\sqrt{t}\big)+\sigma\sqrt{t} \sqrt{\log \big(1+\sigma\sqrt{t}\big)}\\
\leq & 1+\big(1+\sigma\sqrt{t}\big) \log \big(1+\sigma\sqrt{t}\big)+\frac{\sigma\sqrt{t}}{2}\Big(1+\log \big(1+\sigma\sqrt{t}\big)\Big)\\
\lesssim & 1+\big(1+\sigma\sqrt{t}\big) \log \big(1+\sigma\sqrt{t}\big),
\end{align*}
which together with the convexity of $F$ implies that
\[
\ebb[F(\bar{x}_{t})-F^*] \leq \frac{1}{t}\ebb\Big[\sum_{k=1}^t \Delta_{k+1}\Big] \lesssim \frac{1+\log \big(1+\sigma\sqrt{t}\big)}{t}+\frac{\sigma\log \big(1+\sigma \sqrt{t}\big)}{\sqrt{t}}.
\]
The proof is now complete.

\section{Proofs for Section~\ref{sec:acc}}
\label{app:acc-proofs}

With slight abuse of notation, we denote
\[
\Delta_k:=F(y_k)-F^*,\quad \eta_k:=\frac{\eta}{S_k},\quad  \widetilde{\gcal}_{z,k}:=\frac{z_k-z_{k+1}}{\eta_k}
\]
for the analysis of Algorithm \ref{Alg:acc-adaprox} throughout this section. 
Then it follows from the update rule of $S_{k+1}$ that
\begin{equation}\label{eq:S-formula}
    S_k^2=\gamma^2+\sum_{j=1}^{k-1}\|\widetilde{\gcal}_{z,j}\|^2.
\end{equation}

To prove the main results, we first introduce the following descent lemma for Algorithm \ref{Alg:acc-adaprox} under the convexity of $f$.
\begin{lemma}\label{lem:F-descent-cvx-acc}
Suppose $f$ in \eqref{main-prob} is convex. Let $\{x_k, y_k, z_k\}_{k\ge 1}$ be the iterates generated by Algorithm \ref{Alg:acc-adaprox}. Then, for any $k\ge 1$ and $x\in\rbb^d$, we have the following results:
\begin{itemize}[leftmargin=2em]
   \item[(\romannumeral 1).] If $f$ is $G$-Lipschitz continuous, then
    \begin{equation}\label{eq:descent-lip-acc}
    \begin{aligned}
    F(y_{k+1}) \leq & \big(1-\theta_k\big) F(y_k)+\theta_k F(x)+\theta_k\langle \delta_k, z_{k+1}-z_k\rangle+\theta_k\langle\delta_k, z_k-x\rangle \\
    & +\frac{\theta_k^2}{2 \eta_k}\big(\|z_k-x\|^2-\|z_{k+1}-x\|^2\big)-\frac{\theta_k^2}{2 \eta_k}\|z_{k+1}-z_k\|^2+2 G \theta_k\|z_{k+1}-z_k\|.
    \end{aligned}
    \end{equation}
    \item[(\romannumeral 2).] If $f$ is $L$-smooth, then
    \begin{equation}\label{eq:descent-smooth-acc}
    \begin{aligned}
    F(y_{k+1}) \leq & \big(1-\theta_k\big) F(y_k)+\theta_k F(x)+\theta_k\langle\delta_k, z_{k+1}-z_k\rangle+\theta_k\langle\delta_k, z_k-x\rangle \\
    & +\frac{\theta_k^2}{2 \eta_k}\big(\|z_k-x\|^2-\|z_{k+1}-x\|^2\big)+\frac{\theta_k^2}{2}\Big(L-\frac{1}{\eta_k}\Big)\|z_{k+1}-z_k\|^2.
    \end{aligned}
    \end{equation}
\end{itemize}
\end{lemma}
\begin{proof}
We first prove (\romannumeral 1). Since $f$ is convex and $G$-Lipschitz continuous, arguing as in the derivation of \eqref{eq:Lip-f} gives
\begin{equation}\label{eq:yk-lip}
f(y_{k+1}) \leq f(x_k)+\langle \nabla f(x_k), y_{k+1}-x_k\rangle+2G\|y_{k+1}-x_k\|.
\end{equation}
To bound $f(x_k)$, we obtain from the convexity of $f$ that
\begin{equation}\label{eq:f-xk}
\begin{aligned}
f(x_k)-\big[(1-\theta_k) f(y_k)+\theta_k f(x)\big]=&\theta_k\big(f(x_k)-f(x)\big)+\big(1-\theta_k\big)\big(f(x_k)-f(y_k)\big) \\
\leq & \theta_k\langle \nabla f(x_k), x_k-x\rangle+\big(1-\theta_k\big)\langle \nabla f(x_k), x_k-y_k\rangle \\
= & \langle \nabla f(x_k), x_k-\theta_k x-\big(1-\theta_k\big) y_k\rangle.
\end{aligned}
\end{equation}
Applying \eqref{eq:f-xk} to \eqref{eq:yk-lip} leads to
\begin{equation}\label{eq:f-yk}
f(y_{k+1}) \leq\big(1-\theta_k\big) f(y_k)+\theta_k f(x)+\langle\nabla f(x_k), y_{k+1}-\theta_k x-\big(1-\theta_k\big) y_k\rangle+2G\|y_{k+1}-x_k\|.
\end{equation}
To proceed, note from the update rules of $x_k$ and \eqref{alg:y-update} that
\begin{equation}\label{eq:y-rule}
y_{k+1}=x_k+\theta_k\big(z_{k+1}-z_k\big)=(1-\theta_k)y_k + \theta_k z_k +\theta_k\big(z_{k+1}-z_k\big) = \theta_k z_{k+1}+\big(1-\theta_k\big) y_k,
\end{equation}
which together with \eqref{eq:f-yk} implies that
\begin{equation}\label{eq:f-yk-2}
\begin{aligned}
f(y_{k+1}) \leq & \big(1-\theta_k\big) f(y_k)+\theta_k f(x)+\theta_k\langle\nabla f(x_k), z_{k+1}-x\rangle+2G \theta_k\|z_{k+1}-z_k\| \\
=&\big(1-\theta_k\big) f(y_k)+\theta_k f(x)+\theta_k\langle g_k, z_{k+1}-x\rangle+\theta_k\langle\delta_k, z_{k+1}-z_k\rangle \\
& +\theta_k\langle\delta_k, z_k-x\rangle+2G \theta_k\|z_{k+1}-z_k\|.
\end{aligned}
\end{equation}
Next, we bound the inner product $\theta_k\langle g_k, z_{k+1}-x\rangle$ on the RHS of \eqref{eq:f-yk-2}. To this end, we first note from the convexity of $h$ and \eqref{eq:y-rule} that
\[
h(y_{k+1}) \leq \theta_k h(z_{k+1})+\big(1-\theta_k\big) h(y_k),
\]
which can be rearranged as (noting $\alpha_k:=1/\theta_k$)
\begin{equation}\label{eq:h-zk}
h(z_{k+1}) \geq \alpha_k h(y_{k+1})-\alpha_k\big(1-\theta_k\big) h(y_k).
\end{equation}
By the optimality condition of $z_{k+1}$, $\eta_k=\eta/S_k$, and \eqref{eq:h-zk}, we have
\begin{equation*}
\begin{aligned}
h(x) \geq & h(z_{k+1})+\Big\langle g_k+\frac{\theta_k}{\eta_k}\big(z_{k+1}-z_k\big), z_{k+1}-x\Big\rangle \\
\ge & \alpha_k h(y_{k+1})-\alpha_k\big(1-\theta_k\big) h(y_k)+\langle g_k, z_{k+1}-x\rangle+\frac{\theta_k}{ \eta_k}\langle z_{k+1}-z_k, z_{k+1}-x\rangle\\
= & \alpha_k h(y_{k+1})-\alpha_k\big(1-\theta_k\big) h(y_k)+\langle g_k, z_{k+1}-x\rangle\\
&+\frac{\theta_k}{2 \eta_k}\big(\|z_{k+1}-x\|^2+\|z_{k+1}-z_k\|^2-\|z_k-x\|^2\big),
\end{aligned}
\end{equation*}
which gives the following after multiplying by $\theta_k$ and rearranging the terms:
\begin{equation}\label{eq:ip-gk-zk}
\begin{aligned}
\theta_k\langle g_k, z_{k+1}-x\rangle \leq & \theta_k h(x)+\big(1-\theta_k\big) h(y_k)-h(y_{k+1})\\
&\quad +\frac{\theta_k^2}{2 \eta_k}(\|z_{k}-x\|^2-\|z_{k+1}-x\|^2-\|z_{k+1}-z_k\|^2).
\end{aligned}
\end{equation}
Applying \eqref{eq:ip-gk-zk} to \eqref{eq:f-yk-2} and rearranging the terms using $F=f+h$ leads to \eqref{eq:descent-lip-acc}.

We proceed to prove (\romannumeral 2), whose proof largely follows that of (\romannumeral 1). The main difference is that when $f$ is convex and $L$-smooth, instead of \eqref{eq:yk-lip}, we have the following refined bound:
\[
f(y_{k+1}) \leq f(x_k)+\langle \nabla f(x_k), y_{k+1}-x_k\rangle+\frac{L}{2}\|y_{k+1}-x_k\|^2.
\]
Then, following the arguments from \eqref{eq:f-xk} to \eqref{eq:f-yk-2}, we obtain
\begin{equation*}
\begin{aligned}
    f(y_{k+1}) \leq& \big(1-\theta_k\big) f(y_k)+\theta_k f(x)+\theta_k\langle g_k, z_{k+1}-x\rangle+\theta_k\langle\delta_k, z_{k+1}-z_k\rangle\\
    &+\theta_k\langle\delta_k, z_k-x\rangle+\frac{L\theta_k^2}{2}\|z_{k+1}-z_k\|^2,
\end{aligned}    
\end{equation*}
which together with \eqref{eq:ip-gk-zk} leads to \eqref{eq:descent-smooth-acc}. The proof is thus complete.
\end{proof}

To estimate the summation of inner products involving the stochastic noise, we also need the following lemma, whose proof is similar to that of Lemma \ref{lem:stoc-ip-sum}. 
\begin{lemma}\label{lem:stoc-ip-sum-acc}
    Suppose Assumptions \ref{ass:bd-stoch} and \ref{ass:z-bd} hold. Let $\{x_k, y_k, z_k\}_{k\ge 1}$ be the iterates generated by Algorithm \ref{Alg:acc-adaprox} with unit batch sizes $b_k\equiv 1$ and any $\gamma, \eta>0$. Then for any $t \ge 1$, we have
    \[
    \ebb\Big[\sum_{k=1}^t\alpha_k\langle\delta_k, z_{k+1}-z_k\rangle\Big] \leq \sigma \sqrt{2(\eta^2+4D^2) \sum_{k=1}^t \alpha_k^2} \sqrt{\log \ebb\bigg[\frac{S_{t+1}}{\gamma}\bigg]}.
    \]
\end{lemma}
\begin{proof}
Firstly, note from the Cauchy--Schwarz inequality and $\|z_{k+1}-z_k\|=\eta\|\widetilde{\gcal}_{z,k}\|/S_k$ that
\[
\sum_{k=1}^t\alpha_k\langle\delta_k, z_{k+1}-z_k\rangle\leq\sum_{k=1}^t\alpha_k\|\delta_k\| \|z_{k+1}-z_k\|=\eta\sum_{k=1}^t\alpha_k\|\delta_k\| \frac{\|\widetilde{\gcal}_{z,k}\|}{S_k}.
\]
Taking expectations on both sides and applying the Cauchy--Schwarz inequality for expectations $|\mathbb{E}[XY]| \leq \sqrt{\mathbb{E}[X^2] \mathbb{E}[Y^2]}$, we obtain
\begin{equation}\label{eq:stoch-error-acc}
\begin{aligned}
&\ebb\Big[\sum_{k=1}^t\alpha_k\langle\delta_k, z_{k+1}-z_k\rangle\Big]\\
\leq & \eta\sum_{k=1}^t\alpha_k\ebb\bigg[\|\delta_k\| \frac{\|\widetilde{\gcal}_{z,k}\|}{S_k}\bigg]
\leq \eta\sum_{k=1}^t\alpha_k \sqrt{\ebb[\|\delta_k\|^2] \ebb\bigg[\frac{\|\widetilde{\gcal}_{z,k}\|^2}{S_k^2}\bigg]}\\
\leq& \sigma \eta\sum_{k=1}^t\alpha_k\sqrt{\ebb\bigg[\frac{\|\widetilde{\gcal}_{z,k}\|^2}{S_k^2}\bigg]}
\leq \sigma \eta\sqrt{\sum_{k=1}^t \alpha_k^2} \sqrt{\ebb\bigg[\sum_{k=1}^t \frac{\|\widetilde{\gcal}_{z,k}\|^2}{S_k^2}\bigg]},
\end{aligned}
\end{equation}
where we used $\ebb_k[\|\delta_k\|^2]\leq\sigma^2$ for the third inequality, and the Cauchy--Schwarz inequality for the last one. Applying \eqref{eq:log-bound} in Lemma \ref{lem:summation} with $a_k:=\|\widetilde{\gcal}_{z,k}\|^2$, $A_k:=S_{k+1}^2$, and $C:=4D^2/\eta^2$, we further derive from \eqref{eq:stoch-error-acc} that
\begin{align*}
\ebb\Big[\sum_{k=1}^t\alpha_k\langle\delta_k, z_{k+1}-z_k\rangle\Big]
\leq & \sigma\eta \sqrt{\sum_{k=1}^t \alpha_k^2} \sqrt{\Big(1+\frac{4D^2}{\eta^2}\Big) \ebb\bigg[\log \frac{S_{t+1}^2}{\gamma^2}\bigg]}\\
\leq & \sigma \sqrt{2(\eta^2+4D^2) \sum_{k=1}^t \alpha_k^2} \sqrt{\log \ebb\bigg[\frac{S_{t+1}}{\gamma}\bigg]},
\end{align*}
where the last inequality follows from the concavity of $\log(\cdot)$. The proof is now complete.
\end{proof}
In our analysis, we leverage the following properties of the weights $\{\alpha_k\}_{k\ge 1}$:
\begin{equation}\label{eq:alpha-property}
    \alpha_k^2-\alpha_k=\alpha_{k-1}^2~~\textit{and}~~\frac{k+1}{2}\leq\alpha_k\leq k+1,\quad\forall k\ge 1.
\end{equation}
Both results in \eqref{eq:alpha-property} follow from the definition of $\alpha_k$: the first identity is the quadratic equation in $\alpha_k$, while the inequality can be verified by induction.

\subsection{Proof of Theorem~\ref{thm:cvx-lip-acc}}
\label{app:thm-cvxlipacc}

We start by plugging $x=x^*$ into \eqref{eq:descent-lip-acc} to obtain
\begin{align}
\Delta_{t+1} \leq & \big(1-\theta_t\big) \Delta_t+\theta_t\langle\delta_t, z_{t+1}-z_t\rangle+\theta_t\langle\delta_t, z_t-x^*\rangle +\frac{\theta_t^2}{2 \eta_t}\big(\|z_t-x^*\|^2-\|z_{t+1}-x^*\|^2\big)\notag\\
&-\frac{\theta_t^2}{2 \eta_t}\|z_{t+1}-z_t\|^2+2 G \theta_t\|z_{t+1}-z_t\|\label{eq:lip-acc-1}.
\end{align}
Multiplying both sides of \eqref{eq:lip-acc-1} by $\alpha_t^2$ and applying $\alpha_t\theta_t=1$ gives 
\begin{equation*}
\begin{aligned}
\alpha_t^2 \Delta_{t+1} \leq & (\alpha_t^2-\alpha_t) \Delta_t +\alpha_t\langle\delta_t, z_{t+1}-z_t\rangle+\alpha_t\langle\delta_t, z_t-x^*\rangle\\
&\quad +\frac{1}{2 \eta_t}\big(\|z_t-x^*\|^2-\|z_{t+1}-x^*\|^2\big)\\
&-\frac{1}{2 \eta_t}\|z_{t+1}-z_t\|^2+2 G \alpha_t\|z_{t+1}-z_t\|,
\end{aligned}
\end{equation*}
which together with the equality in \eqref{eq:alpha-property} and $\|z_{t+1}-z_t\|=\eta_t\|\widetilde{\gcal}_{z,t}\|=\eta\|\widetilde{\gcal}_{z,t}\|/S_t$ yields
\begin{align}
    \alpha_t^2 \Delta_{t+1} \leq & \alpha_{t-1}^2 \Delta_t +\alpha_t\langle\delta_t, z_{t+1}-z_t\rangle+\alpha_t\langle\delta_t, z_t-x^*\rangle+\frac{1}{2 \eta_t}\big(\|z_t-x^*\|^2-\|z_{t+1}-x^*\|^2\big)\notag\\
    &-\frac{\eta}{2}\frac{\|\widetilde{\gcal}_{z,t}\|^2}{S_t}+2 G \eta\alpha_t\frac{\|\widetilde{\gcal}_{z,t}\|}{S_t}\label{eq:lip-acc-2}.
\end{align}
Applying \eqref{eq:lip-acc-2} recursively and noting from Assumption \ref{ass:z-bd} that
\begin{equation}\label{eq:bound-z-diff}
    \begin{aligned}
    &\sum_{k=1}^t \frac{1}{2 \eta_k}\big(\|z_k-x^*\|^2-\|z_{k+1}-x^*\|^2\big) \\
    =& \frac{\|z_1-x^*\|^2}{2\eta_1}+\sum_{k=2}^t\Big(\frac{1}{2\eta_k} - \frac{1}{2\eta_{k-1}}\Big)\|z_k-x^*\|^2 - \frac{\|z_{t+1}-x^*\|^2}{2\eta_t}\\
    \leq & \frac{D^2}{2\eta_1} + D^2\sum_{k=2}^t\Big(\frac{1}{2\eta_k} - \frac{1}{2\eta_{k-1}}\Big)=\frac{D^2}{2\eta_t}\leq \frac{D^2}{2 \eta}S_{t+1},
    \end{aligned}
\end{equation}
we further obtain
\begin{equation}\label{eq:lip-acc-desent}
    \begin{aligned}
         \alpha_t^2 \Delta_{t+1} \leq & \alpha_0^2 \Delta_1-\frac{\eta}{2} \sum_{k=1}^t \frac{\|\widetilde{\gcal}_{z,k}\|^2}{S_k}+\frac{D^2}{2 \eta}S_{t+1}+2 G \eta \sum_{k=1}^t \alpha_k \frac{\|\widetilde{\gcal}_{z,k}\|}{S_k}\\
         &+\sum_{k=1}^t \alpha_k\langle\delta_k, z_{k+1}-z_k\rangle+\sum_{k=1}^t \alpha_k\langle\delta_k, z_k-x^*\rangle.
    \end{aligned}
\end{equation}
To proceed, from the definition of $S_{t}$ \eqref{eq:S-formula}, we apply \eqref{eq:sqrt-bound} in Lemma \ref{lem:summation} to derive
\begin{equation}\label{eq:acc-lip-sum1}
    \frac{\eta}{2} \sum_{k=1}^t \frac{\|\widetilde{\gcal}_{z,k}\|^2}{S_k}\ge \eta\big(S_{t+1}-\gamma\big).
\end{equation}
Additionally, since $\|\widetilde{\gcal}_{z,t}\|^2=\|z_{t+1}-z_t\|^2/\eta_t^2\leq 4D^2S_t^2/\eta^2$, it follows from the Cauchy--Schwarz inequality and \eqref{eq:log-bound} in Lemma \ref{lem:summation} that
\begin{equation}\label{eq:acc-lip-sum2}
\begin{aligned}
& 2 G \eta \sum_{k=1}^t \alpha_k \frac{\|\widetilde{\gcal}_{z,k}\|}{S_k} \leq 2 G \eta \sqrt{\sum_{k=1}^t \alpha_k^2} \sqrt{\sum_{k=1}^t \frac{\|\widetilde{\gcal}_{z,k}\|^2}{S_k^2}} \\
\leq & 2 G \eta \sqrt{\sum_{k=1}^t \alpha_k^2} \sqrt{\Big(1+\frac{4D^2}{\eta^2}\Big) \log \frac{S_{t+1}^2}{\gamma^2}} \leq 2 G \sqrt{2(\eta^2+4D^2) \sum_{k=1}^t \alpha_k^2} \sqrt{\log \frac{S_{t+1}}{\gamma}}.
\end{aligned}
\end{equation}
Applying \eqref{eq:acc-lip-sum1} and \eqref{eq:acc-lip-sum2} to \eqref{eq:lip-acc-desent} leads to
\begin{equation}\label{eq:lip-acc-desent-1}
    \begin{aligned}
         \alpha_t^2 \Delta_{t+1} \leq & \alpha_0^2 \Delta_1 + \gamma\eta-\Big(\eta-\frac{D^2}{2\eta}\Big)S_{t+1}+2 G \sqrt{2(\eta^2+4D^2) \sum_{k=1}^t \alpha_k^2} \sqrt{\log \frac{S_{t+1}}{\gamma}}\\
         &+\sum_{k=1}^t \alpha_k\langle\delta_k, z_{k+1}-z_k\rangle+\sum_{k=1}^t \alpha_k\langle\delta_k, z_k-x^*\rangle.
    \end{aligned}
\end{equation}
Taking expectations on both sides of \eqref{eq:lip-acc-desent-1} and noting that $\eta-\frac{D^2}{2\eta}> 0$ (since $\eta> \frac{\sqrt{2}}{2}D$) and that $\ebb_k[\langle\delta_k, z_k-x^*\rangle]=0$ for all $k\ge 1$, we derive
\begin{align}
    &\alpha_t^2 \ebb[\Delta_{t+1}] \leq \alpha_0^2 \Delta_1 + \gamma\eta+2 G \sqrt{2(\eta^2+4D^2) \sum_{k=1}^t \alpha_k^2}\ebb\Bigg[\sqrt{\log \frac{S_{t+1}}{\gamma}}\Bigg]+\ebb\bigg[\sum_{k=1}^t \alpha_k\langle\delta_k, z_{k+1}-z_k\rangle\bigg]\notag\\
    \leq & \alpha_0^2 \Delta_1 + \gamma\eta+2 G \sqrt{2(\eta^2+4D^2) \sum_{k=1}^t \alpha_k^2}\sqrt{\log \ebb\bigg[\frac{S_{t+1}}{\gamma}\bigg]} + \sigma \sqrt{2(\eta^2+4D^2) \sum_{k=1}^t \alpha_k^2} \sqrt{\log \ebb\bigg[\frac{S_{t+1}}{\gamma}\bigg]}\notag\\
    = & \alpha_0^2 \Delta_1 + \gamma\eta+ (2 G+\sigma) \sqrt{2(\eta^2+4D^2) \sum_{k=1}^t \alpha_k^2}\sqrt{\log \ebb\bigg[\frac{S_{t+1}}{\gamma}\bigg]}\label{eq:lip-acc-desent-2},
\end{align}
where the second inequality follows from the concavity of $\sqrt{\log(\cdot)}$ and Lemma \ref{lem:stoc-ip-sum-acc}.

Next, we bound $\ebb[S_{t+1}]$. To this end, rearranging \eqref{eq:lip-acc-desent-1} gives
\begin{equation}\label{eq:S-lip-acc-1}
\begin{aligned}
\Big(\eta-\frac{D^2}{2\eta}\Big)S_{t+1} \leq & \alpha_0^2 \Delta_1+\gamma\eta-\alpha_t^2 \Delta_{t+1}+2 G \sqrt{2(\eta^2+4D^2) \sum_{k=1}^t \alpha_k^2} \sqrt{\log \frac{S_{t+1}}{\gamma}}\\
&+\sum_{k=1}^t \alpha_k\langle\delta_k, z_{k+1}-z_k\rangle+\sum_{k=1}^t \alpha_k\langle\delta_k, z_k-x^*\rangle.
\end{aligned}
\end{equation}
Taking expectations on both sides of \eqref{eq:S-lip-acc-1} and using $\Delta_{t+1}\ge 0$, the concavity of $\sqrt{\log(\cdot)}$, and Lemma \ref{lem:stoc-ip-sum-acc}, we obtain
\begin{equation}\label{eq:S-lip-acc-2}
\begin{aligned}
&\Big(\eta-\frac{D^2}{2\eta}\Big)\ebb[S_{t+1}]\\ \leq & \alpha_0^2 \Delta_1+\gamma\eta+2 G \sqrt{2(\eta^2+4D^2) \sum_{k=1}^t \alpha_k^2} \ebb\Bigg[\sqrt{\log \frac{S_{t+1}}{\gamma}}\Bigg]+\ebb\bigg[\sum_{k=1}^t \alpha_k\langle\delta_k, z_{k+1}-z_k\rangle\bigg]\\
\leq & \alpha_0^2 \Delta_1 + \gamma\eta+ (2 G+\sigma) \sqrt{2(\eta^2+4D^2) \sum_{k=1}^t \alpha_k^2}\sqrt{\log \ebb\bigg[\frac{S_{t+1}}{\gamma}\bigg]}.
\end{aligned}
\end{equation}
Let $\rho:=1-\frac{D^2}{2\eta^2}\in(0,1)$. Dividing both sides of \eqref{eq:S-lip-acc-2} by $\gamma\eta\rho$ leads to
\[
\ebb\bigg[\frac{S_{t+1}}{\gamma}\bigg] \leq \frac{\alpha_0^2 \Delta_1}{\gamma \eta \rho}+\frac{1}{\rho}+\frac{2 G+\sigma}{\gamma \eta \rho} \sqrt{2(\eta^2+4D^2) \sum_{k=1}^t \alpha_k^2} \sqrt{\log \ebb\bigg[\frac{S_{t+1}}{\gamma}\bigg]}.
\]
Solving the above inequality using Lemma \ref{lem:sqrt-log} yields
\begin{align*}
\ebb\bigg[\frac{S_{t+1}}{\gamma}\bigg] \leq &\frac{4G+2\sigma}{\gamma \eta \rho} \sqrt{2(\eta^2+4D^2) \sum_{k=1}^t \alpha_k^2} \sqrt{\log \bigg(\frac{2(2 G+\sigma)^2(\eta^2+4D^2)}{\gamma^2 \eta^2 \rho^2} \sum_{k=1}^t \alpha_k^2+1\bigg)}\\
&+\frac{2 \alpha_0^2 \Delta_1}{\gamma \eta \rho}+\frac{2}{\rho}+1,
\end{align*}
which leads to the following after omitting the constants $\{\gamma, \eta, D, G, \alpha_0^2\Delta_1\}$ and noting $\alpha_k\leq k+1$ for all $k\ge 1$:
\begin{equation}\label{eq:S-final-lip-acc}
    \ebb[S_{t+1}]\lesssim 1+(1+\sigma) \sqrt{\sum_{k=1}^t \alpha_k^2} \sqrt{\log \Big(1+(1+\sigma^2) \sum_{k=1}^t \alpha_k^2\Big)}\lesssim (1+\sigma) t^{3/2} \sqrt{\log \big(1+(1+\sigma^2) t^3\big)}.
\end{equation}

Finally, applying \eqref{eq:S-final-lip-acc} and $\alpha_k\leq k+1$ to \eqref{eq:lip-acc-desent-2} implies
\begin{equation}\label{eq:lip-acc-desent-3}
\begin{aligned}
\alpha_t^2 \ebb[\Delta_{t+1}] & \lesssim 1+(1+\sigma) \sqrt{\sum_{k=1}^t \alpha_k^2} \sqrt{\log \ebb\bigg[\frac{S_{t+1}}{\gamma}\bigg]} \\
& \lesssim (1+\sigma)t^{3/2} \sqrt{\log \Big(1+(1+\sigma) t^{3/2} \sqrt{\log \big(1+(1+\sigma^2) t^3\big)}\Big)}\\
& \lesssim (1+\sigma)t^{3/2} \sqrt{\log \big(1+(1+\sigma^2) t^{3}\big)},
\end{aligned}
\end{equation}
where the last inequality follows from $\log \big(1+(1+\sigma^2) t^3\big)\leq (1+\sigma^2) t^3\lesssim (1+\sigma)^2 t^3$. Dividing both sides of \eqref{eq:lip-acc-desent-3} by $\alpha_t^2$ and using $\alpha_t\ge \frac{t+1}{2}$ yields
\[
    \ebb[\Delta_{t+1}] \lesssim \frac{1+\sigma}{\sqrt{t}} \sqrt{\log \big(1+(1+\sigma^2) t^3\big)},
\]
which completes the proof.

\subsection{Proof of Theorem~\ref{thm:smooth-acc}}
\label{app:thm-smtacc}

We start by plugging $x=x^*$ into \eqref{eq:descent-smooth-acc} to obtain
\begin{equation*}
\begin{aligned}
\Delta_{t+1} \leq & \big(1-\theta_t\big) \Delta_t+\theta_t\langle\delta_t, z_{t+1}-z_t\rangle+\theta_t\langle\delta_t, z_t-x^*\rangle \\
& +\frac{\theta_t^2}{2 \eta_t}\big(\|z_t-x^*\|^2-\|z_{t+1}-x^*\|^2\big)+\frac{\theta_t^2}{2}\Big(L-\frac{1}{\eta_t}\Big)\|z_{t+1}-z_t\|^2.
\end{aligned}
\end{equation*}
Multiplying both sides of the above inequality by $\alpha_t^2$ and using the equality in \eqref{eq:alpha-property} leads to
\begin{equation*}
\begin{aligned}
\alpha_t^2 \Delta_{t+1} \leq & \alpha_{t-1}^2 \Delta_t+\frac{1}{2}\Big(L-\frac{1}{\eta_t}\Big)\|z_{t+1}-z_t\|^2+\frac{1}{2 \eta_t}\big(\|z_t-x^*\|^2-\|z_{t+1}-x^*\|^2\big) \\
& +\alpha_t\langle\delta_t, z_{t+1}-z_t\rangle+\alpha_t\langle\delta_t, z_t-x^*\rangle,
\end{aligned}
\end{equation*}
and applying this inequality recursively together with $\|z_{t+1}-z_t\|=\eta_t\|\widetilde{\gcal}_{z,t}\|=\eta\|\widetilde{\gcal}_{z,t}\|/S_t$ and \eqref{eq:bound-z-diff} yields
\begin{equation}\label{eq:decent-acc-smooth-1}
\begin{aligned}
    \alpha_t^2 \Delta_{t+1} \leq & \alpha_0^2 \Delta_1+\frac{\eta^2}{2} \sum_{k=1}^t\Big(L-\frac{1}{\eta_k}\Big) \frac{\|\widetilde{\gcal}_{z,k}\|^2}{S_k^2}+\frac{D^2}{2 \eta}S_{t+1}\\
    &\quad +\sum_{k=1}^t \alpha_k\langle\delta_k, z_{k+1}-z_k\rangle+\sum_{k=1}^t \alpha_k\langle\delta_k, z_k-x^*\rangle.
\end{aligned}
\end{equation}
To bound the first summation term on the RHS of \eqref{eq:decent-acc-smooth-1}, we consider two cases. Firstly, if $S_1:=\gamma> L\eta$, then it follows that $1/\eta_k=S_k/\eta\ge S_1/\eta>L$ for all $k\ge 1$, which implies that
\[
\frac{\eta^2}{2}\sum_{k=1}^t\Big(L-\frac{1}{\eta_k}\Big)\frac{\|\widetilde{\gcal}_{z, k}\|^2}{S_k^2}\leq 0.
\]
Otherwise, if $\gamma\leq L\eta$, then $t^* := \max\{k \in [t] : L \geq 1/\eta_k\}$ is well-defined. Since $\eta_k=\eta/S_k$ is non-increasing, we have $L \geq S_k/\eta$ for all $1\leq k \leq t^*$ and hence $S_{t^*} \leq L\eta$. For $k > t^*$, we know from the definition of $t^*$ that $L\leq 1/\eta_k$. Therefore, by applying \eqref{eq:log-bound} in Lemma \ref{lem:summation} and noting that $\|\widetilde{\gcal}_{z, k}\|^2=\|z_{k+1}-z_k\|^2/\eta_k^2\leq 4D^2S_k^2/\eta^2$ for all $k\ge 1$, we obtain
\begin{equation*}
\begin{aligned}
& \frac{\eta^2}{2} \sum_{k=1}^t\Big(L-\frac{1}{\eta_k}\Big) \frac{\|\widetilde{\gcal}_{z,k}\|^2}{S_k^2} \leq \frac{L \eta^2}{2}\Big(\sum_{k=1}^{t^*-1} \frac{\|\widetilde{\gcal}_{z,k}\|^2}{S_k^2}+\frac{\|\widetilde{\gcal}_{z,t^*}\|^2}{S_{t^*}^2}\Big) \\
\leq& \frac{L \eta^2}{2}\bigg[\Big(1+\frac{4D^2}{\eta^2}\Big) \log \frac{S_{t^*}^2}{\gamma^2}+\frac{4D^2}{\eta^2}\bigg] \leq L\big(\eta^2+4D^2\big) \log \frac{L \eta}{\gamma}+2L D^2.
\end{aligned}
\end{equation*}
Combining the above two cases, we have
\begin{equation}\label{eq:smooth-acc-sum1}
    \frac{\eta^2}{2}\sum_{k=1}^t\Big(L-\frac{1}{\eta_k}\Big)\frac{\|\widetilde{\gcal}_{z, k}\|^2}{S_k^2} \leq \Big[L(\eta^2+4D^2) \log \frac{L \eta}{\gamma}+2L D^2\Big]_+.
\end{equation}
Taking expectations on both sides of \eqref{eq:decent-acc-smooth-1} and applying $\ebb_k[\langle\delta_k, z_k-x^*\rangle]=0$ for all $k\ge 1$, Lemma \ref{lem:stoc-ip-sum-acc}, and the bound \eqref{eq:smooth-acc-sum1}, we obtain
\begin{equation}\label{eq:smooth-acc-descent}
\begin{aligned}
\alpha_t^2 \ebb[\Delta_{t+1}] \leq & \alpha_0^2 \Delta_1+\Big[L(\eta^2+4D^2) \log \frac{L \eta}{\gamma}+2L D^2\Big]_++\frac{D^2\gamma}{2 \eta} \ebb\bigg[\frac{S_{t+1}}{\gamma}\bigg] \\
& +\sigma \sqrt{2(\eta^2+4D^2) \sum_{k=1}^t \alpha_k^2} \sqrt{\log \ebb\bigg[\frac{S_{t+1}}{\gamma}\bigg]}.
\end{aligned}
\end{equation}

To proceed, we bound $\ebb\big[\frac{S_{t+1}}{\gamma}\big]$. To this end, we rearrange \eqref{eq:decent-acc-smooth-1} and apply Lemma \ref{lem:summation} to derive
\begin{equation*}
\begin{aligned}
& \Big(\eta-\frac{D^2}{2 \eta}\Big)S_{t+1}-\gamma\eta=\eta\big(S_{t+1}-\gamma\big)-\frac{D^2}{2 \eta}S_{t+1} \leq \frac{\eta}{2} \sum_{k=1}^t \frac{\|\widetilde{\gcal}_{z,k}\|^2}{S_k}-\frac{D^2}{2 \eta}S_{t+1} \\
\leq & \alpha_0^2 \Delta_1-\alpha_t^2 \Delta_{t+1}+\frac{L\eta^2}{2} \sum_{k=1}^t \frac{\|\widetilde{\gcal}_{z,k}\|^2}{S_k^2}+\sum_{k=1}^t \alpha_k\langle\delta_k, z_{k+1}-z_k\rangle+\sum_{k=1}^t \alpha_k\langle\delta_k, z_k-x^*\rangle \\
\leq & \alpha_0^2 \Delta_1+\frac{L \eta^2}{2}\Big(1+\frac{4D^2}{\eta^2}\Big) \log \frac{S_{t+1}^2}{\gamma^2}+\sum_{k=1}^t \alpha_k\langle\delta_k, z_{k+1}-z_k\rangle+\sum_{k=1}^t \alpha_k\langle\delta_k, z_k-x^*\rangle
\end{aligned}
\end{equation*}
Since $\eta>\sqrt{2} D/2$, it follows that $\rho:=1-\frac{D^2}{2 \eta^2}\in (0, 1)$.
For brevity, in this proof we define $\kappa_t:=\frac{\sigma}{2\gamma\eta\rho}\sqrt{2(\eta^2+4D^2)\sum_{k=1}^t\alpha_k^2}$.
Taking expectations on both sides of the above inequality and dividing by $\gamma\eta\rho$ yields
\begin{align*}
\ebb\bigg[\frac{S_{t+1}}{\gamma}\bigg] &\leq \frac{1}{\rho}+\frac{\alpha_0^2 \Delta_1}{\gamma \eta \rho}+\frac{L(\eta^2+4D^2)}{\gamma \eta \rho} \ebb\bigg[\log \frac{S_{t+1}}{\gamma}\bigg]+2\kappa_t\sqrt{\log \ebb\bigg[\frac{S_{t+1}}{\gamma}\bigg]}\\
&\leq \frac{1}{\rho}+\frac{\alpha_0^2 \Delta_1}{\gamma \eta \rho}+\frac{L(\eta^2+4D^2)}{\gamma \eta \rho}\log \ebb\bigg[\frac{S_{t+1}}{\gamma}\bigg]+\kappa_t\bigg(1+\log \ebb\bigg[\frac{S_{t+1}}{\gamma}\bigg]\bigg)\\
&= \frac{1}{\rho}+\frac{\alpha_0^2 \Delta_1}{\gamma \eta \rho}+\kappa_t+\bigg(\frac{L(\eta^2+4D^2)}{\gamma \eta \rho}+\kappa_t\bigg) \log \ebb\bigg[\frac{S_{t+1}}{\gamma}\bigg],
\end{align*}
where the second inequality follows from the concavity of $\log(\cdot)$ and the basic inequality $\sqrt{x}\leq \frac{1}{2}(1+x)$. Solving the above inequality using Lemma \ref{lem:log} and noting $\alpha_k\leq k+1$ for all $k\ge 1$ gives
\begin{align}
\ebb\bigg[\frac{S_{t+1}}{\gamma}\bigg] &\leq \bigg(\frac{2 L\big(\eta^2+4D^2\big)}{\gamma \eta \rho}+2\kappa_t\bigg)\log \bigg(\frac{2L(\eta^2+ 4D^2)}{\gamma \eta \rho}+2\kappa_t\bigg)\notag\\
&\quad+\frac{2}{\rho}+\frac{2 \alpha_0^2 \Delta_1}{\gamma \eta \rho}- \frac{2 L\big(\eta^2+4D^2\big)}{\gamma \eta \rho}\notag\\
&\lesssim 1+\Bigg(1+\sigma\sqrt{\sum_{k=1}^t \alpha_k^2}\Bigg)\log \bigg(1+\sigma \sqrt{\sum_{k=1}^t \alpha_k^2}\bigg)\notag\\
&\lesssim 1+\big(1+\sigma t^{3/2}\big)\log \big(1+\sigma t^{3/2}\big) \label{eq:S-final-smooth-acc}.
\end{align}

Finally, applying \eqref{eq:S-final-smooth-acc} to \eqref{eq:smooth-acc-descent} yields
\begin{equation}\label{eq:smooth-acc-descent-final}
\begin{aligned}
&\alpha_t^2 \ebb[\Delta_{t+1}] \lesssim 1+ \ebb\bigg[\frac{S_{t+1}}{\gamma}\bigg]+\sigma\sqrt{\sum_{k=1}^t\alpha_k^2}\sqrt{\log \ebb\bigg[\frac{S_{t+1}}{\gamma}\bigg]}\\
\lesssim & 1+\big(1+\sigma t^{3/2}\big)\log \big(1+\sigma t^{3/2}\big)+\sigma t^{3/2} \sqrt{\log \bigg(1+\big(1+\sigma t^{3/2}\big)\log \big(1+\sigma t^{3/2}\big)\bigg)} \\
\lesssim & 1+\log \big(1+\sigma t^{3/2}\big)+\sigma t^{3/2}\log \big(1+\sigma t^{3/2}\big).
\end{aligned}
\end{equation}
Dividing both sides of \eqref{eq:smooth-acc-descent-final} by $\alpha_t^2$ and using $\alpha_t\ge \frac{t+1}{2}$ yields
\[
\ebb[\Delta_{t+1}] \lesssim \frac{1+\log \big(1+\sigma t^{3/2}\big)}{t^2}+\frac{\sigma\log \big(1+\sigma t^{3/2}\big)}{\sqrt{t}},
\]
which completes the proof.

\section{Discussion}
\label{app:discussion}

\subsection{The boundedness assumption}
\label{app:bd-discussion}
Assumption~\ref{ass:seq-bd} consists of two parts: \emph{bounded steps} ($\|x_{k+1}-x_k\|\le D$), which is used in all settings, and \emph{bounded iterates} ($\|x_k-x^*\|\le D$), which is used only when $f$ is convex. We now show that the bounded-step part holds automatically for any smooth $f$ in the deterministic setting, and hence can be removed from Theorem~\ref{thm:nc-deter} and Corollary~\ref{coro:stoch-convex-smooth}.

To this end, define the inverse step size $D_k:=1/\eta_k=S_k/\eta$, so that $\gcal_k=D_k(x_k-x_{k+1})$ and, from the update rule \eqref{eq:Sk-def},
\begin{equation}\label{eq:app-Dk-update}
D_{k+1}^2=D_k^2+\frac{\|\gcal_k\|^2}{\eta^2}.
\end{equation}
Plugging $g_k=\nabla f(x_k)$ into the descent inequality of Lemma~\ref{lem:F-descent} yields
\[
F(x_{k+1}) \leq F(x_k)-\Big(D_k-\frac{L}{2}\Big)\|x_{k+1}-x_k\|^2.
\]
Summing over $k=1,\dots,t$ and rearranging the terms, we obtain the telescoping inequality
\begin{equation}\label{eq:app-tel}
\frac{1}{2}\sum_{k=1}^t D_k\|x_{k+1}-x_k\|^2 \leq \Delta_1+\frac{1}{2}\sum_{k=1}^t (L-D_k)\|x_{k+1}-x_k\|^2.
\end{equation}
Since $\{D_k\}_{k\ge 1}$ is nondecreasing, we consider two cases. If $D_k\leq L$ for all $k\ge 1$, then it follows from \eqref{eq:app-Dk-update} that
\[
\sum_{k=1}^t\|\gcal_k\|^2=\eta^2(D_{t+1}^2-D_1^2)\leq \eta^2(L^2-D_1^2),
\]
and hence
\[
\sum_{k=1}^t\|x_{k+1}-x_k\|^2=\sum_{k=1}^t\frac{\|\gcal_k\|^2}{D_k^2}\leq \frac{\eta^2(L^2-D_1^2)}{D_1^2}<\infty.
\]
Otherwise, there exists $N\ge 0$ such that $D_k>L$ for all $k>N$, so $(L-D_k)\le 0$ for $k>N$ and the RHS of \eqref{eq:app-tel} is bounded by $\Delta_1+\frac{L}{2}\sum_{k=1}^{N}\|x_{k+1}-x_k\|^2<\infty$. In both cases, \eqref{eq:app-tel} together with $D_k\ge D_1$ implies that $\|x_{k+1}-x_k\|^2$ is uniformly bounded for all $k\ge 1$ by a constant depending only on $\Delta_1$, $L$, $\gamma$, and $\eta$.

The bounded-iterate part holds if $\dom(h)$ is bounded or the operator $\argmin_{x\in\rbb^d}\{\langle u, x\rangle+h(x)\}$ is bounded for all $u\in\rbb^d$~\citep[Lemma 2]{ghadimi2016accelerated}.
For general composite problems, removing boundedness is challenging because of the delayed accumulation and the proximal structure.
The sum inequality~\eqref{eq:log-bound} requires $\|\widetilde{\gcal}_k\|^2\leq C S_k^2$, which bounded steps ensure through $\|\widetilde{\gcal}_k\|^2/S_k^2=\|x_{k+1}-x_k\|^2/\eta^2$.
Moreover, the descent term is $\langle\nabla f(x_k),x_{k+1}-x_k\rangle=-\eta_k\langle\nabla f(x_k),\widetilde{\gcal}_k\rangle$, which generally differs from $-\eta_k\langle\nabla f(x_k),g_k\rangle$.
Thus, arguments based on the gradient-step identity in the non-composite setting do not directly extend to the proximal update.

A natural fix is to use the non-delayed accumulator
\[
S_k^2=\gamma^2+\sum_{j=1}^k\|\widetilde{\gcal}_j\|^2.
\]
This would avoid the bounded-step condition in the corresponding summation estimate, but $\widetilde{\gcal}_k$ depends on $x_{k+1}$ and hence is unavailable until the proximal step is evaluated.
An intuitive alternative is the hybrid scheme
\[
S_k^2=\gamma^2+\|g_k\|^2+\sum_{j=1}^{k-1}\|\widetilde{\gcal}_j\|^2,
\]
which recovers non-delayed AdaGrad-Norm when $h=0$.
For general $h$, however, the gradient norm does not control the gradient mapping. Consequently, adding $\|g_k\|^2$ may contribute only a bounded offset: for example, for $f(x)=x^2/2$ and $h$ the indicator function of the closed interval $[1,2]$, we have $1\leq\|g_k\|^2\leq4$ in the deterministic setting.

\subsection{Block-diagonal extension}
\label{app:block-diag}
The scalar step size in Algorithm~\ref{Alg:adaprox} can be extended to a block-diagonal (in particular, coordinate-wise) scheme. Partition $x\in\rbb^{md}$ into $m$ blocks $x=(x_1,\dots,x_m)$ with $x_i\in\rbb^d$, and define $D_k=\operatorname{diag}(D_{k,1}I_d,\dots,D_{k,m}I_d)$ with $D_{k,i}>0$ for $i\in[m]$. Replace the core update of Algorithm~\ref{Alg:adaprox} with
\begin{equation}\label{eq:app-prox-mat}
x_{k+1}=\argmin_{x\in\rbb^{md}}\Big\{\langle g_k, x\rangle+\frac{1}{2}\|x-x_k\|_{D_k}^2+h(x)\Big\},
\end{equation}
where $\|x\|_{D_k}^2:=\sum_{i=1}^m D_{k,i}\|x_i\|^2$, and replace the scalar update of $S_k$ with the block-wise update
\begin{equation}\label{eq:app-block-update}
D_{k+1,i}^2=D_{k,i}^2\Big(1+\frac{\|x_{k+1,i}-x_{k,i}\|^2}{\eta_i^2}\Big),\qquad i=1,\dots,m.
\end{equation}
Setting $d=1$ and $m=d_{\rm orig}$ recovers the fully coordinate-wise AdaGrad scaling, while setting $d=d_{\rm orig}$ and $m=1$ recovers Algorithm~\ref{Alg:adaprox}.

For the nonconvex smooth analysis, the gradient mapping of block $i$ is $\gcal_{k,i}:=D_{k,i}(x_{k,i}-x_{k+1,i})$, so $\|\gcal_{k,i}\|^2=D_{k,i}^2\|x_{k+1,i}-x_{k,i}\|^2$ and $\|\gcal_k\|^2=\sum_{i=1}^m\|\gcal_{k,i}\|^2$. Suppose $f$ is $\mathcal{B}$-smooth in the sense that
\[
f(x) \leq f(y)+\langle \nabla f(y), x-y\rangle+\frac{1}{2}\|x-y\|_{\mathcal{B}}^2,\quad \forall x,y\in\rbb^{md},
\]
with $\mathcal{B}=\operatorname{diag}(\beta_1 I_d,\dots,\beta_m I_d)$ (the scalar $L$-smoothness corresponds to $m=1$ and $\beta_1=L$). Then the descent inequality in Lemma~\ref{lem:F-descent} generalizes to
\[
F(x_{k+1}) \leq F(x_k)-\|x_{k+1}-x_k\|_{D_k}^2+\frac{1}{2}\|x_{k+1}-x_k\|_{\mathcal{B}}^2,
\]
which, after summing over $k=1,\dots,t$ and rearranging, gives
\begin{equation}\label{eq:app-tel-block}
\frac{1}{2}\sum_{k=1}^t\|x_{k+1}-x_k\|_{D_k}^2 \leq \Delta_1+\frac{1}{2}\sum_{i=1}^m\sum_{k=1}^t(\beta_i-D_{k,i})\|x_{k+1,i}-x_{k,i}\|^2.
\end{equation}
We bound the RHS of \eqref{eq:app-tel-block} by a per-block case split. Under boundedness $\|x_{k+1}-x_k\|\le R$, for each $i\in[m]$: if $D_{1,i}\ge \beta_i$, every term $(\beta_i-D_{k,i})\le 0$ and contributes nothing; otherwise, let $t_i:=\max\{k\in[t]: D_{k,i}\le\beta_i\}$. Applying the log-bound \eqref{eq:log-bound} in Lemma~\ref{lem:summation} with $a_k=\|\gcal_{k,i}\|^2/\eta_i^2$, $A_k=D_{k+1,i}^2$, and $C=R^2/\eta_i^2$ yields
\[
\sum_{k=1}^{t_i-1}\|x_{k+1,i}-x_{k,i}\|^2=\eta_i^2\sum_{k=1}^{t_i-1}\frac{\|\gcal_{k,i}\|^2/\eta_i^2}{D_{k,i}^2}\leq 2(\eta_i^2+R^2)\log\frac{\beta_i}{D_{1,i}}.
\]
Therefore, $\sum_{k=1}^{t}(\beta_i-D_{k,i})\|x_{k+1,i}-x_{k,i}\|^2\leq \beta_i\sum_{k=1}^{t_i}\|x_{k+1,i}-x_{k,i}\|^2\le 2\beta_i(\eta_i^2+R^2)\log\frac{\beta_i}{D_{1,i}}+\beta_i R^2$, and combining both cases yields
\begin{equation}\label{eq:app-M1}
\frac{1}{2}\sum_{k=1}^t\|x_{k+1}-x_k\|_{D_k}^2 \leq \Delta_1+\sum_{i=1}^m\Big[\beta_i(\eta_i^2+R^2)\log\frac{\beta_i}{D_{1,i}}+\frac{\beta_i}{2}R^2\Big]_+:=M_1.
\end{equation}
On the other hand, applying the lower bound \eqref{eq:sqrt-bound} in Lemma~\ref{lem:summation} with $a_k=\|\gcal_{k,i}\|^2/\eta_i^2$ and $A_k=D_{k+1,i}^2$ gives
\[
\sum_{k=1}^t\|x_{k+1}-x_k\|_{D_k}^2=\sum_{i=1}^m\eta_i^2\sum_{k=1}^t\frac{\|\gcal_{k,i}\|^2/\eta_i^2}{D_{k,i}}\geq 2\sum_{i=1}^m\eta_i^2(D_{t+1,i}-D_{1,i}),
\]
which together with \eqref{eq:app-M1} implies (denoting $\underline{\eta}^2:=\min_{i\in[m]}\eta_i^2$)
\[
D_{t+1,i}\leq \sum_{i=1}^m D_{t+1,i}\leq\frac{1}{\underline{\eta}^2}\sum_{i=1}^m\eta_i^2D_{1,i}+\frac{M_1}{\underline{\eta}^2}:=M_2,\quad\forall i\in[m].
\]
Since $D_{k+1,i}^2=D_{1,i}^2+\frac{1}{\eta_i^2}\sum_{j=1}^k\|\gcal_{j,i}\|^2$, summing over $i=1,\dots,m$ gives $\sum_{k=1}^t\|\gcal_k\|^2=\sum_{i=1}^m\eta_i^2(D_{t+1,i}^2-D_{1,i}^2)$, which together with the boundedness of $D_{t+1,i}$ implies that
\[
\sum_{k=1}^t\|\gcal_k\|^2 \leq \sum_{i=1}^m \eta_i^2(M_2+D_{1,i})(D_{t+1,i}-D_{1,i}) \leq (M_2+\max_{i}D_{1,i})M_1,
\]
hence $\frac{1}{t}\sum_{k=1}^t\|\gcal_k\|^2=O(1/t)$. The block-diagonal scheme extends to the convex settings with an analogous analysis, and the bounded-step derivation in Appendix~\ref{app:bd-discussion} carries over verbatim with $\beta_i$ in place of $L$.

\section{Extension to H\"older Smoothness for Convex Functions}
\label{app:holder-stoch}
For convex objectives, \citet{orabona2023normalized} showed that AdaGrad can adapt to H\"older smoothness, which interpolates between smoothness and Lipschitzness. This universal adaptation was later extended in \citet{rodomanov2024universality}.
Inspired by their work, we extend the convex analysis for both Algorithm~\ref{Alg:adaprox} and Algorithm~\ref{Alg:acc-adaprox} under the H\"older-smooth condition.

A function $f:\mathbb R^d\to\mathbb R$ is $(L_\nu,\nu)$-H\"older smooth with $L_\nu>0$ and $\nu\in[0,1]$ if
\[
    \|\nabla f(x)-\nabla f(y)\|\leq L_\nu\|x-y\|^\nu,\qquad x,y\in\mathbb R^d.
\]
The endpoint $\nu=1$ gives standard $L$-smoothness; $\nu=0$ includes differentiable $G$-Lipschitz functions with $L_0=2G$.
Note that the H\"older smoothness condition also implies~\citep[Theorem~1]{orabona2023normalized}:
\begin{equation}
    \label{eq:holder-model-cvx}
    f(y)\leq f(x)+\langle\nabla f(x),y-x\rangle+\frac{L_\nu}{1+\nu}\|y-x\|^{1+\nu},\qquad x,y\in\mathbb R^d.
\end{equation}

\subsection{Convergence of Algorithm~\ref{Alg:adaprox} under H\"older smoothness}
\label{app:cvx-holder-stoch}

\begin{theorem}\label{thm:cvx-holder-stoch}
Suppose $f$ is convex and $(L_\nu,\nu)$-H\"older smooth for some $\nu\in[0,1]$, and Assumptions~\ref{ass:seq-bd}(b) and~\ref{ass:bd-stoch} hold.
Let $\{x_k\}_{k\ge 1}$ be generated by Algorithm~\ref{Alg:adaprox} with $b_k\equiv1$ and any $\gamma,\eta>0$.
Then, for every $t\geq1$, the averaged iterate $\bar x_t=t^{-1}\sum_{k=1}^t x_{k+1}$ satisfies
\[
 \ebb[F(\bar x_t)-F^*]
 \lesssim\Big(t^{-\frac{1+\nu}{2}}+\frac{\sigma}{\sqrt t}\Big)\big[1+\log\big(1+t^{\frac{1-\nu}{2}}+\sigma\sqrt t\big)\big]^{\frac{1+\nu}{2}}.
\]
\end{theorem}

\begin{proof}
For the cases $\nu=0$ and $\nu=1$, the result follows directly from Theorem~\ref{thm:stoch-convex-nonsmooth} and Theorem~\ref{thm:stoch-convex-smooth}, respectively. Hereafter, we assume $0<\nu<1$.

Recall that $\Delta_k:=F(x_k)-F^*$ and $\delta_k:=\nabla f(x_k)-g_k$, and let $d_k:=\|x_{k+1}-x_k\|$.
Since $S_{k+1}^2=S_k^2(1+d_k^2/\eta^2)$ and $d_k\leq2D$ by Assumption~\ref{ass:seq-bd}(b), applying \eqref{eq:log-bound} in Lemma~\ref{lem:summation} with $a_k=S_k^2d_k^2/\eta^2$, $A_k=S_{k+1}^2$, and $C=4D^2/\eta^2$ then gives
\[
\sum_{k=1}^t d_k^2=\eta^2\sum_{k=1}^t \frac{a_k}{A_{k-1}}\leq 2(\eta^2+4D^2)\log\frac{S_{t+1}}{\gamma}.
\] 
Applying H\"older's inequality to $\{d_k^{1+\nu}\}$ and $\{1\}$ with conjugate exponents $p=2/(1+\nu)$ and $q=2/(1-\nu)$, respectively, we obtain from the above inequality that
\begin{equation}\label{eq:holder-dk-bd1}
\sum_{k=1}^t d_k^{1+\nu}\leq t^{\frac{1-\nu}{2}} \left(\sum_{k=1}^t d_k^2\right)^{\frac{1+\nu}{2}}\lesssim t^{\frac{1-\nu}{2}}\left(\log\frac{S_{t+1}}{\gamma}\right)^{\frac{1+\nu}{2}}.
\end{equation}
To proceed, rearranging~\eqref{eq:stoch-h-descent} gives
\[
\langle g_k,x_{k+1}-x_k\rangle+h(x_{k+1})-h(x_k)\leq-\frac{S_k}{\eta}d_k^2.
\]
Therefore, applying \eqref{eq:holder-model-cvx} and using $\nabla f(x_k)=g_k+\delta_k$, we get
\begin{align}
\Delta_{k+1}-\Delta_k&\leq\langle g_k,x_{k+1}-x_k\rangle+h(x_{k+1})-h(x_k)+\langle\delta_k,x_{k+1}-x_k\rangle+\frac{L_\nu}{1+\nu}d_k^{1+\nu}\notag\\
&\leq-\frac{S_k}{\eta}d_k^2+\frac{L_\nu}{1+\nu}d_k^{1+\nu}+\langle\delta_k,x_{k+1}-x_k\rangle.\label{eq:holder-F-gap1}
\end{align}
Summing~\eqref{eq:holder-F-gap1} over $k=1,\ldots,t$ and rearranging then yields
\[
2\eta(S_{t+1}-\gamma)\leq \sum_{k=1}^t \frac{S_kd_k^2}{\eta} \leq\Delta_1+\frac{L_\nu}{1+\nu}\sum_{k=1}^t d_k^{1+\nu}+\sum_{k=1}^t\langle\delta_k,x_{k+1}-x_k\rangle,
\]
where the first inequality uses $\sum_{k=1}^t S_kd_k^2/\eta=\eta\sum_{k=1}^t \|\widetilde{\gcal}_k\|^2/S_k\geq2\eta(S_{t+1}-\gamma)$ from~\eqref{eq:sqrt-bound}.

We next bound $m_t:=\ebb[S_{t+1}/\gamma]\geq1$. Taking expectations on both sides of the above inequality and applying~\eqref{eq:holder-dk-bd1} and Lemma~\ref{lem:stoc-ip-sum} with $b_k=1$ and $2D$ in place of $D$ gives
\begin{align*}
m_t\lesssim&1+t^{\frac{1-\nu}{2}}\ebb\big[(\log(S_{t+1}/\gamma))^{\frac{1+\nu}{2}}\big]+\sigma\sqrt{t\log m_t}\notag\\
\lesssim& 1+t^{\frac{1-\nu}{2}}(\log m_t)^{\frac{1+\nu}{2}}+\sigma\sqrt{t\log m_t},
\end{align*}
where the last inequality uses the concavity of $u\mapsto(\log u)^a$ for $a\in[0,1]$ on $[1,\infty)$.
For simplicity, set
\[
H_t:=1+t^{\frac{1-\nu}{2}}+\sigma\sqrt t.
\] 
Then it follows from $\nu\in[0,1]$ that
\begin{equation}\label{eq:holder-mt-bd}
m_t\leq C H_t(1+\log m_t)^{\frac{1+\nu}{2}}\leq CH_t(1+\log m_t)
\end{equation}
for some constant $C>0$. Applying $\log x\leq x-1$ with $x=m_t/(2CH_t)$ to the above inequality gives
\[
m_t\leq 2CH_t\log(2CH_t),
\]
which then implies $\log m_t\lesssim1+\log H_t$ and hence we obtain from~\eqref{eq:holder-mt-bd} that
\begin{equation}\label{eq:holder-mt-bd2}
m_t\lesssim H_t(1+\log H_t)^{\frac{1+\nu}{2}}.    
\end{equation}

For the objective gap, convexity gives
$f(x_k)-f(x^*)\leq\langle\nabla f(x_k),x_k-x^*\rangle$, which together with \eqref{eq:holder-model-cvx} at $(x, y)=(x_k,x_{k+1})$ yields
\[
f(x_{k+1})-f(x^*)\leq
\langle\nabla f(x_k),x_{k+1}-x^*\rangle
+\frac{L_\nu}{1+\nu}d_k^{1+\nu}.
\]
Combining the above inequality with~\eqref{eq:convex-trick}, $\nabla f(x_k)=g_k+\delta_k$, and $\langle\delta_k,x_{k+1}-x^*\rangle=\langle\delta_k,x_k-x^*\rangle+\langle\delta_k,x_{k+1}-x_k\rangle$ leads to
\begin{align*}
\Delta_{k+1}\leq&\frac{S_k}{2\eta}\bigl(\|x_k-x^*\|^2-\|x_{k+1}-x^*\|^2-d_k^2\bigr)+\frac{L_\nu}{1+\nu}d_k^{1+\nu}\\
&+\langle\delta_k,x_k-x^*\rangle+\langle\delta_k,x_{k+1}-x_k\rangle.
\end{align*}
Summing the above inequality over $k=1,\ldots,t$ and taking expectations, we obtain from \eqref{eq:bound-x-diff}, $\ebb[\langle\delta_k,x_k-x^*\rangle]=0$,~\eqref{eq:holder-dk-bd1},~\eqref{eq:holder-mt-bd2}, and Lemma~\ref{lem:stoc-ip-sum} with $b_k=1$ and $2D$ in place of $D$ that
\begin{align*}
\ebb\Big[\sum_{k=1}^t\Delta_{k+1}\Big]\leq&\frac{D^2\ebb[S_{t+1}]}{2\eta}+\frac{L_\nu}{1+\nu}\sum_{k=1}^t \ebb[d_k^{1+\nu}]+\sum_{k=1}^{t}\ebb[\langle\delta_k,x_{k+1}-x_k\rangle]\\
\lesssim& m_t+t^{\frac{1-\nu}{2}}(\log m_t)^{\frac{1+\nu}{2}}+\sigma\sqrt{t\log m_t}\\
\lesssim& H_t(1+\log H_t)^{\frac{1+\nu}{2}}.
\end{align*}
Finally, the result follows from the convexity of $F$ and $H_t/t\lesssim t^{-(1+\nu)/2}+\sigma/\sqrt t$, which completes the proof.
\end{proof}

\subsection{Convergence of Algorithm~\ref{Alg:acc-adaprox} under H\"older smoothness}
\label{app:acc-holder-stoch}

\begin{theorem}\label{thm:acc-holder-stoch}
Suppose $f$ is convex and $(L_\nu,\nu)$-H\"older smooth for some $\nu\in[0,1]$, and Assumptions \ref{ass:bd-stoch} and \ref{ass:z-bd} hold. Let $\{y_k\}_{k\ge 1}$ be generated by Algorithm \ref{Alg:acc-adaprox} with $b_k\equiv 1$, $\gamma>0$, and $\eta> \sqrt{2}D/2$. Then for any $t \ge 1$, we have
\[
 \ebb[F(y_{t+1})-F^*]\lesssim\Big(t^{-\frac{1+3\nu}{2}}+\frac{\sigma}{\sqrt t}\Big)\big[1+\log\big(1+t^{\frac{3(1-\nu)}{2}}+\sigma t^{3/2}\big)\big]^{\frac{1+\nu}{2}}.
\]
\end{theorem}

\begin{proof}
For the cases $\nu=0$ and $\nu=1$, the result follows directly from Theorem~\ref{thm:cvx-lip-acc} and Theorem~\ref{thm:smooth-acc}, respectively. Hereafter, we assume $0<\nu<1$.

Recall that $\Delta_k:=F(y_k)-F^*$ and $\delta_k:=\nabla f(x_k)-g_k$. For ease of notation, let $d_k:=\|z_{k+1}-z_k\|$ and $W_t:=\sum_{k=1}^t\alpha_k^2$.
Since $S_{k+1}^2=S_k^2(1+d_k^2/\eta^2)$, with $d_k\leq2D$ by Assumption~\ref{ass:z-bd}, we have from~\eqref{eq:log-bound} in Lemma~\ref{lem:summation} with $a_k=S_k^2d_k^2/\eta^2$, $A_k=S_{k+1}^2$, and $C=4D^2/\eta^2$ that
\[
\sum_{k=1}^t d_k^2=\eta^2\sum_{k=1}^t \frac{a_k}{A_{k-1}}\leq 2(\eta^2+4D^2)\log\frac{S_{t+1}}{\gamma}.
\] 
Applying H\"older's inequality to $\{\alpha_k^{1-\nu}\}$ and $\{d_k^{1+\nu}\}$ with conjugate exponents $p=2/(1-\nu)$ and $q=2/(1+\nu)$ respectively then gives
\begin{equation}\label{eq:holder-alpha-dk-bd}
\sum_{k=1}^t\alpha_k^{1-\nu}d_k^{1+\nu}\leq W_t^{\frac{1-\nu}{2}}\left(\sum_{k=1}^t d_k^2\right)^{\frac{1+\nu}{2}}\lesssim W_t^{\frac{1-\nu}{2}}\left(\log\frac{S_{t+1}}{\gamma}\right)^{\frac{1+\nu}{2}}.
\end{equation}
Applying~\eqref{eq:holder-model-cvx} at $(x_k,y_{k+1})$ and using $y_{k+1}-x_k=\theta_k(z_{k+1}-z_k)$ and~\eqref{eq:f-xk} with $x=x^*$ leads to
\begin{align*}
f(y_{k+1})&\leq f(x_k)+\theta_k\langle\nabla f(x_k),z_{k+1}-z_k\rangle+\frac{L_\nu}{1+\nu}\theta_k^{1+\nu}d_k^{1+\nu}\\
&\leq(1-\theta_k)f(y_k)+\theta_k f(x^*)+\theta_k\langle\nabla f(x_k),z_{k+1}-x^*\rangle+\frac{L_\nu}{1+\nu}\theta_k^{1+\nu}d_k^{1+\nu}.
\end{align*}
Substituting $\nabla f(x_k)=g_k+\delta_k$ into the above inequality, applying~\eqref{eq:ip-gk-zk} with $x=x^*$, and multiplying both sides by $\alpha_k^2$, we obtain
\begin{align*}
\alpha_k^2\Delta_{k+1}\leq&\alpha_{k-1}^2\Delta_k+\frac{S_k}{2\eta}\bigl(\|z_k-x^*\|^2-\|z_{k+1}-x^*\|^2-d_k^2\bigr)+\frac{L_\nu}{1+\nu}\alpha_k^{1-\nu}d_k^{1+\nu}\\
&+\alpha_k\langle\delta_k,z_{k+1}-z_k\rangle+\alpha_k\langle\delta_k,z_k-x^*\rangle,
\end{align*}
where we used $\theta_k=1/\alpha_k$ and $\alpha_k^2-\alpha_k=\alpha_{k-1}^2$ to simplify the result.
Summing this inequality over $k=1,\ldots,t$, we further obtain from~\eqref{eq:bound-z-diff} and $\sum_{k=1}^tS_kd_k^2/(2\eta)=\eta\sum_{k=1}^t\|\widetilde{\gcal}_{z,k}\|^2/(2S_k)\geq\eta(S_{t+1}-\gamma)$ due to~\eqref{eq:sqrt-bound} that
\begin{align}
\alpha_t^2\Delta_{t+1}+\left(\eta-\frac{D^2}{2\eta}\right)S_{t+1}\leq&\eta\gamma+\frac{L_\nu}{1+\nu}\sum_{k=1}^t\alpha_k^{1-\nu}d_k^{1+\nu}\notag\\
&+\sum_{k=1}^t\alpha_k\langle\delta_k,z_{k+1}-z_k\rangle+\sum_{k=1}^t\alpha_k\langle\delta_k,z_k-x^*\rangle.\label{eq:holder-F-gap-acc1}
\end{align}
Let $m_t:=\ebb[S_{t+1}/\gamma]\geq1$.
Note that the coefficient of $S_{t+1}$ in~\eqref{eq:holder-F-gap-acc1} is positive due to $\eta>D/\sqrt2$ and that the last sum there has zero expectation.
Therefore, taking expectations on both sides of~\eqref{eq:holder-F-gap-acc1} and applying~\eqref{eq:holder-alpha-dk-bd} and Lemma~\ref{lem:stoc-ip-sum-acc} gives
\[
m_t\lesssim 1+W_t^{\frac{1-\nu}{2}}(\log m_t)^{\frac{1+\nu}{2}}+\sigma\sqrt{W_t\log m_t}.
\]
For simplicity, set 
\[
H_t:=1+W_t^{\frac{1-\nu}{2}}+\sigma\sqrt{W_t}.
\] 
Then it follows from the same argument as in the proof of Theorem~\ref{thm:cvx-holder-stoch} that
\[
m_t\lesssim H_t(1+\log H_t)^{\frac{1+\nu}{2}}.
\]

To bound $\ebb[\Delta_{t+1}]$, we drop the nonnegative $S_{t+1}$ term in~\eqref{eq:holder-F-gap-acc1} and take expectations to obtain
\begin{equation}\label{eq:holder-F-gap-acc2}
\alpha_t^2\ebb[\Delta_{t+1}]\lesssim 1+W_t^{\frac{1-\nu}{2}}(\log m_t)^{\frac{1+\nu}{2}}+\sigma\sqrt{W_t\log m_t}\lesssim H_t(1+\log H_t)^{\frac{1+\nu}{2}}.
\end{equation}
Finally, \eqref{eq:alpha-property} gives $\alpha_t\geq(t+1)/2$ and $W_t\leq t(t+1)^2\leq4t^3$, which further implies
\[
\frac{H_t}{\alpha_t^2}\lesssim t^{-\frac{1+3\nu}{2}}+\frac{\sigma}{\sqrt t},\qquad \log H_t\lesssim 1+\log(1+t^{\frac{3(1-\nu)}{2}}+\sigma t^{3/2}).
\]
Then the result follows from dividing both sides of~\eqref{eq:holder-F-gap-acc2} by $\alpha_t^2$, which completes the proof.
\end{proof}

\section{Additional Experiments}
\label{app:experiments}
This section provides additional experimental results that complement Section~\ref{sec:experiments}.

\subsection{Regularized logistic regression}
\label{app:logistic}
We evaluate the performance of Algorithms~\ref{Alg:adaprox} and~\ref{Alg:acc-adaprox} on a convex composite problem: binary logistic regression with $\ell_1$ regularization and box constraints,
\begin{equation}\label{prob:logistic}
\min_{x\in[-R,R]^d} F(x):=\frac{1}{n} \sum_{i=1}^n \log \left(1+\exp \left(-b_i x^\top a_i\right)\right)+\lambda\|x\|_1.
\end{equation}
Here, $(a_i, b_i)\in \rbb^d\times \{-1,1\}$ denotes the $i$-th sample with feature vector $a_i\in\rbb^d$ and label $b_i\in\{-1,1\}$. This problem fits into the composite formulation~\eqref{main-prob} by defining
\[
f:=\frac{1}{n} \sum_{i=1}^n \log \left(1+\exp \left(-b_i \langle\cdot, a_i\rangle\right)\right),\quad h := \lambda\|\cdot\|_1 + \iota_{[-R, R]^d}(\cdot).
\]
For the experiments on \eqref{prob:logistic}, we set $R=50$. Since the objective function $F$ is convex, we evaluate optimality using the suboptimality gap $F(x)-F^*$, where the optimal objective value $F^*$ is approximated using the Python package CVXPY.

\textbf{Comparison between Algorithms \ref{Alg:adaprox} and \ref{Alg:acc-adaprox}.} We first set the $\ell_1$ regularization parameter to $\lambda=10^{-3}$ and compare the convergence behavior of Algorithms~\ref{Alg:adaprox} and~\ref{Alg:acc-adaprox} on the w6a dataset~\citep{chang2011libsvm} using a fixed batch size of 512. To assess robustness with respect to hyperparameter choices, we fix $\gamma=1$ and consider four values of $\eta \in \{0.1, 1, 10, 100\}$ for both algorithms.
Recall that Theorem~\ref{thm:stoch-convex-smooth} establishes convergence at the averaged iterate for Algorithm~\ref{Alg:adaprox}, whereas Theorem~\ref{thm:smooth-acc} guarantees last-iterate convergence for Algorithm~\ref{Alg:acc-adaprox}. This motivates evaluating performance at both the last and averaged iterates.

The results are presented in Figure~\ref{fig:log_eta}. We observe that Algorithm~\ref{Alg:adaprox} achieves stable convergence in terms of the averaged iterate across all choices of $\eta$, but its last-iterate performance degrades as $\eta$ increases. In contrast, Algorithm~\ref{Alg:acc-adaprox} exhibits robust convergence at both the last and averaged iterates, with similar convergence behavior across different values of $\eta$. These observations are consistent with Theorems~\ref{thm:stoch-convex-smooth} and~\ref{thm:smooth-acc}.

\begin{figure}[t]
    \centering
    \begin{subfigure}[b]{0.48\textwidth}
        \includegraphics[width=\textwidth]{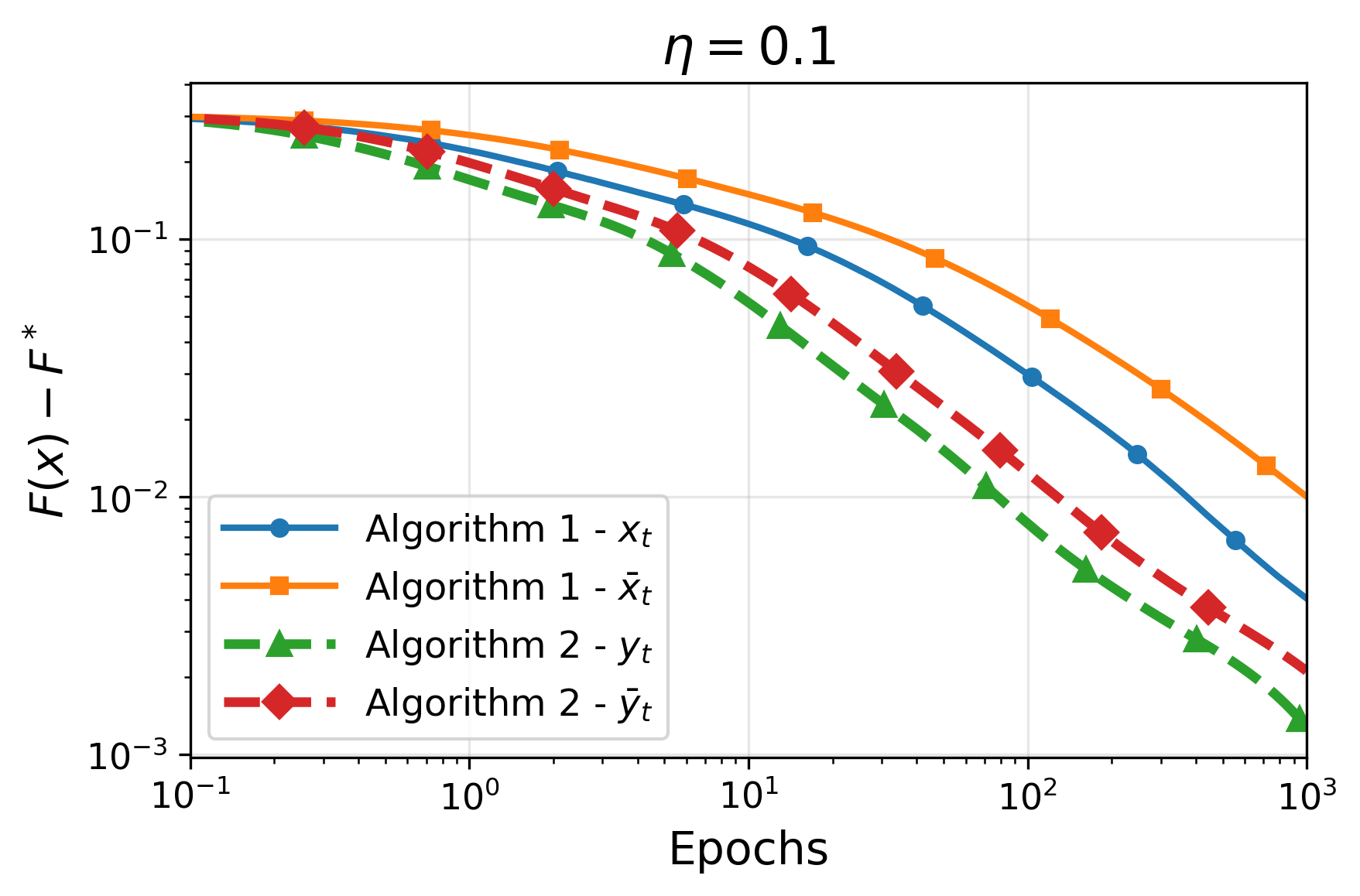}
    \end{subfigure}
    \begin{subfigure}[b]{0.48\textwidth}
        \includegraphics[width=\textwidth]{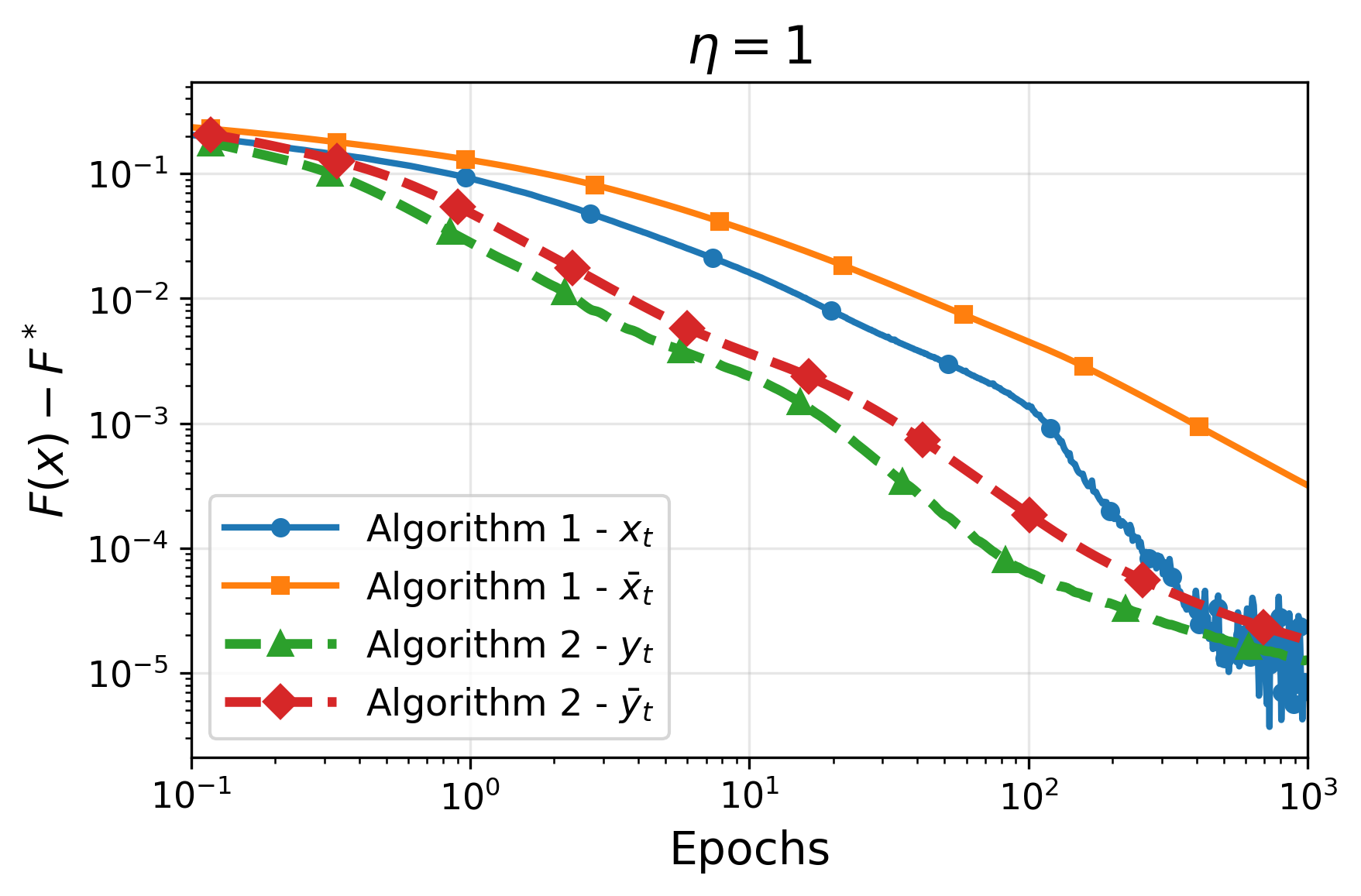}
    \end{subfigure}
    \begin{subfigure}[b]{0.48\textwidth}
        \includegraphics[width=\textwidth]{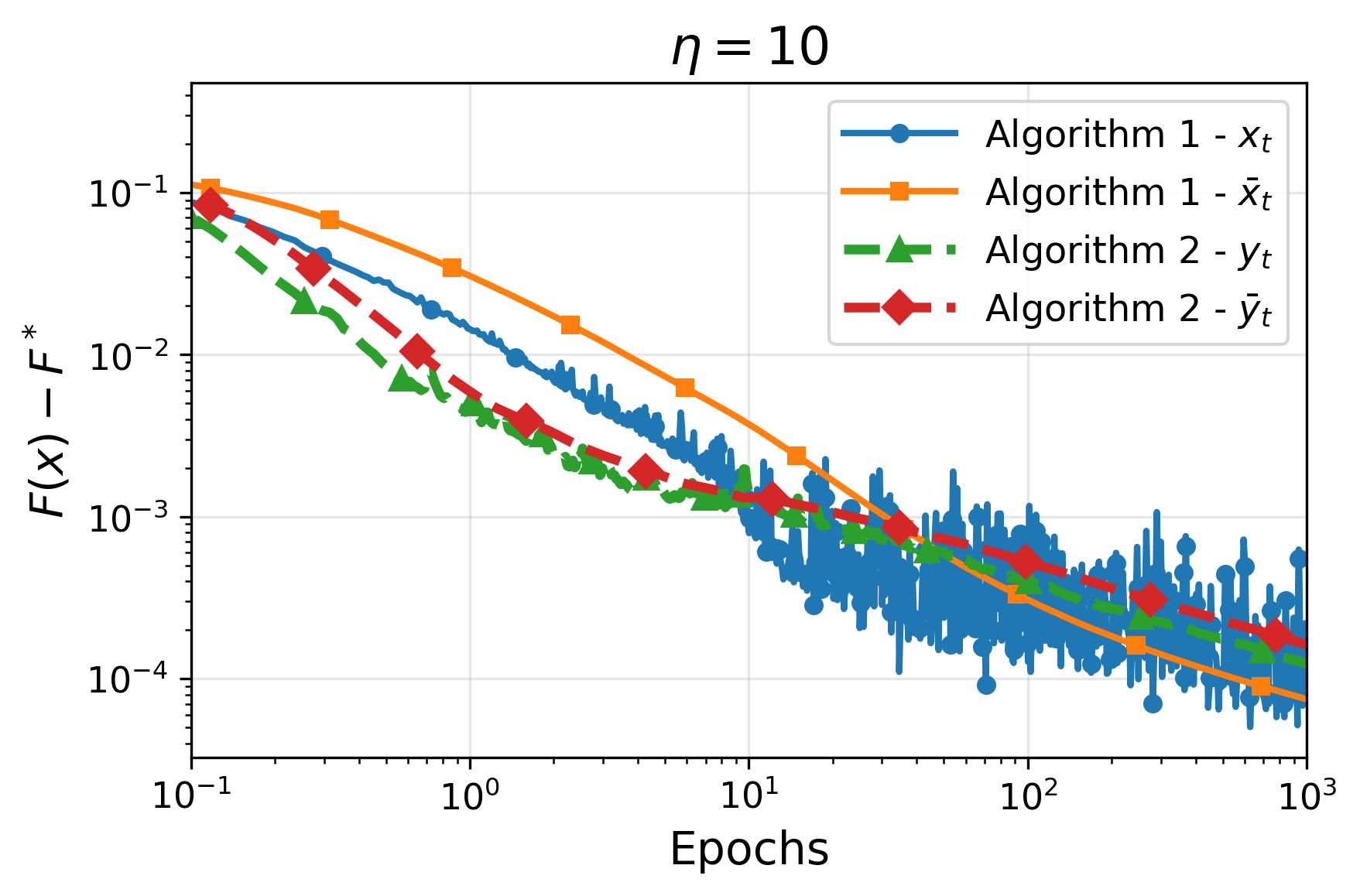}
    \end{subfigure}
    \begin{subfigure}[b]{0.48\textwidth}
        \includegraphics[width=\textwidth]{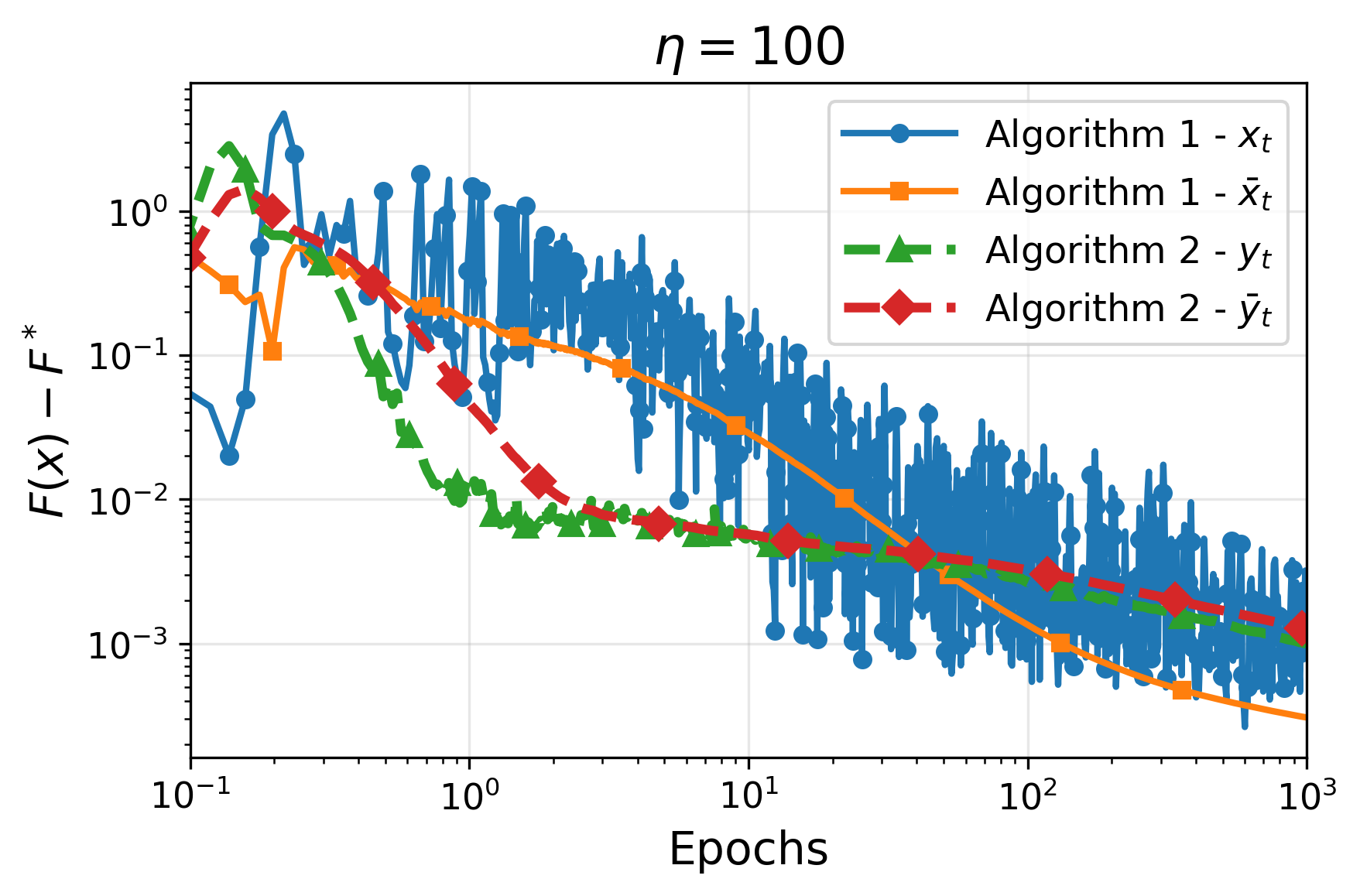}
    \end{subfigure}
    \caption{Comparison between Algorithms \ref{Alg:adaprox} and \ref{Alg:acc-adaprox} for the regularized logistic regression problem \eqref{prob:logistic} on the w6a dataset under different values of $\eta$. The averaged iterates are defined as $\bar{x}_t:= \frac{1}{t}\sum_{k=1}^t x_{k+1}$ for Algorithm \ref{Alg:adaprox} and $\bar{y}_t:= \frac{1}{\sum_{k=1}^t \alpha_k}\sum_{k=1}^t \alpha_k y_k$ for Algorithm \ref{Alg:acc-adaprox}.}
    \label{fig:log_eta}
\end{figure}

\textbf{Comparison in Non-Composite Settings.} We further compare Algorithms~\ref{Alg:adaprox} and~\ref{Alg:acc-adaprox} with existing adaptive accelerated methods UniXGrad~\citep{kavis2019unixgrad} and AcceleGrad~\citep{levy2018online}. Since the latter two methods are designed for non-composite problems, we focus on the special case $\lambda=0$. 
To conduct the experiments, we use the connect-4 dataset~\citep{chang2011libsvm} and split it into training ($80\%$) and test ($20\%$) sets. For all methods, we set the batch size to $b=512$ and run the algorithms for 1000 epochs. We fix $\eta=10$ for Algorithms \ref{Alg:adaprox} and \ref{Alg:acc-adaprox}, and use the same initial step sizes for UniXGrad and AcceleGrad to ensure a fair comparison. All other settings are kept consistent with the previous experiment.

Figure~\ref{fig:log_comparison} reports the convergence at both the last and averaged iterates, as well as test accuracies for the four methods. We observe that Algorithm~\ref{Alg:acc-adaprox} and AcceleGrad achieve comparable and strong performance in both convergence and test accuracy. In contrast, Algorithm~\ref{Alg:adaprox} converges more slowly and yields lower test accuracy in the early stages. Moreover, UniXGrad exhibits fast convergence in terms of averaged iterates and improves test accuracy quickly, but its last-iterate behavior is highly unstable.

\begin{figure}[t]
    \centering
    \begin{subfigure}[b]{0.33\textwidth}
        \includegraphics[width=\textwidth]{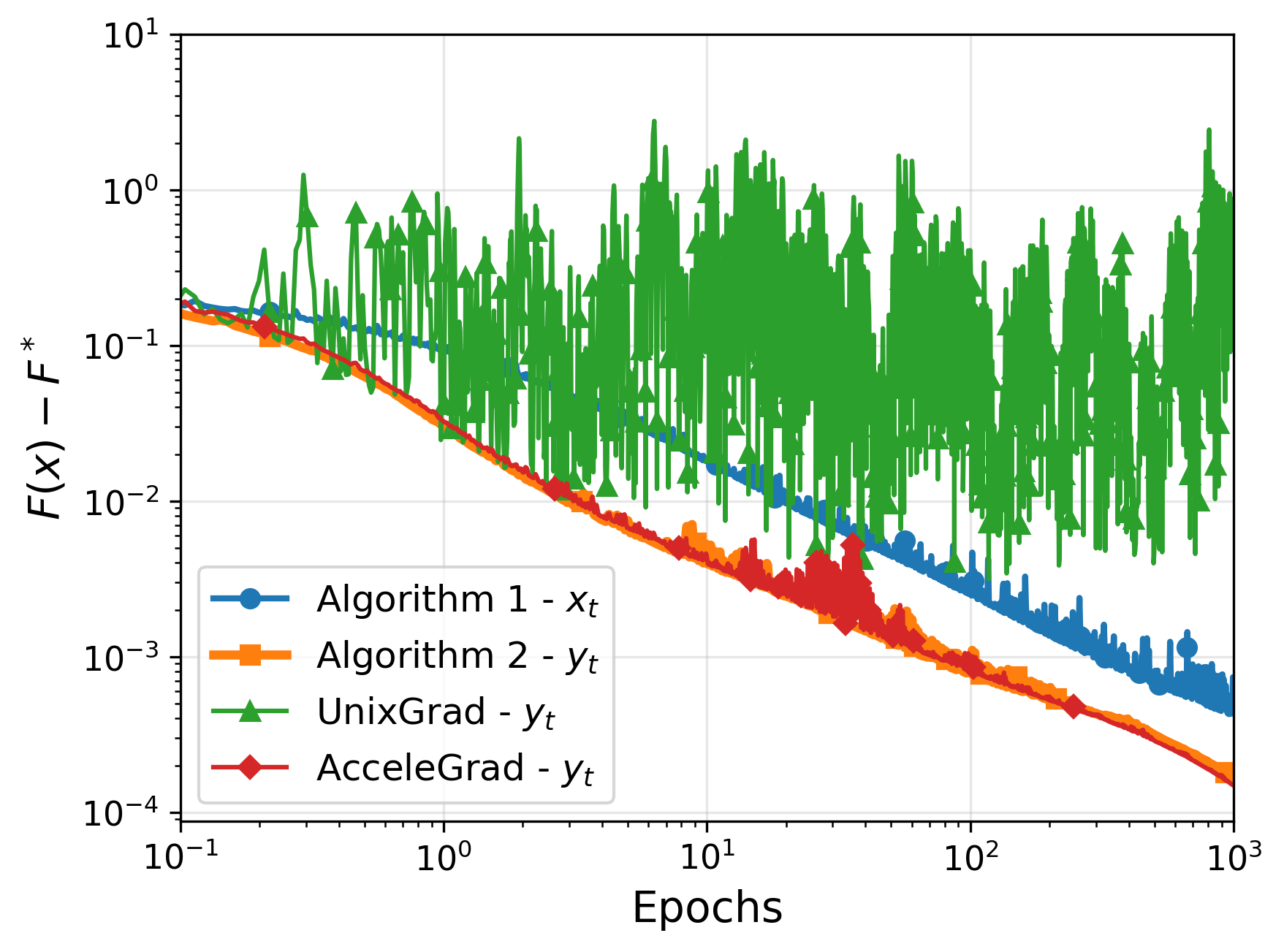}
        \caption{Last iterate}
    \end{subfigure}
    \begin{subfigure}[b]{0.325\textwidth}
        \includegraphics[width=\textwidth]{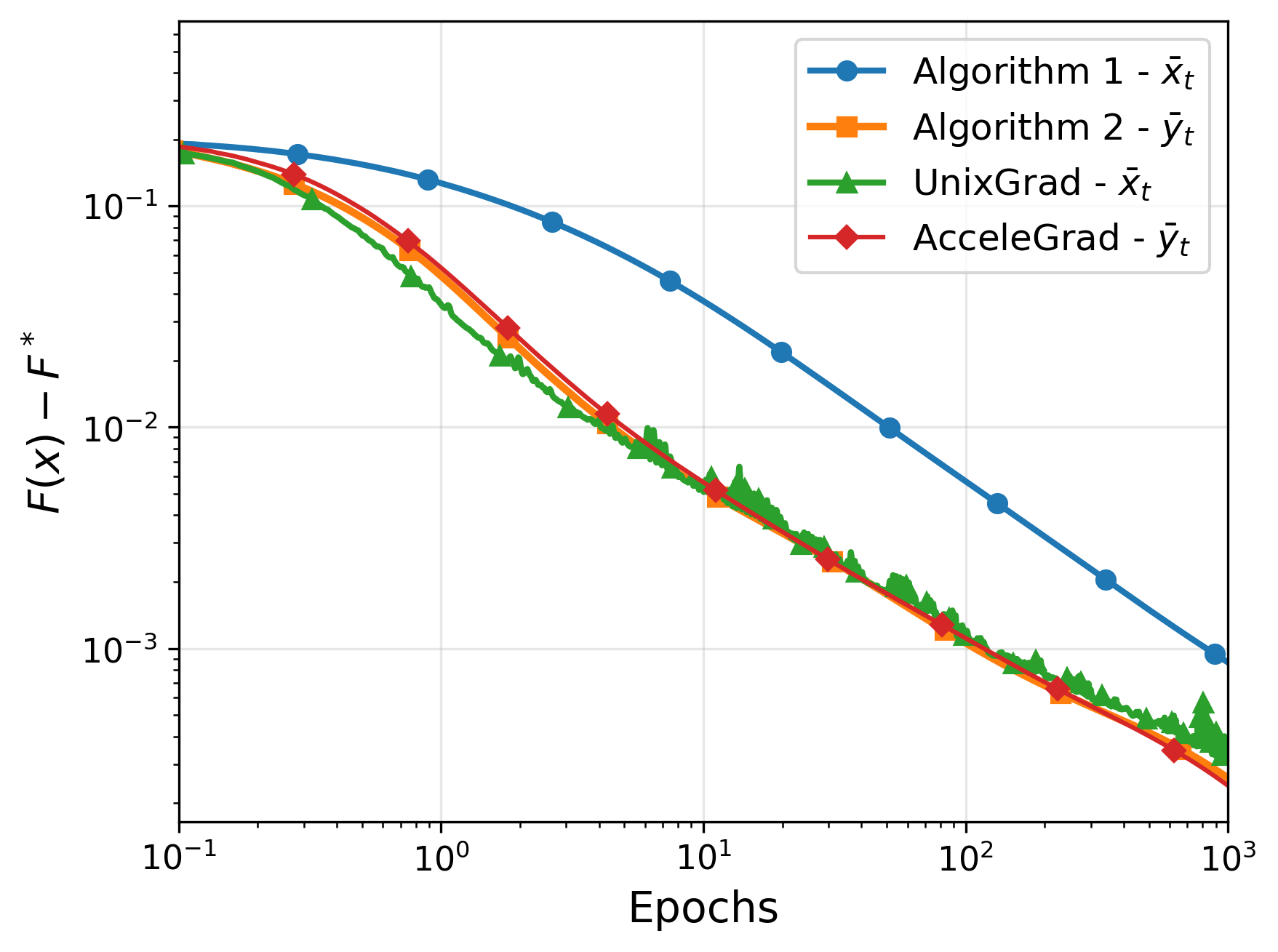}
        \caption{Averaged iterate}
    \end{subfigure}
    \begin{subfigure}[b]{0.33\textwidth}
        \includegraphics[width=\textwidth]{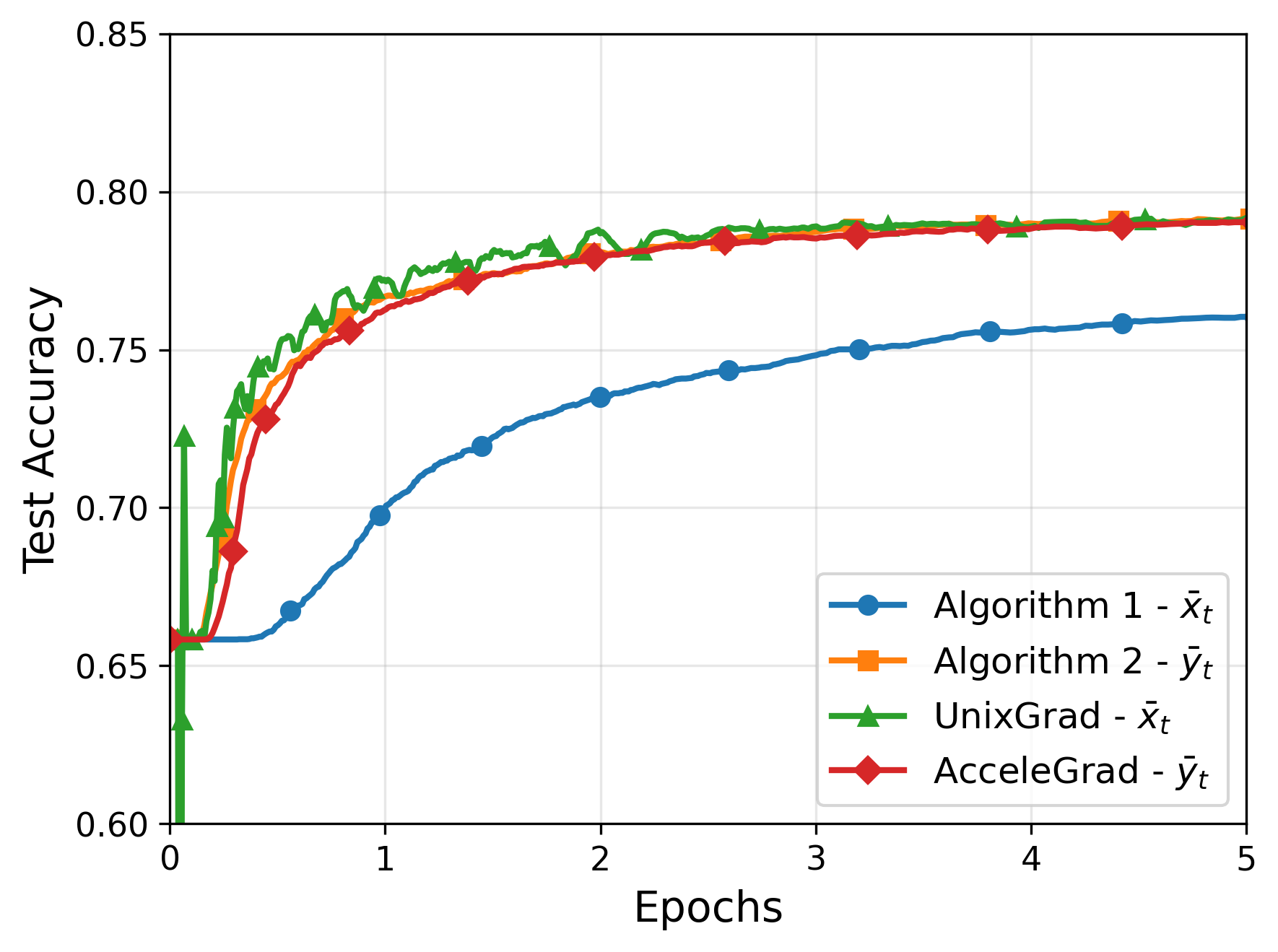}
        \caption{Test accuracy}
    \end{subfigure}
    \caption{Comparison among Algorithm \ref{Alg:adaprox}, Algorithm \ref{Alg:acc-adaprox}, UniXGrad, and AcceleGrad for solving the logistic regression problem \eqref{prob:logistic} with $\lambda=0$ on the connect-4 dataset.}
    \label{fig:log_comparison}
\end{figure}

\subsection{Full-gradient comparisons}
\label{app:det-benchmark}
\textbf{Setup.} We solve the nonconvex SVM problem \eqref{prob:svm} with $\mu=\lambda=10^{-3}$ on MNIST, a9a, and w6a, and logistic regression \eqref{prob:logistic} with $\lambda=10^{-3}$ on w6a, a9a, and connect-4 \citep{chang2011libsvm}, using $R=50$, $\gamma=1$, and $x_1=0$.
For MNIST, we use the binary labels of Section~\ref{sec:experiments}, and for connect-4, we assign the label $+1$ to wins and $-1$ to losses and draws.
All methods use full gradients for $5000$ iterations.
We report the relative gradient mapping norm $\|\gcal(x_t)\|/\|\gcal(x_1)\|$ for \eqref{prob:svm} and the relative objective gap $(F(x_t)-F^*)/(F(x_1)-F^*)$ for \eqref{prob:logistic}, at the last iterate ($y_t$ for Algorithm~\ref{Alg:acc-adaprox}).
We estimate $F^*$ from a solution satisfying $\|\gcal(\cdot)\|<10^{-9}$.

\textbf{Methods.} We compare Algorithms~\ref{Alg:adaprox} and~\ref{Alg:acc-adaprox} with AdaGrad-Norm and coordinate-wise AdaGrad \citep{duchi2011adaptive} (denoted AdaGrad (diag.) in the tables and figures), UniSgd and UniFastSgd \citep{rodomanov2024universality}, DoG \citep{ivgi2023dog}, DoWG \citep{khaled2023dowg}, and PG, all with proximal steps.
UniSgd and UniFastSgd use the same initial step-size grid as our methods, parameterized by $\hat D/\gamma$ with diameter estimate $\hat D$.
UniFastSgd is used only for logistic regression.
For DoG and DoWG, we set the initial distance estimate $r_\epsilon=10^{-6}$.

\textbf{Hyperparameters.} We use $\eta\in\{0.1,1,10,100,1000\}$ for our methods and both AdaGrad variants, the same grid for $\hat D$, and $\alpha\in\{0.01,0.1,1,10\}/L$ for PG.
Here, $L=\frac{4}{3\sqrt{3}}\lambda_{\max}(A^\top A/n)+\mu$ for \eqref{prob:svm} and $L=\frac14\lambda_{\max}(A^\top A/n)$ for \eqref{prob:logistic}, with $A$ the feature matrix.
For hyperparameter tuning, we select the value reaching relative accuracy $10^{-6}$ with the fewest gradient evaluations, or the smallest final relative error if none reaches it. Table~\ref{tab:det-values} lists the selected values.

\textbf{Results.} Tables~\ref{tab:det-svm}--\ref{tab:det-logistic} and Figures~\ref{fig:det-svm}--\ref{fig:det-logistic} show that Algorithm~\ref{Alg:adaprox} is competitive with the AdaGrad variants and UniSgd, and reaches $10^{-6}$ fastest on connect-4.
DoWG and PG perform best on several other datasets.
Acceleration helps at moderate accuracies in logistic regression but does not improve the highest-accuracy results, while Algorithm~\ref{Alg:acc-adaprox} stalls in the nonconvex runs (see also Remark~\ref{rmk:nc-alg2}).

\begin{table}[tbp]
\centering
\small
\caption{Selected hyperparameters for the full-gradient comparisons. PG step sizes are in units of $1/L$.}
\label{tab:det-values}
\begin{tabular}{@{}ll rrr rrr@{}}
\toprule
 & & \multicolumn{3}{c}{Nonconvex SVM \eqref{prob:svm}} & \multicolumn{3}{c}{Logistic regression \eqref{prob:logistic}} \\
\cmidrule(lr){3-5}\cmidrule(lr){6-8}
Method & Parameter & MNIST & a9a & w6a & w6a & a9a & connect-4 \\
\midrule
Algorithm~\ref{Alg:adaprox} & $\eta$ & 10 & 10 & 100 & 100 & 100 & 1000 \\
Algorithm~\ref{Alg:acc-adaprox} & $\eta$ & 10 & 10 & 100 & 100 & 10 & 10 \\
AdaGrad-Norm & $\eta$ & 10 & 10 & 100 & 100 & 100 & 100 \\
AdaGrad (diag.) & $\eta$ & 10 & 10 & 100 & 100 & 10 & 10 \\
UniSgd & $\hat D$ & 10 & 10 & 100 & 100 & 100 & 100 \\
UniFastSgd & $\hat D$ & -- & -- & -- & 100 & 10 & 10 \\
PG & $\alpha L$ & 1 & 10 & 10 & 10 & 1 & 1 \\
\bottomrule
\end{tabular}
\end{table}

\begin{table}[tbp]
\centering
\small
\caption{Full-gradient SVM \eqref{prob:svm}: gradient evaluations to reach $\|\gcal(x_t)\|/\|\gcal(x_1)\|\le\epsilon$ and time (seconds) to reach $\epsilon=10^{-6}$, with the hyperparameters from Table~\ref{tab:det-values}.
Bold marks the best entry in each column; ``--'' means the target is not reached within $5000$ iterations.}
\label{tab:det-svm}
\setlength{\tabcolsep}{4pt}
\begin{tabular}{@{}l rrrr rrrr rrrr@{}}
\toprule
 & \multicolumn{4}{c}{MNIST} & \multicolumn{4}{c}{a9a} & \multicolumn{4}{c}{w6a} \\
\cmidrule(lr){2-5}\cmidrule(lr){6-9}\cmidrule(lr){10-13}
Method & $10^{-2}$ & $10^{-4}$ & $10^{-6}$ & Time & $10^{-2}$ & $10^{-4}$ & $10^{-6}$ & Time & $10^{-2}$ & $10^{-4}$ & $10^{-6}$ & Time \\
\midrule
Algorithm~\ref{Alg:adaprox} & 115 & 495 & 934 & 2.91 & \textbf{13} & 406 & 850 & 0.28 & \textbf{13} & 32 & 62 & \textbf{0.02} \\
Algorithm~\ref{Alg:acc-adaprox} & 425 & -- & -- & -- & 49 & -- & -- & -- & 34 & 1574 & -- & -- \\
AdaGrad-Norm & 118 & 518 & 1005 & 3.13 & 14 & 407 & 860 & 0.28 & \textbf{13} & 32 & 62 & \textbf{0.02} \\
AdaGrad (diag.) & 115 & 497 & 940 & 2.93 & 17 & 401 & 851 & 0.28 & \textbf{13} & 32 & \textbf{60} & \textbf{0.02} \\
UniSgd & 115 & 488 & 920 & 2.87 & 14 & 406 & 849 & 0.28 & 14 & 34 & 63 & \textbf{0.02} \\
DoG & 366 & 1033 & 1827 & 5.69 & 703 & 2008 & 3860 & 1.27 & 264 & 1246 & 2540 & 0.92 \\
DoWG & \textbf{103} & \textbf{270} & \textbf{500} & \textbf{1.56} & 36 & \textbf{150} & \textbf{249} & \textbf{0.08} & 38 & 70 & 122 & 0.04 \\
PG & 355 & 1530 & 2889 & 8.99 & 22 & 161 & 300 & 0.10 & \textbf{13} & \textbf{31} & \textbf{60} & \textbf{0.02} \\
\bottomrule
\end{tabular}
\end{table}

\begin{figure}[tbp]
    \centering
    \includegraphics[width=\textwidth]{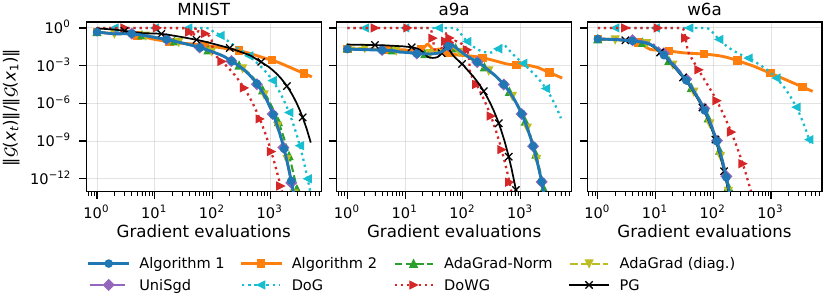}
    \caption{Nonconvex SVM \eqref{prob:svm} with full gradients: relative gradient mapping norm versus gradient evaluations, with the hyperparameters of Table~\ref{tab:det-values}.}
    \label{fig:det-svm}
\end{figure}

\begin{table}[tbp]
\centering
\small
\caption{Full-gradient logistic regression \eqref{prob:logistic}: gradient evaluations to reach $(F(x_t)-F^*)/(F(x_1)-F^*)\le\epsilon$ and time (seconds) to reach $\epsilon=10^{-6}$, with the hyperparameters from Table~\ref{tab:det-values}.
Bold marks the best entry in each column; ``--'' means the target is not reached within $5000$ iterations.}
\label{tab:det-logistic}
\setlength{\tabcolsep}{4pt}
\begin{tabular}{@{}l rrrr rrrr rrrr@{}}
\toprule
 & \multicolumn{4}{c}{w6a} & \multicolumn{4}{c}{a9a} & \multicolumn{4}{c}{connect-4} \\
\cmidrule(lr){2-5}\cmidrule(lr){6-9}\cmidrule(lr){10-13}
Method & $10^{-2}$ & $10^{-4}$ & $10^{-6}$ & Time & $10^{-2}$ & $10^{-4}$ & $10^{-6}$ & Time & $10^{-2}$ & $10^{-4}$ & $10^{-6}$ & Time \\
\midrule
Algorithm~\ref{Alg:adaprox} & 10 & 75 & 173 & 0.07 & 879 & 896 & 1041 & 0.39 & 611 & 830 & \textbf{1262} & \textbf{1.18} \\
Algorithm~\ref{Alg:acc-adaprox} & \textbf{9} & 79 & 812 & 0.31 & \textbf{44} & 397 & 3988 & 1.49 & \textbf{66} & \textbf{331} & 2726 & 2.40 \\
AdaGrad-Norm & 10 & 75 & 174 & 0.06 & 831 & 848 & 998 & 0.38 & 2045 & 2399 & 3675 & 3.42 \\
AdaGrad (diag.) & 10 & 75 & 173 & 0.06 & 216 & 1122 & 2257 & 0.82 & 528 & 3592 & -- & -- \\
UniSgd & 11 & 76 & 174 & 0.07 & 237 & \textbf{264} & 530 & 0.20 & 539 & 1260 & -- & -- \\
UniFastSgd & 18 & 144 & 1440 & 0.53 & 90 & 784 & 7828 & 2.82 & 118 & 502 & 4298 & 3.72 \\
DoG & 284 & 1139 & 2390 & 0.90 & 558 & 1850 & 3391 & 1.26 & 988 & 3523 & -- & -- \\
DoWG & 48 & 79 & 151 & 0.06 & 96 & 277 & \textbf{491} & \textbf{0.18} & 257 & 828 & 2311 & 2.07 \\
PG & 13 & \textbf{26} & \textbf{55} & \textbf{0.02} & 229 & 1197 & 2413 & 0.88 & 951 & -- & -- & -- \\
\bottomrule
\end{tabular}
\end{table}

\begin{figure}[tbp]
    \centering
    \includegraphics[width=\textwidth]{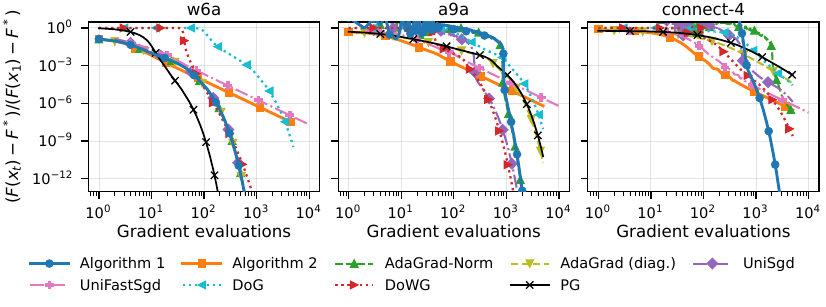}
    \caption{Logistic regression \eqref{prob:logistic} with full gradients: relative objective gap versus gradient evaluations, with the hyperparameters of Table~\ref{tab:det-values}.}
    \label{fig:det-logistic}
\end{figure}

\subsection{Constant and increasing batch sizes}
\label{app:minibatch}
\textbf{Setup.} We run Algorithm~\ref{Alg:adaprox} with $\eta=10$ and $\gamma=1$ for $200$ epochs on the nonconvex SVM problem~\eqref{prob:svm} with a9a and the logistic regression problem~\eqref{prob:logistic} with w6a, using the accuracy measures of Appendix~\ref{app:det-benchmark} for both problems.
One epoch corresponds to $n$ per-sample gradient evaluations.
We use the constant batch sizes $b\in\{0.1\%,1\%,10\%,100\%\}$ of $n$ and the increasing batch sizes $b_k=\min\{n,\lceil b_1\rho^{k-1}\rceil\}$ with $b_1=0.1\%$ of $n$ and $\rho\in\{1.01,1.05\}$.
We compare with UniSgd and DoWG at $b=0.1\%$ of $n$.
UniSgd's $\hat D$ is selected on the first seed using the grid and rule of Appendix~\ref{app:det-benchmark}, giving $0.1$ for SVM and $1$ for logistic regression.
All methods reshuffle the data each epoch.

\textbf{Results.} Figure~\ref{fig:batch} shows that constant batch sizes $b<n$ leave a residual error for the nonconvex SVM, while increasing batch sizes reduce it substantially.
For logistic regression, constant mini-batches continue to improve the objective, with $b=10\%$ giving the smallest final relative gap.

\begin{figure}[htbp]
    \centering
    \includegraphics[width=\textwidth]{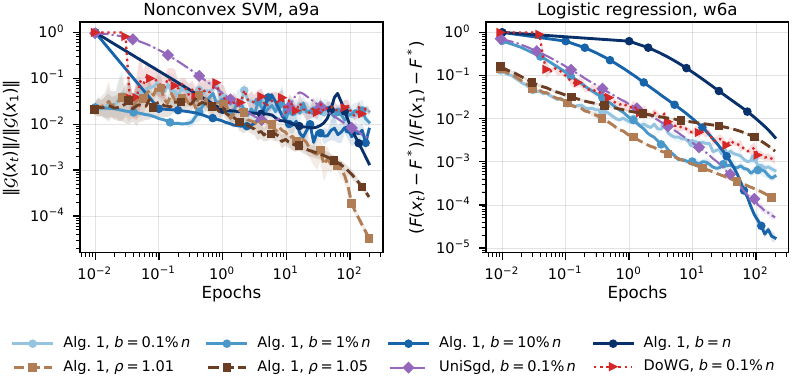}
    \caption{Constant and increasing batch sizes for SVM on a9a (left) and logistic regression on w6a (right), with UniSgd and DoWG as baselines.
Curves show geometric means over $5$ seeds, with shaded bands showing standard deviation.}
    \label{fig:batch}
\end{figure}

\subsection{Sensitivity to \texorpdfstring{$\eta$}{η} and \texorpdfstring{$\gamma$}{γ}}
\label{app:sensitivity}
We test Algorithm~\ref{Alg:adaprox}, Algorithm~\ref{Alg:acc-adaprox}, and AdaGrad-Norm on w6a for $2000$ full-gradient iterations, using $\eta\in\{10^{-2},10^{-1.5},\dots,10^2\}$ and $\gamma\in\{10^{-2},10^{-1},\dots,10^2\}$.
Figure~\ref{fig:eta-gamma} shows similar sensitivity for Algorithm~\ref{Alg:adaprox} and AdaGrad-Norm: small $\eta/\gamma$ slows progress, while $\gamma=1$ and $\eta\ge10$ perform well in the tested range.
Algorithm~\ref{Alg:acc-adaprox} makes limited progress on SVM but is less sensitive to $\eta$ on logistic regression.

\begin{figure}[htbp]
    \centering
    \includegraphics[width=\textwidth]{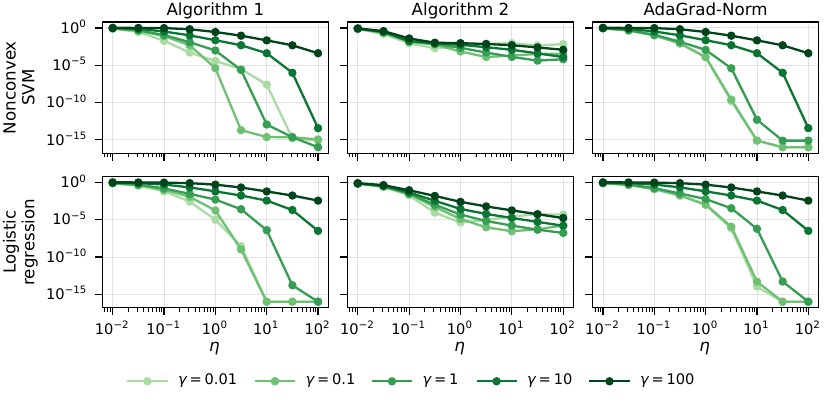}
    \caption{Sensitivity to $\eta$ and $\gamma$ on w6a after $2000$ full-gradient iterations: relative gradient mapping norm for SVM (top) and relative objective gap for logistic regression (bottom).}
    \label{fig:eta-gamma}
\end{figure}

\subsection{Algorithm~\ref{Alg:acc-adaprox} with \texorpdfstring{$\eta\le D/\sqrt{2}$}{η ≤ D/√2}}
\label{app:small-eta}
Recall that our analysis requires $\eta>D/\sqrt{2}$ in Theorems~\ref{thm:cvx-lip-acc} and~\ref{thm:smooth-acc}. To show that the algorithm can still work in practice without this condition, we run Algorithm~\ref{Alg:acc-adaprox} on w6a logistic regression with $\gamma=1$ and $\eta\in\{10^{-2},10^{-1},1,10,10^2\}$, using $5000$ full-gradient iterations or $50$ epochs with $b=1\%$ of $n$.
We use the box $[-3.5,3.5]^d$, which contains the reference solution, and take its diameter $D=7\sqrt d\approx121.2$ as a valid bound in Assumption~\ref{ass:z-bd}.
The resulting sufficient threshold is $D/\sqrt2\approx85.7$.
Figure~\ref{fig:alg2-eta} shows convergence below this threshold, including at $\eta=1$ and $10$.
Smaller values of $\eta$ may slow down the convergence, while moderate values perform better with mini-batches. This is consistent with Remark~\ref{rmk:eta-cond}.

\begin{figure}[htbp]
    \centering
    \includegraphics[width=\textwidth]{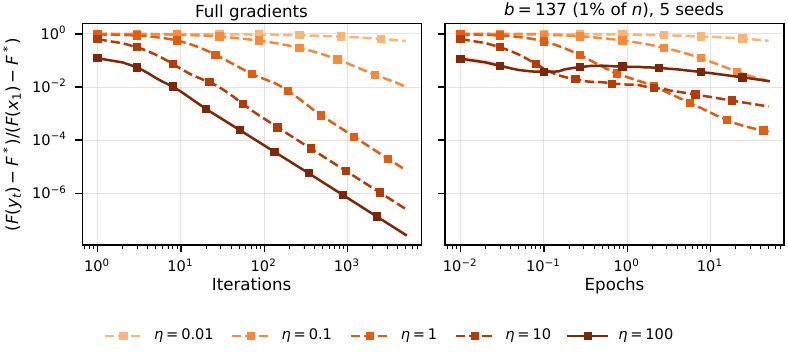}
    \caption{Algorithm~\ref{Alg:acc-adaprox} on w6a logistic regression with full gradients (left) and $b=1\%$ of $n$ (right).
    Dashed curves use $\eta\le D/\sqrt2$.}
    \label{fig:alg2-eta}
\end{figure}

\end{document}